\documentclass[11pt]{article}
\usepackage[dvips]{graphicx}
\usepackage{bm}
\usepackage{amsmath,amsthm,amssymb,citesort,url}
\usepackage{multirow,bigdelim}
\usepackage{amscd}

\newtheorem{thm}{Theorem}[section]
\newtheorem{theorem}[thm]{Theorem}
\newtheorem{proposition}[thm]{Proposition}
\newtheorem{conjecture}[thm]{Conjecture}
\newtheorem{lemma}[thm]{Lemma}

\newtheorem{claim}[thm]{Claim}
\newtheorem{corollary}[thm]{Corollary}

\DeclareMathOperator{\rank}{rank}

\newcommand{\de}{\mathop{\mathrm{def}}\nolimits}
\newcommand{\bp}{\mathbf{p}}
\newcommand{\bq}{\mathbf{q}}

\newcommand{\bc}{{\bm c}}
\newcommand{\bx}{{\bm x}}

\newcommand{\Bvector}[2]{\stackrel{#1}{\mathstrut #2}}
\numberwithin{equation}{section}

\begin{document}
\title{{\Large A Proof of the Molecular Conjecture}}
\author{\ \ \ \ Naoki Katoh \ \  and \ \  
Shin-ichi Tanigawa 
}
\date{{\small
Department of Architecture and Architectural Engineering, Kyoto University, \\
Kyoto Daigaku Katsura, Nishikyo-ku, Kyoto 615-8540 Japan, \\
\url{{naoki,is.tanigawa}@archi.kyoto-u.ac.jp}
}
}
\maketitle

%{
%%
%\setcounter{footnote}{0}
%\def\thefootnote{\arabic{footnote}}
%%
%\footnotetext[1]{Supported by the project {\em  New Horizons in Computing},
%Grant-in-Aid for Scientific Research on Priority Areas, NEXT Japan.}
%\footnotetext[2]{Supported by Grant-in-Aid for JSPS Research Fellowships for Young Scientists.}
%%Supported by JSPS Grant-in-Aid for Scientific Research on priority areas of New Horizons in Computing.}
%}
%
\begin{abstract}
A $d$-dimensional body-and-hinge framework is a structure consisting of  rigid bodies
 connected by hinges in $d$-dimensional space. 
The generic infinitesimal rigidity of a body-and-hinge framework has been characterized 
in terms of the underlying multigraph independently by Tay and Whiteley as follows:
A multigraph $G$ can be realized as an infinitesimally rigid body-and-hinge framework 
by mapping each vertex to a body and each edge to a hinge if and only if 
$\left({d+1 \choose 2}-1\right)G$ contains ${d+1\choose 2}$ edge-disjoint spanning trees,
where $\left({d+1 \choose 2}-1\right)G$ is the graph obtained from $G$ 
by replacing each edge by $\left({d+1\choose 2}-1\right)$ parallel edges.
In 1984 they jointly posed a question about whether their
combinatorial characterization
can be further applied to a nongeneric case.
Specifically, they conjectured that $G$ can be realized as an infinitesimally rigid body-and-hinge framework
if and only if $G$ can be realized as that with the additional ``hinge-coplanar'' property,
i.e., all the hinges incident to each body are contained in a common hyperplane.
This conjecture is called the Molecular Conjecture 
due to the equivalence between the infinitesimal rigidity of ``hinge-coplanar'' body-and-hinge frameworks 
and that of bar-and-joint frameworks derived from molecules in $3$-dimension.
In $2$-dimensional case this conjecture has been proved by Jackson and Jord{\'a}n in 2006.
In this paper we prove this long standing conjecture  affirmatively for general dimension.
%Also, as a corollary, we obtain a combinatorial characterization of the $3$-dimensional bar-and-joint rigidity matroid
%of the square of a graph.
\end{abstract}

%\clearpage

\section{Introduction\label{sec:tree:introduction}}

A $d$-dimensional {\em body-and-hinge framework} is the collection of $d$-dimensional {\em rigid bodies}
connected by {\em hinges}, where a hinge is a $(d-2)$-dimensional affine subspace, i.e.~pin-joints in $2$-space, line-hinges in $3$-space,
plane-hinges in $4$-space and etc.
The bodies are allowed to move continuously in $\mathbb{R}^d$ so that 
the relative motion of any two bodies connected by a hinge is a rotation around it (see Figure~\ref{fig:body-and-hinge})
and the framework is called {\em rigid} if every such  motion provides a framework isometric to the original one.
The infinitesimal rigidity of this physical model can be formulated in terms of a linear homogeneous system 
by using the fact that 
any continuous rotation of a point around a $(d-2)$-dimensional affine subspace or any transformation to a fixed direction 
can be described by a ${d+1 \choose 2}$-dimensional vector, so-called a {\em screw center}.
The formal definition of body-and-hinge frameworks will be given in the next section.

Let $G=(V,E)$ be a {\em multigraph} which may contain multiple edges.
We consider a body-and-hinge framework as a pair $(G,\bq)$ where $\bq$ is a mapping from $e\in E$ 
to a $(d-2)$-dimensional affine subspace $\bq(e)$ in $\mathbb{R}^d$.
Namely, $v\in V$ corresponds to a body and $uv\in E$ corresponds to a hinge $\bq(uv)$ which joins the two bodies 
associated with $u$ and $v$.
The framework $(G,\bq)$ is called a {\em body-and-hinge realization} of $G$ in $\mathbb{R}^d$.

\begin{figure}[t]
\centering
\begin{minipage}{0.26\textwidth}
\centering
\includegraphics[width=0.5\textwidth]{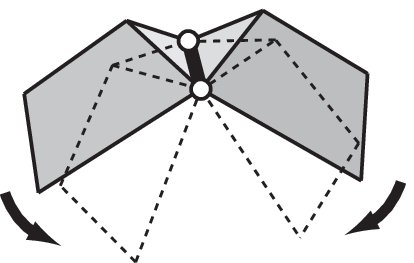}
\par
(a)
\end{minipage}
\begin{minipage}{0.26\textwidth}
\centering
\includegraphics[width=0.6\textwidth]{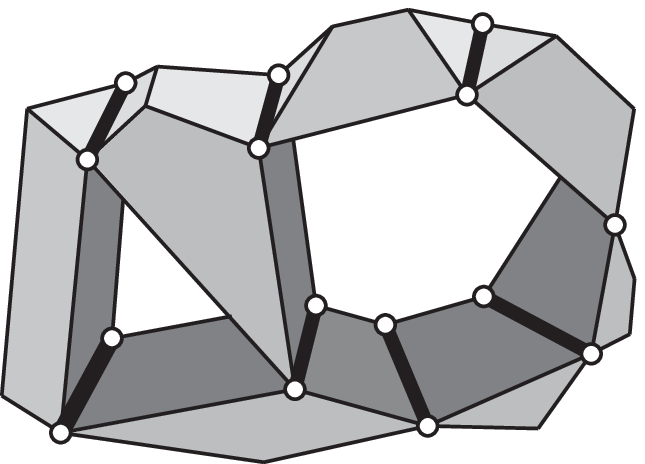}
\par
(b)
\end{minipage}
\begin{minipage}{0.26\textwidth}
\centering
\includegraphics[width=0.6\textwidth]{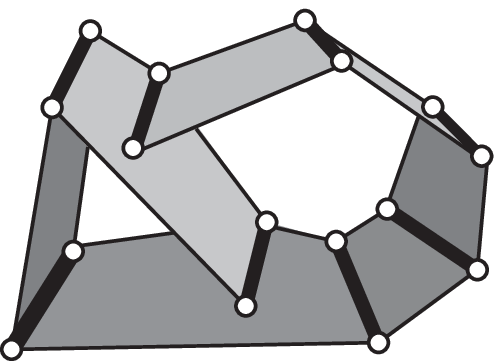}
\par
(c)
\end{minipage}
\caption{(a)Two bodies connected by a hinge in 3-dimension. (b)A body-and-hinge framework. (c)A panel-and-hinge framework.}
\label{fig:body-and-hinge}
\end{figure}

We assume that the dimension $d$ is a fixed integer with $d\geq 2$ and we shall use the notation $D$ to denote ${d+1 \choose 2}$.
For a multigraph $G=(V,E)$ and a positive integer $k$, the graph obtained by replacing each edge by $k$ parallel edges
is denoted by $kG$.
In this paper, for our special interest in $(D-1)G$, we shall use the simple notation $\widetilde{G}$ to denote $(D-1)G$  
and let $\widetilde{E}$ be the edge set of $\widetilde{G}$.
Tay~\cite{tay:89} and Whiteley~\cite{whiteley:88} independently proved that 
the generic infinitesimal rigidity of a body-and-hinge framework is determined by the underlying (multi)graph as follows.
\begin{proposition}$(\cite{tay:89,whiteley:88})$
\label{prop:tay}
A multigraph $G$ can be realized as an infinitesimally rigid body-and-hinge framework in $\mathbb{R}^d$ 
if and only if
$\widetilde{G}$ has $D$ edge-disjoint spanning trees.
\end{proposition}

%For example Figure~\ref{fig:graph}(a) shows the underlying graph $G$ of the body-and-hinge framework shown in Figure~\ref{fig:body-and-hinge}(b).
%Observe that, for each edge $e$ of $G$,  
%the  six spanning trees given in Figure~\ref{fig:graph}(b) contains at most five edge

\begin{figure}[t]
\centering
\begin{minipage}{0.2\textwidth}
\centering
\includegraphics[width=0.45\textwidth]{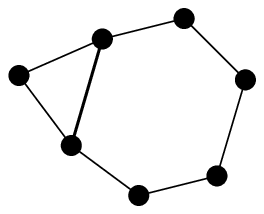}
\par
(a)
\end{minipage}
\begin{minipage}{0.7\textwidth}
\centering
\includegraphics[width=0.9\textwidth]{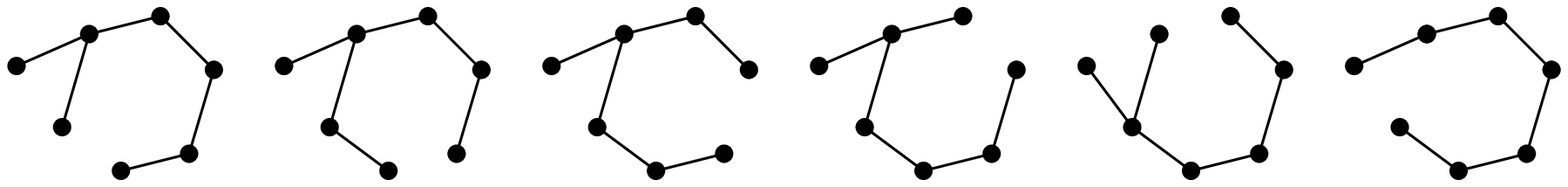}
\par
(b)
\end{minipage}
\caption{(a)The underlying graph $G$ of the body-and-hinge framework depicted in Figure~\ref{fig:body-and-hinge}(b).
(b)Six edge-disjoint spanning tress in $\widetilde{G}$.}
\label{fig:graph}
\end{figure}

For example, Figure~\ref{fig:graph}~(a) shows the underlying graph $G$ of the body-and-hinge framework illustrated in Figure~\ref{fig:body-and-hinge}(b).
Since $5G$ contains six edge-disjoint spanning trees as illustrated in Figure~\ref{fig:graph}(b),
Tay and Whiteley's Theorem (Proposition~\ref{prop:tay}) ensures that 
$G$ can be realized as an infinitesimally rigid body-and-hinge framework $(G,\bp)$ in $\mathbb{R}^3$.

A body-and-hinge framework $(G,\bp)$ is called {\em hinge-coplanar} if, 
for each $v\in V$,
all of the $(d-2)$-dimensional affine subspaces $\bp(e)$ for the edges $e$ incident to $v$ are contained in a common 
$(d-1)$-dimensional affine subspace (i.e.~a hyperplane).
In this case replacing each body by a rigid {\em panel} does not change the rigidity of the framework.
Thus, a hinge-coplanar body-and-hinge framework is said to be a {\em panel-and-hinge framework} (see Figure~\ref{fig:body-and-hinge}(c)).

Note that ``hinge-coplanarity'' implies a special hinge configuration of a body-and-hinge framework,
which may cause a  degree of freedom even if the underlying graph satisfies the tree packing condition of  Proposition~\ref{prop:tay}.  
In 1984, Tay and Whiteley~\cite{tay:whiteley:84} jointly conjectured that such an extra degree of freedom does not appear.
\begin{conjecture}$(\cite{tay:whiteley:84})$
\label{conjecture}
Let $G=(V,E)$ be a multigraph.
Then, $G$ can be realized as an infinitesimally rigid body-and-hinge framework in $\mathbb{R}^d$
if and only if $G$ can be realized as an infinitesimally rigid panel-and-hinge framework in $\mathbb{R}^d$.
\end{conjecture}
Conjecture~\ref{conjecture} is known as the {\em Molecular Conjecture} which has appeared in several different forms~\cite{whiteley:hand,Whitley:1997} 
and has been a long standing open problem in the rigidity theory.
For the special case of $d=2$, Whiteley~\cite{Whiteley:89} proved the conjecture affirmatively for the special class of multigraphs in 1989 and 
recently the conjecture has been completely proved by Jackson and Jord{\'a}n~\cite{Jackson:08}.
The idea of their proof is to replace each body of a panel-and-hinge framework by 
a rigid bar-and-joint framework (called a rigid component) and 
reduce the problem to that for bar-and-joint frameworks.
The definition of a bar-and-joint framework can be found in e.g.~\cite{graver:servatius:servatius:CombinatorialRigidity:1993,Whitley:1997}.
By using well-investigated properties of $2$-dimensional bar-and-joint frameworks,
they successfully proved the conjecture.
Also, Jackson and Jord{\' a}n~\cite{Jackson:07} showed the {\em sufficient} condition of the graph 
to have a panel-and-hinge realization in higher dimension;
$G$ has a panel-and-hinge realization in $\mathbb{R}^d$ if 
$(d-1)G$ has $d$ edge-disjoint spanning trees.
This sufficient condition has been proved by replacing each hinge of a body-and-hinge framework 
by rigid bars connecting the bodies.

In this paper we will settle the Molecular Conjecture affirmatively in general dimension.
Although the overall strategy of our proof is slightly close to that of \cite{Jackson:08} for 2-dimension,
our proof directly provides a construction of an infinitesimally rigid panel-and-hinge framework, 
which is a main (and huge) difference from \cite{Jackson:08}.

In $\mathbb{R}^3$  the rigidity of panel-and-hinge frameworks has a special connection
with the flexibility of molecules.
To identify flexible/rigid region in a protein is one of the
central issues in the field of molecular biology as this could provide insight into its function 
and a means to predict possible changes of structural flexibility by the environmental factors such as temperature and pH. 
One of  standard methods is to model the protein  as a body-and-hinge or a bar-and-joint framework
and it analyzes the protein's rigidity by using the theory of structural rigidity.

Consider a molecule consisting of atoms connected by covalent bonds.
It is known that a molecule can be modeled as a bar-and-joint framework of the {\em square} of a graph 
(see e.g.~\cite{Jackson:05a,Jackson:05,Jacobs:1999,Whiteley:2004,Whiteley:2005}).
The square of a graph $G=(V,E)$ is defined as $G^2=(V,E^2)$,
where $E^2=E\cup\{uv\in V\times V|\, u\neq v \text{ and } uw,wv\in E \text{ for some } w\in V\setminus\{u,v\}\}$.
This framework can be also seen as a body-and-hinge framework by regarding each atom (vertex) as a rigid body and each bond (edge) as a hinge
since in the square of a graph a vertex and its neighbor always form a complete graph.
Notice however that this body-and-hinge framework has a ``special'' hinge configuration, 
i.e., all the hinges (lines) incident to a body are intersecting each other at the center of the body. 
Such a hinge configuration is called {\em hinge-concurrent}, 
and so a molecule can be modeled as a hinge-concurrent body-and-hinge framework~\cite{tay:whiteley:84,whiteley:hand,Whiteley:2004,Whiteley:2005}.

The projective duality reveals the reason why Conjecture~\ref{conjecture} is called the ``Molecular Conjecture''.
Recall that taking projective dual in $\mathbb{R}^3$ transforms points to planes, lines to lines and planes to points 
preserving their incidences.
This means that the dual of a hinge-concurrent body-and-hinge framework is exactly a panel-and-hinge framework.
Crapo and Whiteley~\cite{crapo:1982} showed that the infinitesimal rigidity is invariant under the projective duality,
which implies that $G$ has an infinitesimally rigid panel-and-hinge realization if and only if 
it has an infinitesimally rigid hinge-concurrent body-and-hinge realization.
Therefore, the correctness of the Molecular Conjecture suggests that
the flexibility of proteins can be combinatorially investigated by using the well-developed tree-algorithm on the underlying graphs.
In fact, the so-called ``pebble game'' algorithm~\cite{lee:streinu:theran:2005} for packing spanning trees is implemented in several softwares, 
e.g., FIRST~\cite{flexweb,Jacobs:1999}, ROCK~\cite{lei:2004} and others~\cite{wells,amato}.  
From a mathematical point of view, however, the correctness proof is incomplete because they rely on the Molecular Conjecture
which has been a long standing open problem over twenty-five years.
The result of this paper provides the theoretical validity
of the algorithms behind such softwares as FIRST, ROCK, etc.

%Combining this previous result with our proof of the Molecular Conjecture,
%we obtain a combinatorial characterization of the $3$-dimensional rigidity of the square of a graph;
%$G^2$ can be realized as an infinitesimal rigid bar-and-joint framework in $\mathbb{R}^3$ if and only if 
%$5G$ contains six edge-disjoint spanning trees.
%Also Jackson and Jord{\'a}n~\cite{Jackson:05} rigorously proved that
%the correctness of the Molecular Conjecture provides a clear combinatorial characterization of the rigid components of a molecular structure.
%Therefore the truth of the Molecular Conjecture implies that
%not only the degree of freedom but also rigid components (rigid subregions) of a molecule can be computed combinatorially, deterministically
%and efficiently by using a well-investigated tree packing algorithm.

%In past years, despite  the absence of the rigorous proof of the Molecular Conjecture, 
%empirical data have been accumulated that support this conjecture~\cite{Chubynsky:2008,thorpe:2005,Whiteley:2005}.  
%In this paper we are able to settle the Molecular Conjecture affirmatively in $\mathbb{R}^3$ and in higher dimensions that provides the theoretical validity of the algorithms behind such software as FIRST, FRODA, etc. 
%
%

The paper consists of seven sections.
In Section~2, we shall provide a formal definition of the infinitesimal 
rigidity of body-and-hinge frameworks following the description given in \cite{Jackson:07}.
In Section~\ref{sec:spanning_trees}, we will provide several preliminary results concerning edge-disjoint spanning trees.
In Sections~\ref{sec:body-and-hinge rigid graph} and \ref{sec:inductive}, we will investigate the combinatorial properties of multigraphs $G$ 
such that
$\widetilde{G}$ contains $D$ edge-disjoint spanning trees.
Such graphs are called {\em body-and-hinge rigid graphs}.
In particular, edge-inclusionwise minimal body-and-hinge rigid graphs are called {\em minimally body-and-hinge rigid graphs}.
In Section~5, we will show that 
any minimally body-and-hinge rigid graph can be reduced to a smaller minimally body-and-hinge rigid graph
by the contraction of a rigid subgraph or a splitting off operation (defined in Section~5) at a vertex of degree two.
This implies that any minimally body-and-hinge rigid graph can be constructed from a smaller body-and-hinge rigid graph 
by the inversions of these two operations.
Finally, in Sections~6 and 7, we will provide a proof of the Molecular Conjecture by showing that
any minimally body-and-hinge rigid graph $G$ has a rigid panel-and-hinge realization.
The proof is done by induction on the graph size.
More precisely, following the construction of a graph given in Section~5,
we convert $G$ to a smaller minimally body-and-hinge rigid graph $G'$.
From the induction hypothesis there exists a rigid panel-and-hinge realization of $G'$.
We will show that we can extend this realization to that of $G$ with a slight modification so that the resulting framework becomes rigid.

\section{Body-and-hinge Frameworks}
In this section we shall provide a formal definition of body-and-hinge frameworks following the description given in \cite{Jackson:07,white:whiteley:87}.
Refer to \cite{Jackson:07,crapo:1982,White94,white:whiteley:87} for more detailed descriptions.
Throughout the paper, we simply refer to a $d$-dimensional affine space as a $d$-affine space for any nonnegative integer $d$.

\subsection{Infinitesimal motions of a rigid body}
A {\em body} is a set of points which affinely spans $\mathbb{R}^d$.
An {\em infinitesimal motion} of a body is an isometric linear transformation of the body,
i.e.,~the distance between any two points in the body is preserved after the transformation.
It is known that the set of infinitesimal motions of a body forms a $D$-dimensional vector space,
i.e., an infinitesimal motion is a linear combination of $d$ translations and ${d \choose d-2}$ rotations around $(d-2)$-affine subspaces.
For example, when $d=2$ and $D={3 \choose 2}=3$, a typical basis of the vector space of infinitesimal motions 
is the set of three infinitesimal motions consisting of two translations parallel to each axis and one rotation around the origin.
These infinitesimal motions are elegantly modeled 
by using vectors of length $D$, called {\em screw centers}, described below.

\subsubsection{Extensors}
It is known that rigidity properties are projectively invariant~\cite{crapo:1982} 
and hence it is useful to work with the $d$-dimensional projective space.
To do so, for any point $p_i=(p_{i,1},p_{i,2},\dots,p_{i,d})\in \mathbb{R}^d$, we will assign the homogeneous coordinate 
$(p_{i,1},p_{i,2},\dots,p_{i,d},1)$, denoted by  ${\bm p}_i$.

Let $U$ be a $(k-1)$-affine subspace of $\mathbb{R}^d$ determined by the points $p_1,\dots,p_k$.
We denote by $A({\bm p}_1,\dots,{\bm p}_k)$ the $k\times (d+1)$-matrix whose $i$-th row is ${\bm p}_i$.
For $1\leq i_1<i_2<\dots<i_{d-k+1}\leq d+1$, the {\em Pl{\"u}cker coordinate} $P_{i_1,i_2,\dots,i_{d-k+1}}$ of  $U$ is 
defined as the $(-1)^{1+i_1+i_2+\cdots+i_{d-k+1}}$ times  
the determinant of the $k\times k$-submatrix obtained from $A({\bm p}_1,\dots,{\bm p}_k)$ by deleting $i_j$-th columns for all $j$ with $1\leq j\leq d-k+1$.
The {\em Pl{\"u}cker coordinate vector} of $U$ is defined as the ${d+1 \choose k}$-dimensional vector obtained by writing down all of possible
Pl{\"u}cker coordinates of $U$ in some predetermined order, say, the lexicographic order of the indices (see e.g.~\cite{Jackson:07,White94} for more details).

Grassmann-Cayley algebra (see e.g.~\cite{White94}) treats a Pl{\"u}cker coordinate vector at a symbolic level, that is, no coordinate basis is specified,
and the symbolic version of a Pl{\"u}cker coordinate vector is referred to 
as a {\em $k$-extensor}, which is denoted by ${\bm p}_1\vee {\bm p}_2\vee \dots \vee {\bm p}_k$.
Although we will work on the coordinatized version, we would like to exploit this terminology to follow the conventional notation.

Let $P={\bm p}_1\vee \dots \vee {\bm p}_k$ and $Q={\bm q}_1\vee \dots \vee {\bm q}_l$.
Then, the {\em join} of $P$ and $Q$ is defined as $P\vee Q={\bm p}_1\vee\dots \vee {\bm p}_k\vee {\bm q}_1\vee \dots \vee {\bm q}_l$,
that is, a ${d+1 \choose k+l}$-dimensional vector consisting of $(k+l)\times (k+l)$-minors of $A({\bm p}_1,\dots,{\bm p}_k,{\bm q}_1,\dots,{\bm q}_l)$ if $k+l\leq d+1$ 
and otherwise $0$.
It is known that ${\bm p}_1\vee \dots \vee {\bm p}_k\neq {\bf 0}$ if and only if $\{{\bm p}_1,\dots,{\bm p}_k\}$ is linearly independent
and equivalently $\{p_1,\dots, p_k\}$ is affinely independent.
Also,  let us formally state the following property since it will be used later.
\begin{lemma}
\label{lemma:extensor}
Let $P=\{p_1,p_2,\dots,p_{d+1}\}$ be a set of $d+1$ points in $\mathbb{R}^d$ which is affinely independent.
Then, the set of $(d-1)$-extensors $\{{\bm p}_{j_1}\vee {\bm p}_{j_2}\vee \dots \vee {\bm p}_{j_{d-1}}\mid p_{j_i}\in P, 1\leq j_1<j_2<\dots <j_{d-1}\leq d+1\}$
is linearly independent.
\end{lemma}
\begin{proof}
Although this fact is implicitly used several times in e.g.~\cite{crapo:1982}, let us show a proof for the completeness.
Suppose that it is dependent. Then, there exist scalars $\lambda_{j_1,j_2,\dots,j_{d-1}}$ for all ${d+1 \choose d-1}$ indices,
indicating the dependence,
\begin{equation}
\label{eq:extensor_dependence} 
\sum_{1\leq j_1<\dots< j_{d-1}\leq d+1} \lambda_{j_1,j_2,\dots,j_{d-1}} {\bm p}_{j_1}\vee {\bm p}_{j_2} \vee \dots \vee {\bm p}_{j_{d-1}}={\bf 0}
\end{equation}
and at least one scalar must be nonzero.
Without loss of generality, we assume $\lambda_{1,2,\dots,d-1}\neq 0$.
Then, taking the join of (\ref{eq:extensor_dependence}) with ${\bm p}_{d}\vee {\bm p}_{d+1}$, we obtain 
$\lambda_{1,2,\dots, d-1}{\bm p}_1\vee \dots \vee {\bm p}_{d+1}=0$.
Since $P$ is affinely independent, we also have ${\bm p}_1\vee \dots \vee {\bm p}_{d+1}\neq 0$.
This in turn implies $\lambda_{1,2,\dots,d-1}=0$, which is a contradiction.
\end{proof}

\subsubsection{Rotations}
Let us review how to describe an infinitesimal rotation of a body
around a $(d-2)$-affine subspace $A$ in $\mathbb{R}^d$.
Let $p_1,\dots,p_{d-1}$ be $d-1$ points in $\mathbb{R}^d$ which affinely span $A$.
Then, $C(A)={\bm p}_1\vee \dots \vee {\bm p}_{d-1}$ is called a $(d-1)$-extensor associated with $A$.
Note that the 1-dimensional vector space spanned by $C(A)$ is determined independently of the choice of $p_1,\dots,p_{d-1}$.

Let $p\in\mathbb{R}^d$ be a point in the body, and let us consider $C(A)\vee {\bm p}$, which is a $d$-extensor associated 
with the $(d-1)$-affine subspace $H_p$ spanned by $\{p_1,\dots,p_{d-1},p\}$.
Then, $C(A)\vee {\bm p}$ is a $(d+1)$-dimensional vector and it is not difficult to see that the first $d$ coordinates represent a vector
normal to $H_p$.
In fact, letting  $x=(x_1,\dots,x_d)$ be an arbitrary point in $H_p$,  $C(A)\vee {\bm p}\vee {\bm x}=0$ represents the equation expressing the hyperplane $H_p$,
and the coefficient of $x_i$ in this equation is equal to the $i$-th coordinate of $C(A)\vee {\bm p}$.
More precisely, for some normal vector $v_p$ to $H_p$, $C(A)\vee {\bm p}$ can be expressed by $(v_p, -v_p \cdot p)$, where $\cdot$ means a dot product.
It is also known that the length of $v_p$ is proportional to the distance between $A$ and $p$ 
(and to the volume of the simplex determined  by $p_1,\dots,p_{d-1},p$).
This implies that, 
for some constant scalar $\alpha$, the first $d$ entries of $\alpha (C(A)\vee {\bm p})$ represents the velocity vector of the rotation around $A$ at $p$
(and $\alpha$ can be regarded as the angular velocity).
We will henceforth refer to the $D$-dimensional vector $\alpha C(A)$ as the {\em center} of the rotation.

\subsubsection{Translations}
We then describe an infinitesimal translation of a rigid body in the direction of a (free) vector $x \in\mathbb{R}^d$.
We will consider it as a rotation around an axis at infinity.
Let $\bar{x}_1,\dots,\bar{x}_{d-1}$ be a basis of the orthogonal complement of the vector space spanned by  $x$ in $\mathbb{R}^d$,
and let $\bar{\bm x}_i$ be the projective point at infinity to the direction $\bar{x}_i$, that is, $\bar{\bm x}_i=(\bar{x}_i,0)$ for each $i$.
The definition of $\vee$ is naturally extended to arbitrary projective points 
and hence let $C(x)=\bar{\bm x}_1\vee \dots \vee \bar{\bm x}_{d-1}$.

In this setting, we will obtain the same observation as rotations.
Let us consider an arbitrary point $p$ in the body and the $d$-extensor $C(x)\vee {\bm p}$.
It is not difficult to see that the vector consisting of the first $d$ coordinates is independent of $p$ and is proportional to $x$.
More precisely, for some constant scalar $\alpha$, 
we will have $\alpha C(x)\vee {\bm p}=(x,-x\cdot p)$ as observed above.
We may henceforth refer to the $D$-dimensional vector $\alpha C(x)$ as the {\em center} of the translation to the direction $x$.

\subsubsection{Arbitrary infinitesimal motions}
Recall that a motion of a body is a linear combination of rotations and translations.
Let $C_1,C_2,\dots,C_D$ be the centers corresponding to these rotations and translations.
Then, the infinitesimal motion at a point $p$ in the body is the first $d$ coordinates of $\sum_{i=1}^D (C_i\vee {\bm p})$.
Hence, we call the $D$-dimensional vector $S=\sum_{i=1}^{D} C_i$ the {\em screw center} of the infinitesimal motion.
We note that a screw center  cannot be represented as a $(d-1)$-extensor in general,
but a set of all $(d-1)$-extensors spans the $D$-dimensional vector space by Lemma~\ref{lemma:extensor}.
Therefore, the assignment of an infinitesimal motion to a body may be regarded as the assignment of 
a screw center, that is, the assignment of a $D$-dimensional vector.
We define $S\vee {\bm p}$ by $\sum_{i=1}^D (C_i\vee {\bm p})$.
%It is easy to check that the infinitesimal motion induced by a screw center $S$ 
%actually preserves the distances between two points, say $p$ and $q$ in $\mathbb{R}^d$.
%Suppose $u$ and $v$ represent the motions at $p$ and $q$ induced by $S$.
%Then, as we mentioned above, we have $S\vee {\bm p}=(u,-u\cdot p)$ and $S\vee {\bm q}=(v,-v\cdot q)$.
%Therefore, 
%$(u-v)\cdot(p-q)=u\cdot p-u\cdot q-v\cdot p+v\cdot q=-(u,-u\cdot p)\cdot(q,1)-(v,-v\cdot q)\cdot(p,1)=-S\vee {\bm p}\vee {\bm q}-S\vee {\bm q}\vee {\bm p}=0$,
%where the last equality follows from the definition of determinants.
%The infinitesimal motion thus preserves the length between two points~\cite{White94}.

\subsection{Body-and-hinge frameworks}
Suppose two bodies $B$ and $B'$ are joined to a hinge, which is a $(d-2)$-affine subspace $A$ of $\mathbb{R}^d$.
Recall that $C(A)$ denotes a $(d-1)$-extensor associated with $A$.
Also let us denote by $\langle C(A)\rangle$ the vector space spanned by $C(A)$.
These notations will be used throughout the paper.

Let us consider the situation in which screw centers $S$ and $S'$ are assigned to the bodies $B$ and $B'$, respectively. 
Then, the hinge constraint by $A$ imposes a relative motion of $B$ and $B'$ to be a rotation about $A$,
which means that there exists a constant scalar $\lambda$ satisfying $(S-S')\vee {\bm p}=\lambda (C(A)\vee {\bm p})$ for all $p\in B$.
It is known \cite{crapo:1982} that 
$(S-S')\vee {\bm p}=\lambda (C(A)\vee {\bm p})$ holds for any $p\in B$ and for some $\lambda$ if and only if 
$S-S'\in \langle C(A) \rangle$
(if $B$ affinely spans the $d$-dimensional space).

A {\em d-dimensional body-and-hinge framework} $(G,\bp)$ is a multigraph $G=(V,E)$ with a map $\bp$ 
which associates a $(d-2)$-affine subspace $\bp(e)$ of $\mathbb{R}^d$ with each $e\in E$.
An {\em infinitesimal motion} of $(G,\bp)$ is a map $S:V\rightarrow \mathbb{R}^D$ such that
\begin{equation}
\label{eq:hinge}
S(u)-S(v)\in\langle C(\bp(e))\rangle
\end{equation}
for every $e=uv\in E$.
Namely, $S$ is an assignment of a screw center $S(u)$ to the body of $u\in V$.
An infinitesimal motion $S$ is called {\em trivial} if $S(u)=S(v)$ for all $u,v\in V$, and 
$(G,\bp)$ is said to be {\em infinitesimally rigid} if all infinitesimal motions of $(G,\bp)$ are trivial.

\subsection{Rigidity matrix}
\label{subsec:rigidity_matrix}
The rigidity matrix is defined as the one whose null space is the set of infinitesimal motions of $(G,\bp)$.
Since $S$ is an infinitesimal motion of $(G,\bp)$ if and only if it satisfies (\ref{eq:hinge}),
taking any basis $\{r_1(\bp(e)),r_2(\bp(e)),\dots, r_{D-1}(\bp(e))\}$ 
of the orthogonal complement of $\langle C(\bp(e))\rangle$ in $\mathbb{R}^D$,
we can say that $S$ is an infinitesimal motion of $(G,\bp)$ if and only if 
\begin{equation*}
 (S(u)-S(v))\cdot r_i(\bp(e))=0
\end{equation*}
for all $i$ with $1\leq i\leq D-1$ and for all $e=uv\in E$.
Hence, the constraints to be an infinitesimal motion are described by $(D-1)|E|$ linear equations over $S(v)\in \mathbb{R}^{D}$ for all $v\in V$.
Consequently, we obtain $(D-1)|E|\times D|V|$-matrix $R(G,\bp)$ associated with this homogeneous system of linear equations
such that sequences of consecutive $(D-1)$ rows are indexed by elements of $E$ and 
sequences of consecutive $D$ columns are indexed by elements of $V$.
To describe the rigidity matrix more precisely, let us denote  the $(D-1)\times D$-matrix, whose $i$-th row vector is $r_i(\bp(e))$,  by
\begin{equation*}
r(\bp(e))=
\begin{pmatrix}
r_1(\bp(e)) \\
\vdots \\
r_{D-1}(\bp(e))
\end{pmatrix}.
\end{equation*}
Then, the submatrix $R(G,\bp;e,w)$ of $R(G,\bp)$ induced by the consecutive $(D-1)$ rows 
indexed by $e=uv\in E$ and the consecutive $D$ columns indexed by $w\in V$ is written as
\begin{equation}
\label{eq:matrix}
R(G,\bp;e,w)=\begin{cases}
r(\bp(e)) & \text{if  $w=u$} \\ 
-r(\bp(e))  & \text{if $w=v$} \\
{\bf 0} & \text{otherwise}.
\end{cases}
\end{equation}
We call $R(G,\bp)$ the {\em rigidity matrix} of $(G,\bp)$ (see Fig.~\ref{fig:matrix}).

\begin{figure}[t]
\centering
 \includegraphics[width=0.5\textwidth]{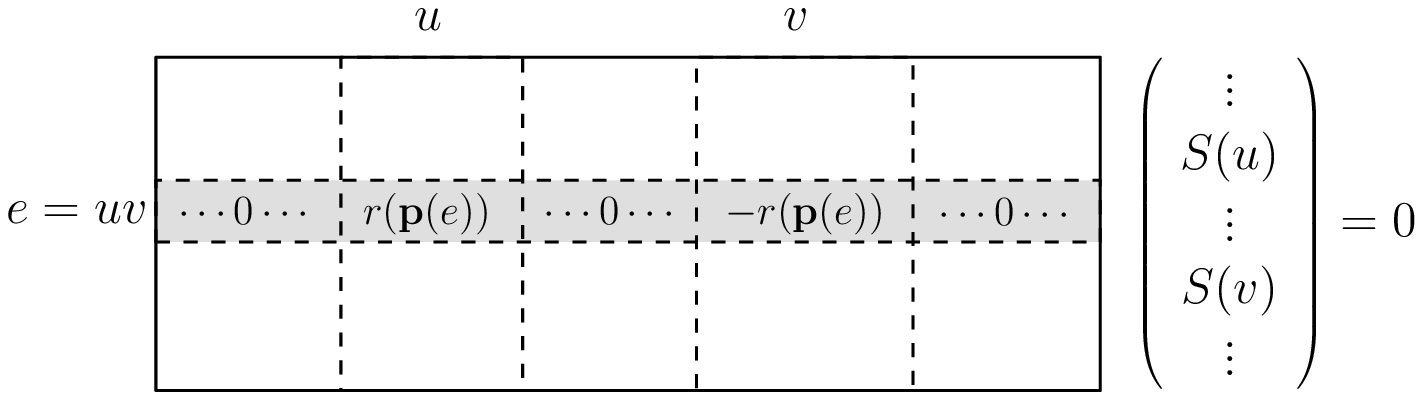}
\caption{The rigidity matrix and the homogeneous system of linear equations expressing the constraints for $S$ to be an infinitesimal motion.}
\label{fig:matrix}
\end{figure}

The null space of $R(G,\bp)$, which is the space of all infinitesimal motions, is denoted by $Z(G,\bp)$.
We remark that the rank of $Z(G,\bp)$, and equivalently the dimension of the space of all infinitesimal motions,
is uniquely determined by $(G,\bp)$ although the entries of $R(G,\bp)$ may vary depending on the choice of basis of the orthogonal complement of 
$\langle C(\bp(e))\rangle$.
%Let $S(V)$ be the $D|V|$-dimensional vector obtained by aligning $S(v)$ in lexicographical ordering of the vertices.

We sometimes regard an infinitesimal motion $S$ as a point in $\mathbb{R}^{D|V|}$, 
which is a composition of $|V|$ vectors $S(v)\in \mathbb{R}^D$ for $v\in V$, if it is clear from the context.
For $1\leq i\leq D$, let $S_i^*$ be the infinitesimal motion of $(G,\bp)$ such that,
for each $v\in V$, the $i$-th coordinate of $S_i^*(v)$ is $1$ and the others are $0$.
It is not difficult to see that $S_i^*$ is contained in $Z(G,\bp)$ and $S_i^*$ is a trivial infinitesimal motion.
The fact that  $\{S_1^*,S_2^*,\dots,S_D^* \}$ is linearly independent implies that the rank of $R(G,\bp)$ is at most $D|V|-D=D(|V|-1)$.
Notice also that $\{S_1^*,\dots,S_D^* \}$ spans the space of all trivial infinitesimal motions
and thus $(G,\bp)$ is infinitesimally rigid if and only if the rank of $R(G,\bp)$ is exactly $D(|V|-1)$.
More generally, the dimension of the space of nontrivial infinitesimal motions  
is called the {\em degree of freedom} of $(G,\bp)$, which is equal to $D(|V|-1)-\rank R(G,\bp)$.
A body-and-hinge framework $(G,\bp)$ is called {\em generic} if 
the ranks of $R(G,\bp)$ and its edge-induced submatrices take the maximum values over all realizations of $G$.

\section{Edge-disjoint Spanning Trees}
\label{sec:spanning_trees}
We use the following notations throughout the paper.
Let $G=(V,E)$ be a multigraph which may contain parallel edges but no self-loop.
For $X\subseteq V$, let $G[X]$ be the graph induced by $X$.
For $F\subseteq E$, let $V(F)$ be the set of the vertices spanned by $F$,
and let $G[F]$ be the graph edge-induced by $F$, i.e.,~$G[F]=(V(F),F)$.
For $X\subseteq V$, let $\delta_G(X)=\{uv\in E\mid u\in X, v\notin X\}$ and let $d_G(X)=|\delta_G(X)|$.
We shall omit set brackets when describing singleton sets, e.g.,~$d_G(v)$ implies $d_G(\{v\})$.
%We also use the notations $\delta_F(X)=\{uv\in F\mid u\in X, v\notin X\}$ and $d_F(X)=|\delta_F(X)|$ for  $X\subseteq V$ and $F\subseteq E$.
%Similarly, for two disjoint vertex sets $X_1, X_2\subset V$, 
%let $\delta_G(X_1,X_2)=\{uv\in E\mid u\in X_1, v\in X_2\}$ and let $d_G(X_1,X_2)=|\delta_G(X_1,X_2)|$. 
Throughout the paper, a {\em partition} ${\cal P}$ of  $V$ implies a collection $\{V_1,V_2,\dots,V_m\}$ of vertex subsets for some positive integer $m$ 
such that $V_i\neq\emptyset$ for $1\leq i\leq m$, $V_i\cap V_j=\emptyset$ for any $1\leq i,j\leq m$ with $i\neq j$ and $\bigcup_{i=1}^mV_i=V$.
Note that $\{V\}$ is a partition of $V$ for $m=1$.
Let  $\delta_{G}({\cal P})$ and $d_{G}({\cal P})$ denote the set, and the number, of edges of $G$ 
connecting distinct subsets of $\cal P$, respectively.

The result of Tay and Whiteley (Proposition~\ref{prop:tay}) 
reveals the strong relation between the rigidity of body-and-hinge frameworks and edge-disjoint spanning trees.
The following  Tutte-Nash-Williams disjoint tree theorem is well-known (see e.g.~\cite[Chapter 51]{Schriver}).
\begin{proposition}$(\cite{Nash:1961,Tutte:1961})$
\label{prop:nash}
A multigraph $H$ contains $c$ edge-disjoint spanning trees if and only if 
$d_H(\mathcal{P})\geq c(|\mathcal{P}|-1)$ holds
for each partition $\mathcal{P}$ of $V$.
\end{proposition}

We use the following conventional notation.
For a partition ${\cal P}$ of $V$ and a multigraph $G$, the {\em $D$-deficiency} of ${\cal P}$ in $\widetilde{G}$ is defined by
\begin{equation*}
{\rm def}_{\widetilde{G}}({\cal P})=D(|{\cal P}|-1)-d_{\widetilde{G}}({\cal P})=D(|{\cal P}|-1)-(D-1)d_G(\mathcal{P}),
\end{equation*}
and the {\em $D$-deficiency} of $\widetilde{G}$ is defined by
\begin{equation*}
{\rm def}(\widetilde{G})=\max\{{\rm def}_{\widetilde{G}}({\cal P}):{\cal P} \text{ is a partition of $V$}\}.
\end{equation*}
Note that ${\rm def}(\widetilde{G})\geq 0$ since ${\rm def}_{\widetilde{G}}(\{V\})=0$.
Proposition~\ref{prop:nash} implies that $\widetilde{G}$ has $D$ edge-disjoint spanning trees if and only if
${\rm def}(\widetilde{G})=0$.

There is the other well-known characterization of an edge set containing $D$ 
edge-disjoint spanning trees, which is written in terms of a {\em matroid} 
(see e.g.~\cite{oxley} for the definition and fundamental results of a matroid).
For $\widetilde{G}=(V,\widetilde{E})$, let us consider the matroid on $\widetilde{E}$, denoted by ${\cal M}(\widetilde{G})$, 
induced by the following nondecreasing submodular function $f:2^{\widetilde{E}}\rightarrow \mathbb{Z}$; for any $F\subseteq \widetilde{E}$,
$f(F)=D(|V(F)|-1)$.
Namely, $F\subseteq \widetilde{E}$ is {\em independent} in ${\cal M}(\widetilde{G})$ if and only if $|F'|\leq f(F')$ holds 
for every nonempty $F'\subseteq F$ (c.f.~\cite[Corollary~12.1.2]{oxley}).
It is known that $\widetilde{G}$ contains $D$ edge-disjoint spanning trees if and only if 
the rank of ${\cal M}(\widetilde{G})$, that is, the rank of $\widetilde{E}$ in ${\cal M}(\widetilde{G})$ is equal to $D(|V|-1)$.
We remark that ${\cal M}(\widetilde{G})$ is actually the union of $D$ graphic matroids on $\widetilde{E}$,
which means that an edge set is  independent if and only if it can be partitioned into $D$ edge-disjoint forests.

Proposition~\ref{prop:tay} now implies that a multigraph $G$ can be realized as an infinitesimally rigid body-and-hinge framework 
if and only if the rank of ${\cal M}(\widetilde{G})$ is equal to $D(|V|-1)$.
A more detailed relation between the deficiency of a graph and the rank of the rigidity matrix can be found in \cite{Jackson:07}. 
Let us summarize these preliminary results.
\begin{proposition}$(\cite{tay:89,whiteley:88,Jackson:07})$
\label{prop:preliminaries}
The followings are equivalent for a multigraph $G=(V,E)$: 
\begin{description}
\item[(i)] A generic body-and-hinge framework $(G,\bp)$ has  $k$ degree of freedom.
\item[(ii)] A generic body-and-hinge framework $(G,\bp)$ satisfies $\rank R(G,\bp)=D(|V|-1)-k$. 
\item[(iii)] ${\rm def}(\widetilde{G})=k$. 
\item[(iv)] The rank of ${\cal M}(\widetilde{G})$ is equal to $D(|V|-1)-k$, i.e.,
a base of ${\cal M}(\widetilde{G})$ can be partitioned into $D$ edge-disjoint forests 
whose total cardinality is equal to $D(|V|-1)-k$.
\end{description}
\end{proposition}

Let us note the following relation, observed from (iii) and (iv), between the deficiency and the cardinality of a base of ${\cal M}(\widetilde{G})$;
for a multigraph $G=(V,E)$ and a base $B$ of ${\cal M}(\widetilde{G})$,
\begin{equation}
\label{eq:pre}
|B|+\de(\widetilde{G})=D(|V|-1).
\end{equation}
This is equivalent to the forest packing theorem described in \cite[Theorem~51.1]{Schriver}.
\section{Body-and-hinge Rigid Graphs}
\label{sec:body-and-hinge rigid graph}
In this section we shall further investigate combinatorial properties of body-and-hinge frameworks.
%Since Molecular Conjecture in $d=2$ has been proved by Jackson and Jord{\'a}n~\cite{Jackson:06},
%we shall consider the case of 
%$d\geq 3$ and we assume $d\geq 3$ and $D={d+1\choose 2}\geq 6$ in the subsequent discussions.
Let $G=(V,E)$ be a multigraph.
Proposition~\ref{prop:preliminaries} says that a graph $G$ satisfying ${\rm def}(\widetilde{G})=k$ for some integer $k$ can be realized as 
a generic body-and-hinge framework having $k$ degree of freedom.
Inspired by this fact,
we simply say that $G$ is a {\em $k$-dof-graph}  if ${\rm def}(\widetilde{G})=k$ for some nonnegative integer $k$.
In particular, to emphasize the relation between $0$-dof-graphs 
and infinitesimal rigidity given in Proposition~\ref{prop:tay}, 
we sometimes refer to a $0$-dof-graph as a {\em body-and-hinge rigid graph}.

Recall that $\widetilde{E}$ denotes the edge set of $\widetilde{G}$.
For $e\in E$, let $\widetilde{e}$ denote the {\em set} of corresponding $D-1$ parallel copies of $e$ 
in $\widetilde{E}$.
For $F\subseteq E$, let $\widetilde{F}=\bigcup_{e\in F} \widetilde{e}$.
We index the edges of $\widetilde{e}$ by $1\leq i\leq D-1$,
and $e_i$, or $(e)_i$, denotes the $i$-th element in $\widetilde{e}$.
A graph is called {\em $k$-edge-connected} for some $k\geq 1$ if removing any $k-1$ edges results in a connected graph. 
It is not difficult to see the following fact.
\begin{lemma}
\label{lemma:2connectivity}
Let $G$ be a body-and-hinge rigid graph.
Then, $G$ is 2-edge-connected.
\end{lemma}
\begin{proof}
Suppose that $G$ is not 2-edge-connected.
Then, there exists a nonempty subset $V'$ of $V$ satisfying $d_G(V')\leq 1$ and $V\setminus V'\neq \emptyset$.
Consider a partition ${\cal P}=\{V',V\setminus V'\}$ of $V$.
Then, we have $\de(\widetilde{G})\geq D(|{\cal P}|-1)-(D-1)d_G(\widetilde{G})\geq 1$,
contradicting $\de(\widetilde{G})=0$.
\end{proof}

\paragraph{Remark.}
Let $b$ and $c$ be positive integers and let $q=b/c$.
A multigraph $G=(V,E)$ is  called {\em $q$-strong} if $cG$ contains $b$ edge-disjoint spanning trees.
The $q$-strong graph was first introduced by Gusfield~\cite{Gusfield:1983},
where he considered the maximum value of $q$ for $G$ to be $q$-strong.
This value was later called the {\em strength} of $G$ by Cunningham~\cite{cunningham:1985}.
Checking whether $G$ is $q$-strong or not can be solved in polynomial time, if $q$ is regarded as a constant,
by explicitly constructing $cG$ and checking the existence of $b$ edge-disjoint spanning trees in it,
which can be efficiently done by a forest packing algorithm~\cite{Gabow:92} 
(and hence it can be checked in polynomial time whether $G$ is a body-and-hinge rigid graph).    
Cunningham~\cite{cunningham:1985} provided a strongly polynomial time algorithm for checking 
whether $G$ is $q$-strong and also computing the strength of $G$.
The concept of the strength has been extended to a general matroid by Catlin et al.~\cite{Catlin:1992}. 

In this paper, for our particular interest in $q=\frac{D}{D-1}$, 
we named a $\frac{D}{D-1}$-strong graph as a body-and-hinge rigid graph.
For the rigidity of body-and-hinge frameworks, 
Jackson and Jord{\'a}n~\cite{Jackson:05a,Jackson:07a} recently provided several results for the $q$-strength of multigraphs,
which are basically concerned with partitions of $V$ maximizing the deficiency.
They also mentioned a {\em minimally $q$-strong graph} (with respect to edge-inclusion) 
and showed that a $q$-strong subgraph contained in a minimally $q$-strong graph is also minimal,
which is a special case of Lemma~\ref{lemma:subgraph} given in the next subsection.
%However there is no result concerning minimally $q$-strong graphs except for this claim to the best of our knowledge.
%In the subsequent sections we shall newly reveal several combinatorial properties of a minimally $\frac{D}{D-1}$-strong graph,
%which is called a {\em minimally body-and-hinge rigid graph}.

\subsection{Minimally body-and-hinge rigid graphs}
A {\em minimal $k$-dof-graph} is a $k$-dof-graph in which removing any edge results in a graph that is not a $k$-dof-graph.
In particular, a minimal $0$-dof-graph is called a {\em minimally body-and-hinge rigid graph}.
In this section we shall prove several new combinatorial properties of a minimal $k$-dof-graph,
which will be utilized in the proof of the Molecular Conjecture.

\begin{lemma}
\label{lemma:3connectivity}
Let $G$ be a minimal $k$-dof-graph for some nonnegative integer $k$.
Then, $G$ is not 3-edge-connected.
\end{lemma}
\begin{proof}
Suppose, for a contradiction, that $G$ is 3-edge-connected.
We shall show that the graph $G_e$ obtained by removing an edge $e=uv\in E$ is still a $k$-dof-graph,
which contradicts the minimality of $G$.

Consider any partition ${\cal P}=\{V_1,V_2,\dots,V_{|{\cal P}|}\}$ of $V$.
If $u$ and $v$ are both in the same vertex subset of ${\cal P}$, then $d_G({\cal P})=d_{G_e}({\cal P})$, 
and consequently ${\rm def}_{\widetilde{G}}({\cal P})={\rm def}_{\widetilde{G_e}}({\cal P})\leq k$ holds. 

Suppose $u$ and $v$ are contained in the distinct subsets of ${\cal P}$.
Without loss of generality, we assume that $u\in V_1$  and $v\in V_2$.
Since $G$ is 3-edge-connected,
we have $d_{G_e}(V_i)\geq 2$ for $i=1,2$ and $d_{G_e}(V_i)\geq 3$ for $i=3,\dots,|{\cal P}|$.
Hence, we have
$d_{G_e}({\cal P})\geq \lceil\frac{3(|{\cal P}|-2)+2\cdot 2}{2}\rceil\geq \frac{3}{2}|{\cal P}|-1$, which implies,
\begin{align*}
\de_{\widetilde{G_e}}({\cal P})&\leq D(|{\cal P}|-1)-(D-1)\left(\frac{3}{2}|{\cal P}|-1\right)  
=-\frac{|{\cal P}|}{2}(D-3)-1.
\end{align*}
Since $d\geq 2$ and  $D\geq 3$, we obtain $\de_{\widetilde{G_e}}({\cal P})<0\leq k$.
Consequently, $\de_{\widetilde{G_e}}({\cal P})\leq k$ holds for any partition ${\cal P}$ of $V$, 
implying that $G_e$ is a $k$-dof-graph and contradicting the minimality of $G$.
\end{proof}

For a multigraph $G=(V,E)$, (\ref{eq:pre}) implies that
an edge $e\in E$ can be removed from $G$ without changing the deficiency
of $\widetilde{G}$ if and only if there exists a base $B$ of the matroid ${\cal M}(\widetilde{G})$ 
such that $B\cap \widetilde{e}=\emptyset$.
Equivalently, a graph $G=(V,E)$ is a minimal $k$-dof-graph for some nonnegative integer $k$
if and only if $B\cap \widetilde{e}\neq\emptyset$ 
for any edge $e\in E$ and any base $B$ of ${\cal M}(\widetilde{G})$.
From this observation, it is not difficult to see the following fact.
\begin{lemma}
\label{lemma:subgraph}
Let $G=(V,E)$ be a minimal $k$-dof-graph for some nonnegative integer $k$
and let $G'=(V',E')$ be a subgraph of $G$.
Suppose $G'$ is a $k'$-dof-graph for some nonnegative integer $k'$.
Then $G'$ is a minimal $k'$-dof-graph.
\end{lemma}
\begin{proof}
Consider ${\cal M}(\widetilde{G'})$, 
which is the matroid ${\cal M}(\widetilde{G})$ {\em restricted}  to $\widetilde{E'}$.
Recall that the set of bases of ${\cal M}(\widetilde{G'})$ 
is the set of maximal members of $\{B\cap \widetilde{E'}\mid B \text{ is a base of }{\cal M}(\widetilde{G})\}$
(see e.g.~\cite[Chapter~3, 3.1.15]{oxley}).
Since $B\cap \widetilde{e}\neq \emptyset$ holds for any base $B$ of ${\cal M}(\widetilde{G})$ and 
for any $e\in E$ from the minimality of $G$,
$(B\cap \widetilde{E'})\cap \widetilde{e}\neq\emptyset$ holds for any base $B$ of ${\cal M}(\widetilde{G})$ and for any $e\in E'$.
Since any base $B'$ of ${\cal M}(\widetilde{G'})$ can be written as $B\cap \widetilde{E'}$ with some base $B$ of ${\cal M}(\widetilde{G})$,
$B'\cap \widetilde{e}\neq \emptyset$ holds for any $e\in E'$.
This implies that $G'$ is a minimal $k'$-dof-graph for some nonnegative integer $k'$.
\end{proof}
 
In a multigraph, an edge pair is called a {\em cut pair} if the removal of these two edges disconnects the graph. 
By Lemmas~\ref{lemma:2connectivity} and~\ref{lemma:3connectivity}, 
we see that any minimally body-and-hinge rigid graph $G$ has a cut pair.
Using this property, we can actually show a nice combinatorial property; any $2$-edge-connected minimal $k$-dof-graph contains 
a vertex of degree two or three. 
Since this is not directly used in our proof of the Molecular Conjecture,
we omit the proof. 
Later, we shall also present a similar property in Lemma~\ref{lemma:degree2}.

\subsection{Rigid subgraphs}
Let $G$ be a multigraph.
We say that a subgraph $G'$ of $G$ is a {\em rigid subgraph}
if $G'$ is a $0$-dof-graph, i.e., $\widetilde{G'}$ contains $D$ edge-disjoint spanning trees on the vertex set of $G'$.
In this subsection we claim the following three lemmas related to  rigid subgraphs.
%It is not difficult to see the following properties.
\begin{lemma}
\label{lemma:circuit_rigidity}
Let $G=(V,E)$ be a multigraph and let $X$ be a circuit of the matroid ${\cal M}(\widetilde{G})$.
Then, $G[V(X)]$ is a rigid subgraph of $G$.
More precisely, $X-e$ can be partitioned into $D$ edge-disjoint spanning trees on $V(X)$ for any $e\in X$. 
\end{lemma}
\begin{proof}
A circuit $X$ is a minimal dependent set of ${\cal M}(\widetilde{G})$  satisfying $|X|>D(|V(X)|-1)$,
and $X-e$ is independent in ${\cal M}(\widetilde{G})$ for any $e\in X$, see e.g.~\cite{oxley}.
From $|X|>D(|V(X)|-1)$, we have $|X-e|\geq D(|V(X)|-1)$. 
On the other hand, since $X-e$ is independent, 
we also have $|X-e|\leq D(|V(X-e)|-1)\leq D(|V(X)|-1)$.
As a result, $|X-e|=D(|V(X)|-1)$ holds and 
hence $X-e$ can be partitioned into $D$ edge-disjoint spanning trees on $V(X)$.
This implies that the graph $G[V(X)]$ induced by $V(X)$ is a $0$-dof-graph and equivalently a rigid subgraph.
\end{proof}

\begin{lemma}
\label{lemma:contraction}
Let $G=(V,E)$ be a minimal $k$-dof-graph for a nonnegative integer $k$
and let $G'=(V',E')$ be a rigid subgraph of $G$.
Then, the graph obtained from $G$ by contracting $E'$ is a minimal $k$-dof-graph.
\end{lemma}
\begin{proof}
Let $H$ be a graph obtained by contracting $E'$. 
By Lemma~\ref{lemma:2connectivity}, $G'$ is connected 
and hence $V'$ becomes a single vertex after the contraction of $E'$.
Let $v^*$ be this new vertex in $H$, that is, $H=((V\setminus V')\cup\{v^*\},E\setminus E')$.
%For an edge set $F\subseteq \widetilde{E}$, 
%we use the notation $F/\widetilde{E}'$ 
%to denote the edge set corresponding to $F\setminus \widetilde{E}'$ in $\widetilde{H}$.

Let $B_{\widetilde{G}'}$ be a base of ${\cal M}(\widetilde{G'})$.
Then, we have  $|B_{\widetilde{G}'}|=D(|V'|-1)$ since $G'$ is a $0$-dof-graph. 
Also, there exists a base  $B$ of ${\cal M}(\widetilde{G})$ which contains $B_{\widetilde{G}'}$ as its subset.
%Note that $|B|=D(|V|-1)-k$ since $G$ is $k$-dof-graph.
Let $\{F_1,F_2,\dots,F_D\}$ be a partition of $B$ into $D$ edge-disjoint forests on $V$.
We claim the followings: 
\begin{equation}
\label{eq:contraction1}
\text{$F_i\cap \widetilde{E}'$ forms a spanning tree on $V'$ for each $1\leq i\leq D$.}
\end{equation}
To see this, notice that $B_{\widetilde{G}'}\subset B\cap \widetilde{E}'$ implies 
$|B\cap \widetilde{E}'|\geq |B_{\widetilde{G}'}|=D(|V'|-1)$.
On the other hand, from the fact that $F_i\cap \widetilde{E}'$ is independent in a graphic matroid,
we also have $|F_i\cap \widetilde{E}'|\leq |V(F_i\cap \widetilde{E}')|-1\leq |V(B\cap \widetilde{E}')|-1\leq |V'|-1$ for each $1\leq i\leq D$.
These imply $|B\cap \widetilde{E}'|=\sum_{i=1}^D |F_i\cap \widetilde{E}'|\leq D(|V'|-1)\leq |B\cap \widetilde{E}'|$ 
and the equalities hold everywhere, implying $|F_i\cap \widetilde{E}'|=|V'|-1$. Thus, (\ref{eq:contraction1}) holds.

Due to (\ref{eq:contraction1}), after the contraction of  $\widetilde{E'}$, 
$F_i\setminus \widetilde{E}'$ does not contain a cycle in $\widetilde{H}$ and again forms a forest on $(V\setminus V')\cup\{v^*\}$.
This implies that $\{F_1\setminus \widetilde{E}', \dots, F_D\setminus \widetilde{E}'\}$ is a partition of $B\setminus \widetilde{E}'$ into 
$D$ edge-disjoint forests on $(V\setminus V')\cup\{v^*\}$ and hence $B\setminus \widetilde{E}'$ is independent in ${\cal M}(\widetilde{H})$.
Since
$|B\setminus \widetilde{E}'|=|B|-|B_{\widetilde{G}'}|=D(|V|-1)-k-D(|V'|-1)=D(|V\setminus V'\cup\{v^*\}|-1)-k$,
$\de (\widetilde{H})\leq k$ follows from (\ref{eq:pre}).

To see $\de(\widetilde{H})\geq k$, let us consider a base $B_{\widetilde{H}}\subseteq \widetilde{E}\setminus \widetilde{E}'$ of ${\cal M}(\widetilde{H})$.
Let $\{S_1,\dots,S_D\}$ be a partition of $B_{\widetilde{H}}$ into $D$ edge-disjoint forests on $(V\setminus V')\cup\{v^*\}$.
Also, since $G'$ is a $0$-dof-graph, a base $B_{\widetilde{G}'}$ of ${\cal M}(\widetilde{G'})$ can be partitioned into 
$D$ edge-disjoint spanning trees $\{T_1,\dots,T_D\}$ on $V'$.
Then, it is not difficult to see that $S_i\cup T_i$ forms a forest on $V$ for each $i$,
and thus $B_{\widetilde{H}}\cup B_{\widetilde{G}'}$ is an independent set of ${\cal M}(\widetilde{G})$.
This implies $|B_{\widetilde{H}}\cup B_{\widetilde{G}'}|\leq D(|V|-1)-k$.
Substituting $|B_{\widetilde{H}}|=D(|V\setminus V'\cup\{v^*\}|-1)-\de(\widetilde{H})$ and $|B_{\widetilde{G}'}|=D(|V'|-1)$,
we obtain $\de(\widetilde{H})\geq k$.

The minimality of $H$ can be checked by the same argument.
Suppose, for a contradiction, that there exists a base $B_{\widetilde{H}}'$ of ${\cal M}(\widetilde{H})$ which contains no 
edge of $\widetilde{e}$ for some $e\in E\setminus E'$.
Then, $B_{\widetilde{H}}'\cup B_{\widetilde{G}'}$ is again a base of ${\cal M}(\widetilde{G})$ 
which contains no edge of $\widetilde{e}$, contradicting the minimality of the original graph $G$.
\end{proof}

Notice that, for every circuit $X$ of ${\cal M}(\widetilde{G})$,
$V(X)$ induces a 2-edge-connected subgraph by 
Lemma~\ref{lemma:2connectivity} and Lemma~\ref{lemma:circuit_rigidity}.
This fact leads to the following property of a multigraph that is not 2-edge-connected.
\begin{lemma}
\label{lemma:small_connectivity}
Let $G=(V,E)$ be a minimal $k$-dof-graph.
Let ${\cal P}=\{V_1,V_2\}$ be a partition of $V$ and let $G_i=G[V_i]$ for $i=1,2$.
Then, we have the followings:
\begin{description}
\item If $d_G({\cal P})=1$, then $k=\de(\widetilde{G}_1)+\de(\widetilde{G}_2)+1$.
\item If $d_G({\cal P})=0$, then $k=\de(\widetilde{G}_1)+\de(\widetilde{G}_2)+D$.
\end{description}
\end{lemma}
\begin{proof}
Let us consider the case of $d_G({\cal P})=1$.
We denote by $uv$ the edge connecting between $V_1$ and $V_2$ with $u\in V_1$ and $v\in V_2$.
Let $E_1$ and $E_2$ be the edge sets of $G_1$ and $G_2$
and let $E_3=\{uv\}$ and $G_3=(\{u,v\},E_3)$.
Then, $\{E_1,E_2,E_3\}$ is a partition of $E$ 
and moreover there exists no circuit in ${\cal M}(\widetilde{G})$ which intersects more than one set among 
$\{\widetilde{E}_1,\widetilde{E}_2,\widetilde{E}_3\}$  
since any circuit induces a 2-edge-connected subgraph in $G$ by Lemma~\ref{lemma:2connectivity} and Lemma~\ref{lemma:circuit_rigidity}.
Thus, we can decompose the matroid as 
${\cal M}(\widetilde{G})={\cal M}(\widetilde{G}_1)\oplus {\cal M}(\widetilde{G}_2)\oplus {\cal M}(\widetilde{G}_3)$, 
where $\oplus$ denotes the {\em direct sum} of the matroids (see~\cite[Chapter~4, 4.2.16]{oxley}).
Since the rank of the direct sum of matroids is the sum of the ranks of these matroids~\cite[Chapter~4, 4.2.17]{oxley},
we obtain $D(|V|-1)-k=D(|V_1|-1)-\de(\widetilde{G}_1)+D(|V_2|-1)-\de(\widetilde{G}_2)+D-1$,
where we used the obvious fact that the rank of 
${\cal M}(\widetilde{G_3})$ is equal to $D-1$.
By $|V|=|V_1|+|V_2|$, we eventually obtain the claimed relation.

The proof for the case $d_G({\cal P})=0$ is basically the same, and hence it is omitted.
\end{proof}

\section{Operations for Minimal $k$-dof-graphs}
\label{sec:inductive}
In this section we shall discuss two simple operations on a minimal $k$-dof-graph.
One operation is the contraction of a {\em proper} rigid subgraph;
$G'=(V',E')$ is called a {\em proper rigid subgraph} if it is
a rigid subgraph of $G$ satisfying $1<|V'|<|V|$.
We have already seen in Lemma~\ref{lemma:contraction} that the contraction of a rigid subgraph produces a smaller minimal $k$-dof-graph.
Another operation is a so-called splitting off operation, whose definition will be given in the next subsection.
Our goal of this section is to show Lemma~\ref{lemma:operation},
which states that any minimal $k$-dof-graph can be always converted to 
a smaller minimal $k$-dof-graph or minimal $(k-1)$-dof-graph by a contraction of a proper rigid subgraph or 
a splitting off at a vertex of degree two.
We will use this result in the proof of the Molecular Conjecture by induction.
Also, as a corollary,  we will obtain Theorem~\ref{theorem:inductive}; any minimally body-and-hinge rigid graph 
can be constructed by a sequence of these two simple operations, which must be an interesting result in its own right.

%The section is organized as follows. 
%In Subsection~\ref{subsec:splitting_off}, we will define a splitting off operation and then reveal the properties of this operation.
%On the other hand, in Subsection~\ref{subsec:rigid_subgraph} we will investigate when a graph contains a proper rigid subgraph.
%Combining these results, we derive Lemma~\ref{lemma:operation} in the last subsection.

\subsection{Splitting off operation at a vertex of degree two}
\label{subsec:splitting_off}
In this subsection, we shall examine a splitting off  operation that converts a minimal $k$-dof-graph  
into a smaller graph,
which is analogous to that for $2k$-edge-connected graphs~\cite{lovasz:1979}. 
For a vertex $v$ of a graph $G$, let $N_G(v)$ be a set of vertices adjacent to $v$ in $G$.
A {\em splitting off at $v$} is an operation which removes $v$
and then inserts new edges between vertices of $N_G(v)$.
We shall consider such an operation only at a vertex $v$ of degree two.
Let $N_G(v)=\{a,b\}$.
We denote by $G_v^{ab}$ the graph obtained from $G$ by removing $v$ (and the edges incident to $v$) 
and then inserting the new edge $ab$.
The operation that produces $G_v^{ab}$ from $G$ is called {\em the splitting off at $v$ (along $ab$)}.
The main result of this subsection is Lemma~\ref{lemma:splitting_off}, which claims that the splitting off does not increase the deficiency
but may not preserve the minimality of the resulting graph.
Before showing Lemma~\ref{lemma:splitting_off}, let us first investigate the  relation 
between independent sets of ${\cal M}(\widetilde{G})$ and those of ${\cal M}(\widetilde{G_v^{ab}})$ in the following lemmas. 
\begin{lemma}
\label{lemma:splitting0}
Let $G=(V,E)$ be a $k$-dof-graph which has a vertex $v$ of degree $2$ with $N_G(v)=\{a,b\}$.
For any independent set $I$ of ${\cal M}(\widetilde{G})$,
there exists an independent set $I'$ of ${\cal M}(\widetilde{G_v^{ab}})$ 
satisfying $|I'|=|I|-D$ and $|\widetilde{ab}\cap I'|<D-1$. 
\end{lemma}
\begin{proof}
Let $h=|(\widetilde{va}\cup \widetilde{vb})\cap I|$.
Let $\{F_1,\dots,F_D\}$ be a partition of $I$ into $D$ edge-disjoint forests on $V$.
Since $d_G(v)=2$, clearly $d_{F_i}(v)\leq 2$ holds for each $i=1,\dots,D$,
(where $d_{F_i}(v)$ denotes the number of edges of $F_i$ incident to $v$).
Let $h'$ be the number of forests $F_i$ satisfying $d_{F_i}(v)=2$.
Note that $2h'+(D-h')=h$ and $h\leq 2(D-1)$, which imply $h'=h-D< D-1$.
For a forest $F_i$ satisfying $d_{F_i}(v)=1$,
removing the edge incident to $v$ results in a forest on $V-v$.
For a forest $F_i$ satisfying $d_{F_i}(v)=2$, 
removing the edges incident to $v$ and inserting an edge of $\widetilde{ab}$,
we can also obtain a forest on $V-v$.
We hence convert each $F_i$ to a forest $F_i'$ on $V-v$ by the above operations 
such that $|F_i'|=|F_i|-1$ for each $i$.
Moreover, since the total number of edges of $\widetilde{ab}$ needed to convert $F_i$ to $F_i'$ 
is equal to $h'$, which is less than $|\widetilde{ab}|=D-1$,
$F_1',\dots,F_D'$ can be taken to be edge-disjoint in $\widetilde{G_v^{ab}}$.
Let $I'=\bigcup_{i=1}^D F_i'$.
Then, clearly, $I'$ is an independent set of ${\cal M}(\widetilde{G_v^{ab}})$ with the cardinality $|I'|=|I|-D$ as required.
\end{proof}

The inverse operation of the splitting off at a vertex of degree two is called {\em edge-splitting}.
More formally, the edge-splitting (along an edge $ab$) is the operation that removes an edge $ab$
and then inserts a new vertex $v$ with the two new edges $va$ and $vb$.
The following lemma supplies the converse direction of Lemma~\ref{lemma:splitting0}.
\begin{lemma}
\label{lemma:splitting}
Let $H=(V,E)$ be a $k$-dof-graph, $ab$ be an edge of $H$, and $H_{ab}^v$ 
be the graph obtained by the edge-splitting along $ab$.
Let $I'$ be an independent set of ${\cal M}(\widetilde{H})$ with $h'=|\widetilde{ab}\cap I'|$.
Then, (i) if $h'< D-1$, there exists an independent set $I$ of ${\cal M}(\widetilde{H_{ab}^v})$ 
satisfying $|I|=|I'|+D$ and $|I\cap \widetilde{vb}|=h'+1$ 
 and (ii) otherwise there exists an independent set $I$ of ${\cal M}(\widetilde{H_{ab}^v})$ 
satisfying $|I|=|I'|+D-1$.
\end{lemma}
\begin{proof}
%The proof is basically done by performing the inverse operation of that given in the previous lemma.
Let $\{F_1',\dots,F_D'\}$ be a partition of $I'$ into $D$ edge-disjoint forests on $V$.
Without loss of generality, we assume $(ab)_i\in F_i'$ for each $1\leq i\leq h'$. 

Let us first consider the case of $h'< D-1$.
Consider the extension of each forest as follows: $F_i=F_i'-(ab)_i+(va)_i+(vb)_i$ for each $1\leq i\leq h'$,
$F_i=F_i'+(va)_{i}$ for $h'+1\leq i\leq D-1$ and $F_D=F_D'+(vb)_{h'+1}$ (for $i=D$).
Then, $F_1,\dots,F_D$ are $D$ edge-disjoint forests contained in $\widetilde{H_{ab}^v}$
and $\bigcup_{i=1}^{D} F_i$, denoted by $I$, is an independent set of ${\cal M}(\widetilde{H_{ab}^v})$.
Since $|F_i|=|F_i'|+1$, we have $|I|=|I'|+D$ and also $|I\cap \widetilde{vb}|=h'+1$ as required.

When $h'=D-1$, let $F_i=F_i'-(ab)_i+(va)_i+(vb)_i$ for each $1\leq i\leq D-1$ and 
$F_D=F_D'$ (for $i=D$).
Then, $\bigcup_{i=1}^{D} F_i$ is an independent set of ${\cal M}(\widetilde{H_{ab}^v})$ with
the cardinality $|I'|+D-1$.
\end{proof}

We are now ready to discuss the deficiency of the graph obtained by a splitting off operation. 
\begin{lemma}
\label{lemma:splitting_off}
Let $G=(V,E)$ be a minimal $k$-dof-graph which has a vertex $v$ of degree $2$ with $N_G(v)=\{a,b\}$.
Then, (i) $G_v^{ab}$ is either a $k$-dof-graph or a minimal $(k-1)$-dof-graph.
Moreover, (ii) $G_v^{ab}$ is a $k$-dof-graph 
if and only if there is a base $B'$ of ${\cal M}(\widetilde{G_v^{ab}})$ satisfying $|\widetilde{ab}\cap B'|<D-1$.
\end{lemma}
\begin{proof}
Let $B$ be a base of ${\cal M}(\widetilde{G})$.
Then, by Lemma~\ref{lemma:splitting0},
there exists an independent set $I'$ of ${\cal M}(\widetilde{G_v^{ab}})$ satisfying $|I'|=|B|-D$ and $|\widetilde{ab}\cap I'|<D-1$.
Since $|I'|=|B|-D=D(|V|-1)-k-D=D(|V\setminus\{v\}|-1)-k$, the rank of ${\cal M}(\widetilde{G_v^{ab}})$ is at least $D(|V\setminus\{v\}|-1)-k$.
This implies
\begin{equation}
\label{eq:splitting_off}
\de(\widetilde{G_v^{ab}})\leq k 
\end{equation}
by (\ref{eq:pre}),
where the equality holds if and only if $I'$ is a base of ${\cal M}(\widetilde{G_v^{ab}})$.
Thus, $\de(\widetilde{G_v^{ab}})=k$ holds 
if and only if there is a base $B'$ of ${\cal M}(\widetilde{G_v^{ab}})$ with $|\widetilde{ab}\cap B'|< D-1$. 
In this case, $G^{ab}_v$ is a $k$-dof-graph.

%Conversely, suppose there is a base $B'$ of ${\cal M}(\widetilde{G_v^{ab}})$ satisfying $|\widetilde{ab}\cap B'|< D-1$.
%Then, by Lemma~\ref{lemma:splitting} (i),
%there exists an independent set $I$ of ${\cal M}(\widetilde{G})$ satisfying $|I|=|B'|+D$,
%which is equal to $D(|V\setminus\{v\}|-1)-\de(\widetilde{G_v^{ab}})+D=D(|V|-1)-\de(\widetilde{G_v^{ab}})$.
%This implies that a base $B$ of $\de(\widetilde{G})$ satisfies $|B|\geq D(|V|-1)-\de(\widetilde{G_v^{ab}})$
%and hence we have $D(|V|-1)-k=|B|\geq D(|V|-1)-\de(\widetilde{G_v^{ab}})$ by (\ref{eq:pre}).
%Thus, we obtain $k\leq \de(\widetilde{G_v^{ab}})$. 
%Combining it with (\ref{eq:splitting_off}) we found that $\de(\widetilde{G_v^{ab}})=k$, which implies that $G_v^{ab}$ is a $k$-dof-graph.
%As a result, $G_v^{ab}$ is a $k$-dof-graph if and only if there is a base $B'$ of ${\cal M}(\widetilde{G_v^{ab}})$ satisfying $|\widetilde{ab}\cap B'|<D-1$.

Let us consider the case where every base $B'$ of ${\cal M}(\widetilde{G_v^{ab}})$ satisfies $|\widetilde{ab}\cap B'|=D-1$.
In this case $G_v^{ab}$ is not a $k$-dof-graph and hence (\ref{eq:splitting_off}) implies
\begin{equation}
\label{eq:splitting_off2}
\de(\widetilde{G_v^{ab}})\leq k-1.
\end{equation}  
By Lemma~\ref{lemma:splitting} (ii), 
there exists an independent set $J$ of ${\cal M}(\widetilde{G})$ satisfying 
$|J|=|B'|+D-1=D(|V|-1)-(\de(\widetilde{G_v^{ab}})+1)$,
where $|B'|=D(|V\setminus\{v\}|-1)-\de(\widetilde{G_v^{ab}})$.
We thus obtain $k=\de(\widetilde{G})\leq \de(\widetilde{G_v^{ab}})+1$ by (\ref{eq:pre}).
Combining it with (\ref{eq:splitting_off2}), we eventually obtain $\de(\widetilde{G_v^{ab}})=k-1$.
Thus, $G^{ab}_v$ is a $(k-1)$-dof-graph.

The minimality of $G_v^{ab}$ can be checked by using Lemma~\ref{lemma:splitting}(ii) again.
Suppose there exists an edge $e$ such that $G_v^{ab}-e$ is still a $(k-1)$-dof-graph.
Note $e\neq ab$ since every base $B'$ of ${\cal M}(\widetilde{G_v^{ab}})$ 
satisfies $|\widetilde{ab}\cap B'|=D-1$ now.
Hence, taking a base which contains no edge of $\widetilde{e}$ in ${\cal M}(\widetilde{G_v^{ab}})$
and then extending it by applying Lemma~\ref{lemma:splitting}(ii) we will have a base of ${\cal M}(\widetilde{G})$
which also contains no edge of $\widetilde{e}$, 
contradicting the minimality of $G$. 
Therefore, $G_v^{ab}$ is a minimal $(k-1)$-dof-graph.
\end{proof}

Applying Lemma~\ref{lemma:splitting_off} to the case of $k=0$,
we see that, for a minimally body-and-hinge rigid graph $G$, 
$G_v^{ab}$ is always body-and-hinge rigid. 
However, as we mentioned, a splitting off may not preserve the minimality of $G_v^{ab}$.
%In fact, the splitting off at $v$ along $ab$ does not preserve the minimality in general.
For example, for a minimally body-and-hinge rigid graph $G'$ shown in Figure~\ref{fig:bad}(a), 
consider the graph $G$ obtained from $G'$ by attaching a new vertex $v$ via the two new edges $va$ and $vb$.
Then, $G$ (the left one of Figure~\ref{fig:bad}(a)) is a minimally body-and-hinge rigid graph.
On the other hand, the graph obtained from $G$ by splitting off at $v$ is not minimally body-and-hinge rigid
while  just removing $v$ (without inserting the new edge $ab$) from $G$ produces a
minimally body-and-hinge rigid graph.
We call this operation as {\em the removal of $v$} to distinguish it from the splitting off at $v$.
More formally, we denote by $G_v$ the graph obtained by the removal of a vertex $v$ of degree two if one exists. 

\begin{figure}[t]
\centering
\begin{minipage}{0.35\textwidth}
\centering
\includegraphics[width=0.75\textwidth]{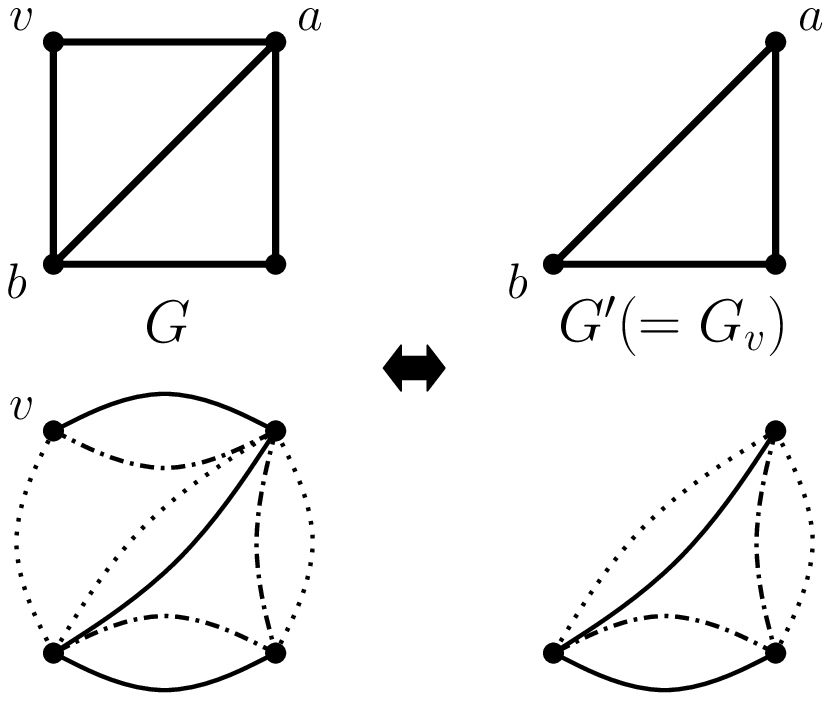}
\par
(a)
\end{minipage}\ 
\begin{minipage}{0.4\textwidth}
\centering
\includegraphics[width=0.95\textwidth]{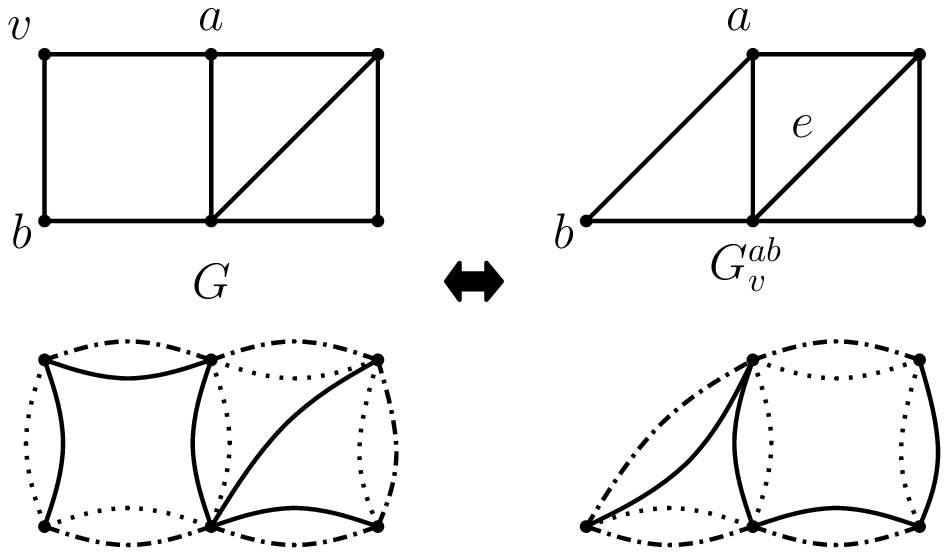}
\par
(b)
\end{minipage}
\caption{(a)An example of a minimal $0$-dof-graph $G$ such that $G_v^{ab}$ is not a minimal $0$-dof-graph for $d=2$ and $D=3$.
Notice that $G'=G_v$ is a $0$-dof-graph and hence $G_v^{ab}$ is not minimal. 
(b) An example of a minimal $0$-dof-graph $G$ 
such that $G_v^{ab}$ is not minimal and also $G_v$ is not a $0$-dof-graph
for $d=2$ and $D=3$.
Notice that there exist three edge-disjoint spanning trees in $\widetilde{G_v^{ab}}$ 
that contain no edge of $\widetilde{e}$.
%This implies that $G_v^{ab}$ is not a minimally body-and-hinge rigid graph.
}
\label{fig:bad}
\end{figure}

\begin{lemma}
\label{lemma:removal}
Let $G=(V,E)$ be a $k$-dof-graph in which there exists a vertex $v$ of degree $2$ with $N_G(v)=\{a,b\}$.
%Let $G_v$ be the graph obtained from $G$ by the removal of $v$.
Then, $\de(\widetilde{G_v})\geq k$ holds.
Moreover, if $\de(\widetilde{G_v})=k$,
then there exists a base $B$ of ${\cal M}(\widetilde{G})$ satisfying
$|\widetilde{vb}\cap B|=1$.
\end{lemma}
\begin{proof}
Consider a base $B'$ of ${\cal M}(\widetilde{G_v})$.  
Since $\widetilde{G_v}$ is a subgraph of $\widetilde{G_v^{ab}}$, $B'$ is an independent set of ${\cal M}(\widetilde{G_v^{ab}})$.
Also, since $G$ can be obtained from $G_v^{ab}$ by the edge-splitting along $ab$, 
Lemma~\ref{lemma:splitting}(i) can be applied to $B'$ with $h'=0$ 
to derive that there exists an independent set $I$ of ${\cal M}(\widetilde{G})$ 
satisfying $|I|=|B'|+D$ and $|\widetilde{vb}\cap I|=1$.
This implies that $|I|=|B'|+D=D(|V|-1)-\de(\widetilde{G_v})$ by $|B'|=D(|V|-2)-\de(\widetilde{G_v})$.
Thus,  $k=\de(\widetilde{G})\leq \de(\widetilde{G_v})$ holds by (\ref{eq:pre}).

If $k=\de(\widetilde{G})=\de(\widetilde{G_v})$ holds,
then $I$ is a base of ${\cal M}(\widetilde{G})$ and thus a desired base exists.
\end{proof} 

We remark that there is a situation in which
$G_v$ is not a $k$-dof-graph and also $G_v^{ab}$ is not a minimal $k$-dof-graph.
Figure~\ref{fig:bad}(b) shows such an example.
%We shall investigate this type of graph in the next subsection. 

\paragraph{Remarks.}
The concept of the splitting off is originated by Lov{\'a}sz~\cite{lovasz:1979},
where he proved that the splitting off at a vertex $v$ does not decrease the (local) connectivity between two vertices except for $v$. 
Also, he proved that a graph is $2k$-edge-connected if and only if it can be constructed from a single vertex by 
a sequence of two operations keeping the $2k$-edge-connectivity: one is an edge addition and 
the other is the inverse of a splitting off operation.
Characterizing graphs having some specific property in terms of an inductive construction is 
one of popular topics.
In particular  it is known that 
a graph has $k$ edge-disjoint spanning trees if and only if it can be constructed from a single vertex by 
a sequence of three simple operations keeping the property~\cite{Nash:1961}.
We should remark that this result cannot directly apply to $q$-strong graphs for a rational $q$ which is the case for our problem
and there is no result concerning a construction of $q$-strong graphs to the best of our knowledge.

\subsection{Minimal $k$-dof-graphs having no proper rigid subgraph}
\label{subsec:rigid_subgraph}
As shown in Figure~\ref{fig:bad}, 
a splitting off does not preserve the minimality in general.
However, if we concentrate on a graph which has no proper rigid subgraph, 
it can be shown that a splitting off preserves the minimality.
We hence concentrate on graphs having no proper rigid subgraph throughout this subsection.
Let us first show properties of such graphs in Lemmas~\ref{lemma:cardinality} and~\ref{lemma:degree2}.
\begin{lemma}
\label{lemma:cardinality}
Let $G=(V,E)$ be a minimal $k$-dof-graph which contains no proper rigid subgraph.
\begin{description}
\item[(i)] If $k=0$, then $(D-1)|E|<D(|V|-1)+D-1$ holds.
\item[(ii)] If $k>0$, then $\widetilde{E}$ is the base of ${\cal M}(\widetilde{G})$ and hence $(D-1)|E|=D(|V|-1)-k$.
\end{description}
\end{lemma}
\begin{proof}
(i) Let us consider the case of $k=0$.
Let $e$ be an arbitrary edge of $E$
and let $h^*$ be the minimum value of $|\widetilde{e}\cap B|$ taken over all bases $B$ of ${\cal M}(\widetilde{G})$.
Also let $B^*$ be a base of ${\cal M}(\widetilde{G})$ satisfying $|\widetilde{e}\cap B^*|=h^*$.
Notice $h^*\geq 1$ due to the minimality of $G$.
We shall show the following fact:
\begin{equation}
\label{eq:cardinality1}
\widetilde{E}\setminus \widetilde{e}\subset B^*.
\end{equation}
Suppose, for a contradiction, that an edge $f_i\in \widetilde{E}\setminus \widetilde{e}$ is not contained in $B^*$.
We consider the fundamental circuit $X$ within $B^*+f_i$. 
Then,  $G[V(X)]$ is a rigid subgraph by Lemma~\ref{lemma:circuit_rigidity}.
Since there exists no proper rigid subgraph in $G$,
$V(X)=V$ must hold.
Moreover, by Lemma~\ref{lemma:circuit_rigidity}, $X\cap \widetilde{e}\neq \emptyset$ also holds since otherwise
there exist $D$ edge-disjoint spanning trees on $V$  which contain no edge of $\widetilde{e}$, 
contradicting the minimality of $G$.
Therefore, there exists the base $B=B^*+f_i-e_j$ of ${\cal M}(\widetilde{G})$ satisfying $|B\cap \widetilde{e}|<|B^*\cap \widetilde{e}|=h^*$,
where $e_j\in X\cap \widetilde{e}$.
This contradicts the minimality of $h^*$ and hence (\ref{eq:cardinality1}) follows.

By (\ref{eq:cardinality1}) and $|B^*|=D(|V|-1)$, we obtain that the total number of edges in $\widetilde{G}$ is equal to
$|B^*|+(|\widetilde{e}|-h^*)=D(|V|-1)+(D-1-h^*)$, which is less than $D(|V|-1)+D-1$ by $h^*\geq 1$.

(ii) Let us consider the case of $k>0$.
We shall show the following fact, which is analogous to (\ref{eq:cardinality1}):
\begin{equation}
\label{eq:cardinality2} 
\text{$\widetilde{E}$ is independent in ${\cal M}(\widetilde{G})$.}
\end{equation}
Suppose for a contradiction that $\widetilde{E}$ is dependent.
Then, there exists a base $B$ of ${\cal M}(\widetilde{G})$ that does not contain an edge $f_i\in \widetilde{E}$.
Consider again the fundamental circuit $X$ within $B+f_i$. 
By Lemma~\ref{lemma:circuit_rigidity}, $G[V(X)]$ is a rigid subgraph.
Since $G$ contains no proper rigid subgraph from the lemma assumption, 
$V(X)=V$ must hold.
Moreover, $V(X)=V$ implies that $G$ contains $D$ edge-disjoint spanning trees on $V$, which consist of edges of $X$.
This in turn implies $k=0$, contracting $k>0$. Thus, (\ref{eq:cardinality2}) follows.
%Recall that the cardinality of a base of ${\cal M}(\widetilde{G})$ is equal to $D(|V|-1)-k$.
%Therefore, since $\widetilde{E}$ is the base, (ii) follows.
\end{proof}

The following lemma claims the existence of small degree vertices.
\begin{lemma}
\label{lemma:degree2}
Let $G=(V,E)$ be a 2-edge-connected minimal $k$-dof-graph which contains no proper rigid subgraph.
Then, either $G$ is a cycle of at most $d$ vertices or 
it contains a chain $v_0v_1\dots v_{d}$ of length $d$ 
such that $v_iv_{i+1}\in E$ for $0\leq i\leq d-1$ and $d_G(v_i)=2$ for $1\leq i\leq d-1$.
\end{lemma}
\begin{proof}
Let us denote the average degree of the vertices of $G$ by $d_{\rm avg}$.
Lemma~\ref{lemma:cardinality} implies $(D-1)|E|<D(|V|-1)+D-1<D|V|$.
Hence, we have 
\begin{equation}
\label{eq:avg1}
d_{\rm avg}=\frac{2|E|}{|V|}<\frac{2D}{D-1}=2+\frac{2}{D-1}\leq 3,
\end{equation}
where the last inequality follows from $D\geq 3$.
This implies that $G$ has a vertex of degree two.

If $G$ is a cycle, then the statement clearly holds.
(If it consists of more than $d$ vertices, then the latter property holds.)
Hence let us consider the case where $G$ contains a vertex of degree more than two. 
For a nonnegative integer $i$, let $X_i=\{v\in V\mid d_G(v)=i\}$.
Note that, since $G$ is 2-edge-connected, $X_0=\emptyset$ and $X_1=\emptyset$ hold.
We say that a chain $u_0u_1\dots u_j$ is maximal if $d_G(u_0)>2$, $d_G(u_j)>2$ and $d_G(u_i)=2$ for all $1\leq i\leq j-1$.
Let ${\cal C}$ be the collection of all maximal chains in $G$.
Note that ${\cal C}$ is nonempty (because $G$ is not a cycle), and each vertex of degree two belongs to exactly one maximal chain.
Suppose, for a contradiction, that the length of each maximal chain is at most $d-1$.
Then, each maximal chain contains at most $d-2$ vertices of degree two
and hence we have
\begin{equation}
\label{eq:chain1}
|X_2|\leq (d-2)|{\cal C}|.
\end{equation}

For a maximal chain $u_1u_2\dots u_j$, we call the edges $u_1u_2$ and $u_{j-1}u_j$ the end edges of the chain.
Then the set of all end edges of the maximal chains in ${\cal C}$ is a subset of the edges incident to the vertices of $\bigcup_{i\geq 3}X_i$.
Hence we have 
\begin{equation}
\label{eq:chain2}
2|{\cal C}|\leq \sum_{i\geq 3}i|X_i|.
\end{equation}
Combining (\ref{eq:chain1}) and (\ref{eq:chain2}), we obtain
\begin{equation*}
\label{eq:chain3}
2|X_2|\leq \sum_{i\geq 3}i(d-2)|X_i|.
\end{equation*}
Summing up this inequality and (the twice of) $|V|=\sum_{i\geq 2}|X_i|$,
we further obtain
\begin{equation}
\label{eq:chain4}
\sum_{i\geq 3}(i(d-2)+2)|X_i|\geq 2|V|.
\end{equation}
It is not difficult to see that the following inequality holds.
\begin{equation*}
\label{eq:chain5}
(D-1)(i-2)\geq i(d-2)+2 \qquad \text{for all } i\geq 3.
\end{equation*}
Hence, by (\ref{eq:chain4}), we have
\begin{equation}
\label{eq:chain6}
\sum_{i\geq 3}(D-1)(i-2)|X_i|\geq 2|V|.
\end{equation}
As a result,
\begin{align*}
d_{\rm avg}&=\frac{\sum_{i\geq 2}i|X_i|}{|V|} \\
&=2+\frac{\sum_{i\geq 3}(i-2)|X_i|}{|V|} && (\text{by } 2|V|=\sum_{i\geq 2}2|X_i|) \\
&=2+\frac{\sum_{i\geq 3}(D-1)(i-2)|X_i|}{(D-1)|V|} \\
&\geq 2+\frac{2}{D-1} && (\text{by (\ref{eq:chain6})}) \\
&>d_{\rm avg}, && (\text{by (\ref{eq:avg1})})
\end{align*}
which is a contradiction.
\end{proof}

Let us start to investigate the deficiencies of graphs obtained by the operations defined in Subsection~\ref{subsec:splitting_off},
assuming that $G$ contains no proper rigid subgraph.
\begin{lemma}
\label{lemma:existence_subgraph1}
Let $G=(V,E)$ be a minimal $k$-dof-graph with $|V|\geq 3$ 
which contains no proper rigid subgraph.
Let $v$ be a vertex of degree two.
Then, $\de(\widetilde{G_v})>k$.
\end{lemma}
\begin{proof}
Note that $G_v$ is a proper subgraph of $G$.
Since there exists no proper rigid subgraph in $G$, 
$G_v$ is not a $0$-dof-graph. 
This proves the statement for $k=0$.

When $k>0$, $\widetilde{E}$ is the base of ${\cal M}(\widetilde{G})$ from Lemma~\ref{lemma:cardinality}(ii).
Since  $\widetilde{E}$ is the unique base of ${\cal M}(\widetilde{G})$ and $|\widetilde{vb}\cap \widetilde{E}|\neq 1$ holds,  
Lemma~\ref{lemma:removal} implies ${\rm def}(\widetilde{G_v})> k$.
\end{proof}

\begin{lemma}
\label{lemma:operation}
Let $G=(V,E)$ be a minimal $k$-dof-graph which contains no proper rigid subgraph.
Then, for any vertex $v$ of degree two with $N_G(v)=\{a,b\}$,
the followings hold:
\begin{description}
\item[(i)] If $k=0$, then $G_v^{ab}$ is a minimal $0$-dof-graph.
\item[(ii)] If $k>0$, then $G_v^{ab}$ is a minimal $(k-1)$-dof-graph.
\end{description}
\end{lemma}
\begin{proof}
We remark that $G_v$ is not a $k$-dof-graph by Lemma~\ref{lemma:existence_subgraph1}.
Also, by Lemma~\ref{lemma:splitting_off}, $G_v^{ab}$ is either a $k$-dof-graph or a minimal $(k-1)$-dof-graph. 

Let us show (i). 
In the case of $k=0$, $G_v^{ab}$ is clearly a $0$-dof-graph and thus we only need to show the minimality of $G_v^{ab}$.
To see this, we claim the following:
\begin{equation}
\label{eq:existence_subgraph}
\text{For any circuit $X$ of the matroid ${\cal M}(\widetilde{G_v^{ab}})$, $X\cap \widetilde{ab}\neq \emptyset$ holds.}
\end{equation}
Suppose $X\cap \widetilde{ab}=\emptyset$.
Then, $X$ is a subset of $\widetilde{E}$. Hence, $X$ is a circuit of ${\cal M}(\widetilde{G})$ with $v\notin V(X)$.
Lemma~\ref{lemma:circuit_rigidity} says that 
$G[V(X)]$ is a proper rigid subgraph of $G$,
which contradicts the lemma assumption. 
Thus (\ref{eq:existence_subgraph}) follows.

In order to show (i), suppose that $G_v^{ab}$ is  not a minimal $0$-dof-graph.
Then, there exists an edge $e$ such that $G_v^{ab}-e$ is still a $0$-dof-graph. 
(Note $e\neq ab$ since $G_v$ is not a $0$-dof-graph by Lemma~\ref{lemma:existence_subgraph1}.)
Also, there exists a base $B_1$ of ${\cal M}(\widetilde{G_v^{ab}})$ with $B_1\cap \widetilde{e}=\emptyset$.
Let $h=|B_1\cap \widetilde{ab}|$, 
and let us denote the edges of $B_1\cap \widetilde{ab}$ by $(ab)_1,\dots,(ab)_h$.
Also we denote the edges of $\widetilde{e}$ by $e_1,\dots,e_{D-1}$ as usual.
We repeatedly perform the following process for $i=1,\dots, h$: 
Insert an edge $e_i\in \widetilde{e}$ into $B_i$ and 
let $X_i$ be the fundamental circuit of $B_i+e_i$ (in ${\cal M}(\widetilde{G_v^{ab}})$).
By (\ref{eq:existence_subgraph}), 
$X_i\cap \widetilde{ab}\neq \emptyset$ holds and hence 
we can obtain a new base $B_{i+1}$ by removing $(ab)_{i}\in X_i\cap \widetilde{ab}$.
(Namely, we obtain the base $B_{i+1}=B_i+e_i-(ab)_i$ of ${\cal M}(\widetilde{G_v^{ab}})$.)
Repeating this process, 
we eventually obtain the base $B_{h+1}$ of ${\cal M}(\widetilde{G_v^{ab}})$, 
which contains no edge of $\widetilde{ab}$.
Note that $B_{h+1}$ is a base of ${\cal M}(\widetilde{G_v})$  as well as a base of ${\cal M}(\widetilde{G_v^{ab}})$
with the cardinality $|B_{h+1}|=|B_1|=D(|V\setminus\{v\}|-1)$.
Therefore, $\de(\widetilde{G_v})=0$ holds by (\ref{eq:pre}).
This contradicts that $G_v$ is not a $0$-dof-graph (by Lemma~\ref{lemma:existence_subgraph1}) and thus (i) follows.

Next let us show (ii).
If $G_v^{ab}$ is not a $k$-dof-graph,
then the statement follows because $G_v^{ab}$ is either a $k$-dof-graph or a minimal $(k-1)$-dof-graph by Lemma~\ref{lemma:splitting_off}.
Suppose, for a contradiction, that $G_v^{ab}$ is a $k$-dof-graph.
By Lemma~\ref{lemma:splitting_off} there exists a base $B'$ of ${\cal M}(\widetilde{G_v^{ab}})$ satisfying $|\widetilde{ab}\cap B'|<D-1$.
Without loss of generality, we assume $(ab)_1\notin B'$.
Consider the fundamental circuit $Y$ of $B'+(ab)_1$ and let $G'=G[V(Y)]$.
Then, by Lemma~\ref{lemma:circuit_rigidity}, $G'$ is a $0$-dof-graph on $V(Y)$.
Since $G_v^{ab}$ is a $k$-dof-graph with $k>0$, 
$G'$ must be a proper subgraph of $G_v^{ab}$, i.e., $V(Y)$ is a proper subset of $V\setminus\{v\}$.
Let $I=Y-(ab)_1$.
Then, Lemma~\ref{lemma:circuit_rigidity} also says that $I$ can be partitioned into $D$ edge-disjoint spanning trees on $V(Y)$.
Hence, $I$ is an independent set of ${\cal M}(\widetilde{G'})$ with $|I\cap \widetilde{ab}|<D-1$ 
due to $(ab)_1\notin I$.
We apply the edge-splitting operation in $G'$ along $ab$. 
Lemma~\ref{lemma:splitting}(i) implies that the resulting graph contains an independent set 
with the cardinality $|I|+D$, which is equal to $D(|V(Y)|-1)+D=D(|V(Y)\cup\{v\}|-1)$.
Hence the resulting graph is a $0$-dof-graph.
Moreover, this $0$-dof-graph is a proper subgraph of $G$ since $V(Y)$ is a proper subset of $V\setminus\{v\}$.
Therefore, $G$ contains a proper rigid subgraph, which contradicts the lemma assumption.
\end{proof}

\subsection{Inductive constructions}
%\label{subsec:inductive}
%%Let $G'$ be a subgraph of a multigraph $G$.
%%Then, we shall denote by $G/G'$ the graph obtained from $G$ by contracting all the edges of $$
%Combining the lemmas obtained so far results in the following lemma.
%\begin{lemma}
%\label{lemma:operation}
%Let $G$ be a 2-edge-connected minimal $k$-dof-graph with $|V|\geq 2$.
%If there exists no proper rigid subgraph in $G$,
%there exists a vertex of degree two.
%Moreover, for any vertex $v$ of degree two,  
%$G_v^{ab}$ is either a minimal $k$-dof-graph or a minimal $(k-1)$-graph, where $N_G(v)=\{a,b\}\subset V$.
%\end{lemma}
%\begin{proof}
%By Lemma~\ref{lemma:degree2},
%there exists a vertex $v$ of degree two in $G$
%since $G$ contains no proper rigid subgraph from the lemma assumption.
%By Lemma~\ref{lemma:splitting_off}, $G_v^{ab}$ is either a $k$-dof-graph or a minimal $(k-1)$-graph.
%Suppose $G_v^{ab}$ is a $k$-dof-graph but not a minimal $k$-dof-graph.
%Then, by Lemma~\ref{lemma:existence_subgraph2},
%$G$ contains a proper rigid subgraph, which is a contradiction.
%The statement thus follows.
%\end{proof}

%By Lemma~\ref{lemma:2connectivity},
%any minimally body-and-hinge rigid graph is a 2-edge-connected $0$-dof-graph.
Combining the results obtained so far, it is not difficult to prove the following construction of minimally body-and-hinge 
rigid graphs.
\begin{theorem}
\label{theorem:inductive}
Let $G$ be a minimally body-and-hinge rigid graph with $|V|\geq 2$.
Then, there exists a sequence $G=G_1,G_2,\dots,G_m$ of minimally body-and-hinge rigid graphs such that
\begin{itemize}
\item $G_m$ is a graph consisting of two vertices $\{u,v\}$ 
and two parallel edges connecting $u$ and $v$, and
\item $G_{i+1}$ is obtained from $G_i$ by either 
the splitting off at a vertex of degree 2  or 
the contraction of a proper rigid subgraph for 
each $i=1,\dots,m-1$.
\end{itemize}
\end{theorem}
\begin{proof}
By Lemma~\ref{lemma:2connectivity},
any minimally body-and-hinge rigid graph is a 2-edge-connected $0$-dof-graph.
Hence, if $G$ contains no proper rigid subgraph, then $G$ has a vertex of degree two by Lemma~\ref{lemma:degree2}.
Combining this fact with Lemma~\ref{lemma:operation}, 
either (i) $G$ contains a proper rigid subgraph $G'=(V',E')$ or 
(ii) there exists a vertex $v$ of degree $2$ with $N_G(v)=\{a,b\}$ such that $G_v^{ab}$ is 
a minimal $0$-dof-graph, that is, a minimally body-and-hinge rigid graph.
Recall that, if (i) holds, then the graph obtained by the contraction of $E'$ is again 
a minimally body-and-hinge rigid graph by Lemma~\ref{lemma:contraction}.
Hence, in either case, we can obtain a new minimally body-and-hinge rigid graph $G_2=(V_2,E_2)$
with $2\leq |V_2|<|V|$.
By inductively repeating this process, we eventually obtain a minimally body-and-hinge rigid graph $G_m$
which consists of two vertices and two edges between them.
This inductive process shows a desired sequence of minimally body-and-hinge rigid graphs.
\end{proof}

\section{Infinitesimally Rigid Panel-and-hinge Realizations}
\label{sec:molecular}

We shall use the following notations to indicate submatrices of the rigidity matrix $R(G,\bp)$.
Recall the notations given in Section~\ref{subsec:rigidity_matrix}: 
$\{r_1(\bp(e)),\dots,r_{D-1}(\bp(e))\}$ denotes a basis
of the orthogonal complement of the vector space spanned by a $(d-1)$-extensor $C(\bp(e))$.
Also, $r(\bp(e))$ denotes the $(D-1)\times D$-matrix whose $i$-th row vector is $r_i(\bp(e))$,
and hence $\rank r(\bp(e))=D-1$ holds.

In $R(G,\bp)$,
consecutive $D-1$ rows are associated with an edge $e\in E$ 
and consecutive $D$ columns are associated with a vertex $v\in V$.
More precisely, the consecutive $D-1$ rows associated with $e=uv\in E$ 
are described by the $(D-1)\times D|V|$ submatrix:
\begin{equation}
\label{eq:R[e]}
(\Bvector{\cdots}{{\cdots}} \ \Bvector{\cdots}{{\bf 0}} \ \Bvector{\cdots}{{\cdots}} \ \ 
\Bvector{u}{r(\bp(e))}\ \ \Bvector{\cdots}{{\cdots}} \ \ \Bvector{\cdots}{{\bf 0}} \ \ \Bvector{\cdots}{{\cdots}}\ \ 
\Bvector{v}{-r(\bp(e))}\ \ \Bvector{\cdots}{{\cdots}} \ \ \Bvector{\cdots}{{\bf 0}}\ \ \Bvector{\cdots}{{\cdots}}),
\end{equation}
Let us denote  the $(D-1)\times D|V|$-submatrix  given in (\ref{eq:R[e]}) by $R(G,\bp;e)$ for each $e\in E$.
We remark $\rank R(G,\bp;e)=D-1$ since $\rank r(\bp(e))=D-1$.
Also, we consider the one-to-one correspondence between $e_i\in \widetilde{e}$ and the $i$-th row of $R(G,\bp;e)$, 
which is denoted by $R(G,\bp;e_i)$.
Namely, for $e=uv\in E$ and $1\leq i\leq D-1$, it is a $D|V|$-dimensional vector described as
\begin{equation*}
R(G,\bp;e_i)=(\Bvector{\cdots}{{\cdots}} \ \Bvector{\cdots}{0} \ \Bvector{\cdots}{{\cdots}} \ \ 
\Bvector{u}{r_i(\bp(e))}\ \ \Bvector{\cdots}{{\cdots}} \ \ \Bvector{\cdots}{0} \ \ \Bvector{\cdots}{{\cdots}}\ \ 
\Bvector{v}{-r_i(\bp(e))}\ \ \Bvector{\cdots}{{\cdots}} \ \ \Bvector{\cdots}{0}\ \ \Bvector{\cdots}{{\cdots}}).
\end{equation*}
We assume that any vector in $\mathbb{R}^{D|V|}$ is regarded as a composition of $|V|$ vectors in $\mathbb{R}^D$ each of which 
is associated with a vertex $v\in V$ in a natural way.

Similarly, let us denote by $R(G,\bp;v)$ the $(D-1)|E|\times D$-submatrix of $R(G,\bp)$
induced by the consecutive $D$ columns associated with $v$.
For $F\subseteq E$ and $X\subseteq V$, 
$R(G,\bp;F,X)$ denotes the submatrix of $R(G,\bp)$ induced by the rows of
$R(G,\bp;e)$ for $e\in F$ and the columns of $R(G,\bp;v)$ for $v\in X$.

We need the following technical lemma. 
\begin{lemma}
\label{lemma:vertex_deletion}
Let $(G,\bp)$ be a body-and-hinge framework in $\mathbb{R}^d$ 
for a multigraph $G=(V,E)$.
Then, for any vertex $v\in V$, 
 $\rank R(G,\bp;E,V\setminus\{v\}) = \rank R(G,\bp)$ holds,
i.e., the rank of the rigidity matrix is invariant under the removal of the consecutive $D$ columns associated with $v$.
\end{lemma}
\begin{proof}
For $1\leq i\leq D$, let $b_i$ be the vector in $\mathbb{R}^{D|V|}$ such that 
the $i$-th coordinate of the consecutive $D$ coordinates associated with
$v$ is equal to $1$ and the other entries are all $0$.
Let $R'$ be the matrix obtained from $R(G,\bp)$ by adding $b_i$ as new rows for all $i$.
Then, appropriate fundamental row operations changes $R'$ to the following form:
\begin{equation*}
\begin{CD}
R'=\begin{array}{@{\,}rc|c@{\,}l}

 & \multicolumn{1}{c}{\scriptstyle v}
 & \multicolumn{1}{c}{\scriptstyle V\setminus \{v\}}
 &
 \\
 \ldelim({2}{5pt}[] & \ \ I \ \ &  {\bf 0}\         & \ldelim){2}{5pt}[] \\ \cline{2-3}
                    & \multicolumn{2}{c}{R(G,\bp)} & \\
\end{array}
@> \text{\tiny Row operations} >> 
\begin{array}{@{\,}rc|c@{\,}l}

 & \multicolumn{1}{c}{\scriptstyle v}
 & \multicolumn{1}{c}{\scriptstyle V\setminus \{v\}}
 &
 \\
 \ldelim({2}{5pt}[] & \ \ I \ \ &  {\bf 0}\         & \ldelim){2}{5pt}[] \\ \cline{2-3}
                    & {\bf 0} & R(G,\bp;E,V\setminus\{v\}) & \\
\end{array}
\end{CD}
\end{equation*}
where $I$ denotes the $D\times D$ identity matrix.
This implies $\rank R'=\rank R(G,\bp;E,V\setminus\{v\})+D$.
Hence, the statement is true if and only if $\rank R'=\rank R(G,\bp)+D$ holds
or, equivalently, no vector spanned by $\{b_1,b_2,\dots,b_D\}$  
is contained in the row space of $R(G,\bp)$.
Suppose, for a contradiction, that 
a nonzero vector $b'$ spanned by $\{b_1,b_2\dots,b_D\}$ is contained in the row space of $R(G,\bp)$.
Let us denote it by $b'=\lambda_1b_1+\dots+\lambda_D b_D$ with $\lambda_i\in\mathbb{R}, 1\leq i\leq D$
and, without loss of generality, we assume $\lambda_1\neq 0$.
Recall that, in Section~\ref{subsec:rigidity_matrix}, 
$S_1^*$ was defined as the vector in ${\mathbb R}^{D|V|}$ 
whose first coordinate of the consecutive $D$ coordinates associated with each vertex is equal to $1$ 
and the other entries are all $0$.
Also, recall that $S_1^*$ is in the orthogonal complement of the row space of $R(G,\bp)$.
However, we have $b'\cdot S_1^*=\lambda_1\neq 0$, which contradicts that 
$b'$ is contained in the row space of $R(G,\bp)$.
\end{proof}

\subsection{Generic Nonparallel Panel-and-hinge Realizations}
\label{subsec:genericity}
Before providing a proof of the Molecular Conjecture,
we need to mention the {\em  generic} property of  panel-and-hinge realizations for a simple graph introduced by Jackson and Jord{\'a}n~\cite{Jackson:07}.
For a panel-and-hinge realization $(G,\bp)$, 
let $\Pi_{G,\bp}(v)$ denote the panel associated with $v\in V$, that is, a $(d-1)$-affine subspace 
containing all of the hinges $\bp(e)$ of the edges $e$ incident to $v$.
For a simple graph $G$ (i.e.,~no parallel edges exist in $G$), $(G,\bp)$ is called a {\em nonparallel} panel-and-hinge realization 
if $\Pi_{G,\bp}(u)$ and $\Pi_{G,\bp}(v)$ are not parallel for any distinct $u,v\in V$;
$\Pi_{G,\bp}(u)$ and $\Pi_{G,\bp}(v)$ are said to be nonparallel if $\Pi_{G,\bp}(u)\cap \Pi_{G,\bp}(v)$ is a $(d-2)$-affine subspace.
As Jackson and Jord{\'a}n mentioned in \cite[Section~7]{Jackson:07}, 
each entry of the rigidity matrix $R(G,\bp)$ of a nonparallel panel-and-hinge realization $(G,\bp)$
can be described in terms of the coefficients appearing in the equations expressing $\Pi_{G,\bp}(v)$ for $v\in V$,
and hence each minor of the rigidity matrix is a polynomial of these coefficients.
The rigidity of nonparallel panel-and-hinge realizations thus has a generic property.

Let us look into the details.
For a simple graph $G=(V,E)$,  consider a mapping $\bc:V\rightarrow \mathbb{R}^d$ such that 
$\bc(u)$ and $\bc(v)$ are linearly independent for each $u,v\in V$ with  $u\neq v$.
Then, the $(d-1)$-affine subspace (i.e.~panel) associated with $v\in V$ with respect to $\bc$ is defined as 
$\Pi(v)=\{\bx\in \mathbb{R}^d\mid \bx\cdot \bc(v)=1\}$.
Since $\bc(u)$ and $\bc(v)$ are linearly independent, 
$\Pi(u)\cap \Pi(v)$ is a $(d-2)$-affine space.
Hence, the mapping $\bc$ induces a mapping $\bp$ on $E$, that is, $\bp(uv)=\Pi(u)\cap \Pi(v)$ for each $uv\in E$,
and $(G,\bp)$ is a nonparallel panel-and-hinge framework of $G$.
Conversely, given a nonparallel $(G,\bp)$, $\bp$ induces the unique mapping $\bc:V\rightarrow \mathbb{R}^d$ 
such that $\Pi_{G,\bp}(v)=\{\bx\in\mathbb{R}^d\mid \bx\cdot\bc(v)=1\}$ for each $v\in V$
(provided that no panel passes through the origin and no vertex of degree one exists).
%Therefore, there is a one-to-one correspondence between $\bp$ and $\bc$.

We say that $(G,\bp)$ is {\em generic} if 
the set of coordinates of $\bc(v)$ for all $v\in V$ is algebraically independent 
over the rational field.\footnote{The definition in terms of algebraic independence 
produces a smaller class of frameworks than that of conventional generic frameworks (in terms of the maximality of rigidity matrices).
We just use this definition to make our proof simpler.}
We note that almost all nonparallel panel-and-hinge realizations are generic.
Since each entry of $R(G,\bp)$ is a polynomial of the coordinates of $\bc(v)$ for $v\in V$,
the rank of $R(G,\bp)$ takes the maximum value over all nonparallel panel-and-hinge realizations of $G$ if $(G,\bp)$ is generic.

It is known that, even though $(G,\bp)$ has some parallel panels, 
we can perturb them so that the resulting realization becomes nonparallel without decreasing the rank of the rigidity matrix
(see ~\cite{Jackson:08} or \cite[Lemma~7.1]{Jackson:07}). 
The following lemma states a special case of this result, but let us provide a proof for the completeness.
\begin{lemma}
\label{lemma:perturbation}
Let $G$ be a simple graph and 
$(G,\bp)$ be a panel-and-hinge realization of $G$.
Suppose there exists a pair $(a,b)\in V\times V$ with $a\neq b$ 
satisfying $\Pi_{G,\bp}(a)=\Pi_{G,\bp}(b)$ and 
that $\Pi_{G,\bp}(u)$ and $\Pi_{G,\bp}(v)$ are nonparallel for every pair
$(u,v)\in V\times V$ with $u\neq v$ except for $(a,b)$.
Then, there is a nonparallel panel-and-hinge realization $(G,\bp')$ satisfying $\rank R(G,\bp')\geq \rank R(G,\bp)$.
\end{lemma}
\begin{proof}
We shall only consider the case of $ab\in E$. The case of $ab\notin E$ can be handled similarly.
Note that $ab$ is unique because $G$ is simple.

Since the rank of the rigidity matrix is invariant under an isometric transformation of the whole framework,
we can assume $\Pi_{G,\bp}(a)=\Pi_{G,\bp}(b)=\{x=(x_1,\dots,x_d)\in\mathbb{R}^d\mid x_d=0\}$ 
and $\bp(ab)=\{x\in\mathbb{R}^d\mid x_{d-1}=0, x_d=0\}$.
We shall rotate $\Pi_{G,\bp}(a)$ continuously around $\bp(ab)$.
To indicate the rotation, let us introduce a parameter $t\in\mathbb{R}$ and define 
$\Pi^t(a)=\{x\in\mathbb{R}^d\mid tx_{d-1}+x_d=0\}$.
Note that  $\Pi^0(a)=\Pi_{G,\bp}(a)$ and $\bp(ab)\subseteq \Pi^t(a)\cap \Pi_{G,\bp}(b)$ for any $t\in\mathbb{R}$.

Since $\Pi^0(a)$ and $\Pi_{G,\bp}(v)$ are nonparallel for any $v\in V\setminus \{a,b\}$ from the lemma assumption, 
there exists a small $\varepsilon >0$ such that 
$\Pi^t(a)$ and $\Pi_{G,\bp}(v)$ are nonparallel within $-\varepsilon < t< \varepsilon$.
%Hence, for any $v\in V\setminus\{a,b\}$, 
%$\Pi^t(a)\cap \Pi_{G,\bp}(v)$ is a $(d-2)$-affine space within $-\varepsilon < t< \varepsilon$.
Hence, the following mapping $\bp^t$ on $E$ is thus well defined within $-\varepsilon< t<\varepsilon$:
\begin{equation*}
\bp^t(e)=\begin{cases}
\bp(e) & \text{if } e\in E\setminus \delta_G(a)\cup \{ab\}, \\ 
\Pi^t(a)\cap \Pi_{G,\bp}(v) & \text{if } e=av\in \delta_G(a)\setminus \{ab\}.
\end{cases}
\end{equation*}
Notice that $\bp^0(e)=\bp(e)$ for all $e\in E$ and $(G,\bp^0)=(G,\bp)$ holds.
Also, $(G,\bp^t)$ is a nonparallel panel-and-hinge realization for any $0<t<\varepsilon$.
Since each $\bp^t(e), e\in E$ moves continuously with respect to $t$,
each minor of $R(G,\bp^t)$ can be described as a continuous function of $t$ within $-\varepsilon < t< \varepsilon$.
This implies that there exists a small $\varepsilon'$ with $0<\varepsilon'\leq \varepsilon$ such that 
$(G,\bp^t)$ is a nonparallel panel-and-hinge realization satisfying $\rank R(G,\bp^t)\geq \rank R(G,\bp^0)$ for any $0<t<\varepsilon'$.
\end{proof}

Let us introduce the concept of {\em nondegenerate frameworks} for multigraphs.
Let $G$ be a multigraph and let $(G,\bp)$ be a panel-and-hinge realization of $G$.
We say that $(G,\bp)$ is {\em nondegenerate} if $\Pi_{G,\bp}(u)\cap \Pi_{G,\bp}(v)\neq \emptyset$ for any $u,v\in V$,
i.e., either $\Pi_{G,\bp}(u)$ and $\Pi_{G,\bp}(v)$ are nonparallel or $\Pi_{G,\bp}(u)=\Pi_{G,\bp}(v)$.
Extending the discussions so far, it is not difficult to see the following fact.
\begin{lemma}
\label{lemma:perturbation2}
Let $G$ be a multigraph and $(G,\bp)$ be a panel-and-hinge realization of $G$.
Then, there exists a nondegenerate panel-and-hinge realization $(G,\bp')$ satisfying $\rank R(G,\bp')\geq \rank R(G,\bp)$.
\end{lemma}
\begin{proof} 
Consider the partition ${\cal V}$ of $V$ defined in such a way that $u$ and $v$ belong to the same subset of vertices if and only if
$\Pi_{G,\bp}(u)=\Pi_{G,\bp}(v)$ holds.
For $V_i\in {\cal V}$, we define the panel $\Pi_{G,\bp}(V_i)$ as $\Pi_{G,\bp}(V_i)=\Pi_{G,\bp}(v)$ for a vertex $v\in V_i$.
Also, for $V_i\in {\cal V}$, let $E(V_i)=\{uv\in E\mid u,v\in V_i\}$.
%We also say that, for an edge $uv\in E$, 
%$\bp(uv)$ {\em belongs to} $\Pi_{G,\bp}(V_i)$ if $u$ and $v$ belong to $V_i\in {\cal V}$.
%We can consider the mapping $\bc_0:{\cal V}\rightarrow \mathbb{R}^d$ 
%such that $\Pi_{G,\bp}(V_i)=\{x\in\mathbb{R}^d\mid x\dot\bc_0(V_i)\}$ for $V_i\in{\cal V}$.
%Similarly, for each $1\leq j\leq i$, we can consider the mapping $\bc_j$ from $E(V_j)$ 
%to a $(d-2)$-dimensional affine space\rightarrow \mathbb{R}^{d-1}$
%such that $\bp(uv)=\{x\in \mathbb{R}^d\mid \}$
Then, it is not difficult to see that each entry of the rigidity matrix $R(G,\bp)$ can be described in terms of 
the coefficients of the equations representing $\Pi_{G,\bp}(V_i)$ for $V_i\in {\cal V}$
and the coefficients of the equations representing $\bp(uv)$ of $uv\in E(V_i)$ for $V_i\in {\cal V}$.

Suppose $(G,\bp)$ is degenerate, i.e.,
$\Pi_{G,\bp}(u)$ and $\Pi_{G,\bp}(v)$ are parallel with $\Pi_{G,\bp}(v_i)\neq \Pi_{G,\bp}(v_j)$  
for some $u\in V_i\in {\cal V}$ and $v\in V_j\in {\cal V}$.
Then, from the definition of ${\cal V}$, $V_i\neq V_j$ holds.
We shall continuously rotate the panel $\Pi_{G,\bp}(V_i)$ and the hinges $\bp(e)$ of $e\in E(V_i)$
preserving their incidences.
Since any minor of $R(G,\bp)$ can be written  as a polynomial of the coefficients representing $\Pi_{G,\bp}(V_i)$ 
and the hinges $\bp(e)$ of $e\in E(V_i)$ (when fixing the other panels),
the rank of the rigidity matrix does not decrease if the continuous rotation is small enough.
Repeating this process for any pair of parallel panels,
we eventually obtain a desired nondegenerate framework.
\end{proof}

\subsection{Proof of the Molecular Conjecture} 
\label{subsec:proof}
We now start to show our main result.
We first claim the rigidity of graphs consisting of a small number of  vertices since it will be used several times 
(including the base case of the induction).
\begin{lemma}
\label{lemma:base}
Let $G=(V,E)$ be the graph consisting of two vertices $\{u,v\}$ and two parallel edges $\{e,f\}$ between $u$ and $v$.
Then, $G$ can be realized as an infinitesimally rigid panel-and-hinge framework $(G,\bp)$ such that $\Pi_{G,\bp}(u)=\Pi_{G,\bp}(v)$.
In particular, $\rank R(G,\bp;\{e,f\},v)=D$ holds if $\bp(e)\neq \bp(f)$.
\end{lemma}
\begin{proof}
Let $\langle C(\bp(e))\rangle$ and $\langle C(\bp(f))\rangle$ be the vector spaces spanned by $(d-1)$-extensors $C(\bp(e))$ and $C(\bp(f))$ 
associated with $\bp(e)$ and $\bp(f)$, respectively. 
As we mentioned, a $(d-1)$-extensor for a $(d-2)$-affine space is uniquely determined up to a scalar multiplication.
This implies that $\langle C(\bp(e))\rangle=\langle C(\bp(f))\rangle$ if and only if $\bp(e)=\bp(f)$.
Therefore, if $\bp(e)\neq \bp(f)$, the orthogonal complements of $\langle C(\bp(e))\rangle$ and $\langle C(\bp(f))\rangle$, that is,
the row spaces of $r(\bp(e))$ and $r(\bp(f))$ are distinct.
We hence have $\rank R(G,\bp;\{e,f\},v)=D$ by $\rank R(G,\bp;e,v)=\rank r(\bp(e))=D-1$.
Since we can realize $G$ as a framework $(G,\bp)$ such that $\Pi_{G,\bp}(u)=\Pi_{G,\bp}(v)$ and $\bp(e)\neq \bp(f)$, the statement follows. 
\end{proof}

If $G$ is a cycle, its realization can be easily analyzed
directly from the definition of infinitesimal motions (\ref{eq:hinge}).
The detailed calculation can be seen in \cite[Proposition~3.4]{crapo:1982} or  \cite[Proposition~3]{Whiteley:1999} for the 3-dimensional case and 
the technique can apply to the general dimensional case without any modification.
\begin{lemma}[\cite{crapo:1982,Whiteley:1999}]
\label{lemma:base2}
Let $G=(V,E)$ be a cycle with $3\leq |V|\leq D$.
Then, $G$ can be realized as an infinitesimally rigid nonparallel panel-and-hinge framework $(G,\bp)$.
\end{lemma}
%\begin{proof}
%Let us denote the vertices of $V$ as $v_1,v_2,\dots,v_{|V|}$ in the ordering of the cycle.
%For each $1\leq i\leq |V|$, let $\Pi(v_i)=\{{\bm x}=(x_1,x_2,\dots,x_d)\in\mathbb{R}^d\mid x_i=0\}$.
%Define a mapping $\bp$ on $E$ by $\bp(v_iv_{i+1})=\Pi(v_i)\cap \Pi(v_{i+1})$ for each $1\leq i\leq |V|$ (where $v_{|V|+1}=v_1$).
%Then, $(G,\bp)$ forms a simplex and it is not difficult to check 
%that $(G,\bp)$ is infinitesimally rigid.
%The detailed calculation can be seen in \cite{Whiteley:1999} for 3-dimension and the idea can apply to general dimension directly.
%\end{proof}

Let us claim the main theorem of this paper.
\begin{theorem}
\label{theorem:main}
Let $G=(V,E)$ be a minimal $k$-dof-graph with $|V|\geq 2$ for some nonnegative integer $k$. 
Then, there exists a (nonparallel, if $G$ is simple) panel-and-hinge realization $(G,\bp)$ in $\mathbb{R}^d$ 
satisfying $\rank R(G,\bp)=D(|V|-1)-k$.
\end{theorem}

Since the proof is quite long, let us first write up a corollary which follows from Theorem~\ref{theorem:main}.
The following theorem proves the Molecular Conjecture in a strong sense combined with Proposition~\ref{prop:preliminaries}.
\begin{theorem}
\label{theorem:main2}
Let $G=(V,E)$ be a multigraph.
Then, $G$ can be realized as a panel-and-hinge framework $(G,\bp)$ in $\mathbb{R}^d$ which satisfies $\rank R(G,\bp)=D(|V|-1)-\de(\widetilde{G})$.
\end{theorem}
\begin{proof}
Let $k=\de(\widetilde{G})$.
By Proposition~\ref{prop:preliminaries}, we have $\rank R(G,\bp)\leq D(|V|-1)-k$ for any realization $(G,\bp)$ of $G$.
When $G$ is not a minimal $k$-dof-graph, 
we can remove some edges from $G$ keeping the deficiency of $\widetilde{G}$ so that the resulting graph becomes a minimal $k$-dof-graph.
Let $G'=(V,E')$ be the obtained minimal $k$-dof-graph. (If $G$ is a minimal $k$-dof-graph, then let $G'=G$.)
By Theorem~\ref{theorem:main}, there is a panel-and-hinge realization $(G',\bq)$ satisfying $R(G',\bq)=D(|V|-1)-k$.
Moreover, by Lemma~\ref{lemma:perturbation2},
we may assume that $(G',\bq)$ is nondegenerate.
The definition of a nondegenerate framework says that, for any $u,v\in V$, $\Pi_{G',\bq}(u)\cap \Pi_{G',\bq}(v)$ is 
either a $(d-2)$-affine subspace or a $(d-1)$-affine subspace in $\mathbb{R}^d$.
Define a mapping $\bp$ on $E$ such that
$\bp(uv)=\bq(uv)$ for $uv\in E'$ and otherwise $\bp(uv)$ is a $(d-2)$-affine subspace contained in $\Pi_{G',\bq}(u)\cap \Pi_{G',\bq}(v)$.
It is obvious that $(G,\bp)$ is a panel-and-hinge realization of $G$ and moreover $\rank R(G,\bp)\geq \rank R(G',\bq)=D(|V|-1)-k$.
We thus obtain a panel-and-hinge realization satisfying $\rank R(G,\bp)=D(|V|-1)-k$.
\end{proof}

As we mentioned in Introduction, in $3$-dimensional space 
the projective dual of nonparallel panel-and-hinge frameworks are  ``hinge-concurrent'' body-and-hinge frameworks,
which are also called {\em molecular frameworks}~\cite{Whiteley:1999,Jackson:08} because they are used to study the flexibility of molecules.
Since the infinitesimal rigidity is invariant under the duality~\cite[Section~3.6]{crapo:1982} 
(see also \cite{Whiteley:1987} for more detailed descriptions on the related topic),
it follows from Theorem~\ref{theorem:main2} 
that a simple graph $G=(V,E)$ can be realized as a molecular framework $(G,\bp)$ which satisfies $\rank R(G,\bp)=D(|V|-1)-\de(\widetilde{G})$.

Let us consider the bar-and-joint rigidity in $3$-dimensional space.
For a graph $G$ of the minimum degree at least two,
Whiteley showed in \cite{Whiteley:2004} that 
$G$ can be realized as a rigid molecular framework if and only if $G^2$ can be realized as an infinitesimally rigid bar-and-joint framework.
Jackson and Jord{\'a}n  
proved in \cite{Jackson:05a,Jackson:08} that the Molecular Conjecture (Conjecture~\ref{conjecture}) is equivalent to the following statement:
\begin{corollary}
Let $G=(V,E)$ be a graph with minimum degree at least two. Then
$r(G^2)=3|V|-6-\de(\widetilde{G})$, where $r$ denotes the rank function of the 3-dimensional bar-and-joint rigidity matroid.
\end{corollary}
Further combinatorial results on 3-dimensional bar-and-joint frameworks of square graphs can be found in  \cite{Jackson:05a,Jackson:08}.

\section{Proof of Theorem~\ref{theorem:main}}
The proof is done by induction on $|V|$.
Let us consider the base case $|V|=2$.
Let $V=\{u,v\}$.
By Lemma~\ref{lemma:3connectivity}, 
any minimal $k$-dof-graph is not 3-edge-connected
and hence we have three possible cases;
(i) $E$ is empty,
(ii) $E$ consists of a single edge $e$ connecting $u$ and $v$, and
(iii) $E$ consists of two parallel edges $\{e,f\}$ between $u$ and $v$.
The cases (i) and (ii) are trivial since any realization satisfies the statement.
The case (iii) has been treated in Lemma~\ref{lemma:base}.
%If (i) holds, i.e.,~$G$ consists of only one edge $e$, then $\widetilde{G}$ contains $D$ edge-disjoint forests 
%whose total edge cardinality is equal to $D-1=D(|V|-1)-1$.
%Hence, by (\ref{eq:pre}), $G$ is a $1$-dof-graph.
%Since $\rank R(G,\bp)=\rank R_{G,\bp}[e]=D-1=D(|V|-1)-1$ for any mapping $\bp$, the claim follows with $k=1$.

Let us consider $G$ with $|V|\geq 3$.
We shall split the proof into three cases:
\begin{itemize}
\item Subsection~\ref{subsec:1connected} deals with the case where $G$ is not 2-edge-connected.
\item Subsection~\ref{subsec:case3-1} deals with the case where $G$ contains a proper rigid subgraph.
\item Subsection~\ref{subsec:case3-2} deals with the case where $G$ is 2-edge-connected and does not contain any proper rigid subgraph.
\end{itemize}
In each case, we will assume the following induction hypothesis on $|V|$:
\begin{equation}
\label{eq:hypothesis}
\begin{split}
&\text{For any minimal $k_H$-dof-graph $H=(V_H,E_H)$  for some nonnegative integer $k_H$ with}\\ 
&\text{$|V_H|<|V|$, there is a (nonparallel, if $G_H$ is simple) panel-and-hinge realization $(G_H,\bp_H)$} \\
&\text{that  satisfies $\rank R(G_H,\bp_H)=D(|V_H|-1)-k_H$.}
\end{split}
\end{equation}

\subsection{The case where $G$ is not 2-edge-connected}
\label{subsec:1connected}
%The following Lemma~\ref{lemma:disconnected} and Lemma~\ref{lemma:1connected} consider 
%the cases where $G$ is disconnected and 1-edge-connected, respectively.
This case can be handled rather easily
but present a basic strategy of the subsequent arguments. 

\begin{lemma}
\label{lemma:1connected}
Let $G=(V,E)$ be a minimal $k$-dof-graph which is not 2-edge-connected.
Suppose that (\ref{eq:hypothesis}) holds. 
Then, there is a (nonparallel, if $G$ is simple) panel-and-hinge realization $(G,\bp)$ in $\mathbb{R}^d$ 
satisfying $\rank R(G,\bp)=D(|V|-1)-k$.
\end{lemma}
\begin{proof}
%Let us consider the case where $G$ is connected but not 2-edge-connected.
%The strategy is basically the same as the proof given in Lemma~\ref{lemma:disconnected}.
Let us consider the case where $G$ is connected.
(The case of disconnected graphs is much easier and hence is omitted.)
Since $G$ has a cut edge $uv$, $G$ can be partitioned into 
two subgraphs $G_1=(V_1,E_1)$ and $G_2=(V_2,E_2)$ such that
$u\in V_1$, $v\in V_2$, 
$V_1\cap V_2=\emptyset$, $V_1\cup V_2=V$ and $\delta_G(\{V_1,V_2\})=\{uv\}$.
Let $k_1$ and $k_2$ be the deficiencies of $\widetilde{G_1}$ and $\widetilde{G_2}$, respectively.
Then, $k=k_1+k_2+1$ holds by Lemma~\ref{lemma:small_connectivity} 
and also $G_i$ is a minimal $k_i$-dof-graph for each $i=1,2$ by Lemma~\ref{lemma:subgraph}.
By (\ref{eq:hypothesis}), we have a (nonparallel, if $G_i$ is simple) panel-and-hinge realization $(G_i,\bp_i)$
satisfying $\rank R(G_i,\bp_i)=D(|V_i|-1)-k_i$ for each $i=1,2$.
Since the choices of $\bp_1$ and $\bp_2$ are independent of each other and also since 
the rank of the rigidity matrix is invariant under an isometric transformation of the whole framework, 
we can take $\bp_1$ and $\bp_2$ such that $\Pi_{G_1,\bp_1}(v_1)$ and $\Pi_{G_2,\bp_2}(v_2)$ are nonparallel for any pair of $v_1\in V_1$ and $v_2\in V_2$.
In particular, $\Pi_{G_1,\bp_1}(u)\cap \Pi_{G_2,\bp_2}(v)$ is a $(d-2)$-affine subspace in $\mathbb{R}^d$.
Define a mapping $\bp$  as
\begin{equation*}
\bp(e)=\begin{cases}
\bp_1(e) & \text{if }  e\in E_1 \\
\bp_2(e) & \text{if } e\in E_2 \\
\Pi_{G_1,\bp_1}(u)\cap \Pi_{G_2,\bp_2}(v) & \text{if } e=uv.
\end{cases}
\end{equation*}
Then, $(G,\bp)$ is a (nonparallel, if $G$ is simple) panel-and-hinge realization of $G$.
By $\delta_G(\{V_1,V_2\})=\{uv\}$, the rigidity matrix $R(G,\bp)$ can be described as
\begin{equation*}
\begin{array}{r@{\,}l}
R(G,\bp) &=
\begin{array}{c@{\,}rc|c@{\,}l}

 &
 & \multicolumn{1}{c}{\scriptstyle V_1}
 & \multicolumn{1}{c}{\scriptstyle V_2}
 &
 \\
 \scriptstyle E_1 & \ldelim({3}{5pt}[] & R(G_1,\bp_1) &  {\bf 0}         & \ldelim){3}{5pt}[] \\ \cline{3-4} 
 \scriptstyle uv  &                    & \multicolumn{2}{c}{R(G,\bp;uv)} &                  \\ \cline{3-4}
 \scriptstyle E_2 &                    & {\bf 0}      & R(G_2,\bp_2)     &                  \\
\end{array}\\
 & =
\begin{array}{c@{\,}rc|c|c|c@{\,}l}

 &
 & \multicolumn{1}{c}{\scriptstyle V_1\setminus\{u\}}
 & \multicolumn{1}{c}{\scriptstyle u}
 & \multicolumn{1}{c}{\scriptstyle v}
 & \multicolumn{1}{c}{\scriptstyle V_2\setminus\{v\}}
 &
 \\
 \scriptstyle E_1 & \ldelim({3}{5pt}[] & R(G_1,\bp_1;E_1,V_1\setminus\{u\}) & R(G_1,\bp_1;E_1,u) & {\bm 0}     & {\bm 0}  & \ldelim){3}{5pt} \\ \cline{3-6}
 \scriptstyle uv  &                    & {\bf 0}                            & r(\bp(uv))         & -r(\bp(uv)) & {\bm 0}   & \ \ \ .  \\ \cline{3-6}
 \scriptstyle E_2 &                    & {\bf 0}                            & {\bm 0}            & \multicolumn{2}{c}{R(G_2,\bp_2)}  &        \\
\end{array}
\end{array}
\end{equation*}
Notice that 
$\rank R(G_1,\bp_1;E_1,V_1\setminus\{u\})=\rank R(G_1,\bp)=D(|V_1|-1)-k_1$ holds by Lemma~\ref{lemma:vertex_deletion}.
Note also $\rank r(\bp(uv))=D-1$ from the definition (see Section~\ref{subsec:rigidity_matrix}).
Hence, by $k=k_1+k_2+1$ and $|V|=|V_1|+|V_2|$, we obtain 
$\rank R(G,\bp)\geq \rank R(G_1,\bp_1;E_1,V_1\setminus\{u\})+\rank r(\bp(uv))+\rank R(G_2,\bp_2)
=D(|V_1|-1)-k_1+(D-1)+D(|V_2|-1)-k_2
=D(|V|-1)-(k_1+k_2+1) 
=D(|V|-1)-k$.
\end{proof}

\subsection{The case where $G$ contains a proper rigid subgraph}
\label{subsec:case3-1}

%In this subsection, we shall show the followings:
%Let $G=(V,E)$ be a 2-edge-connected minimal $k$-dof-graph with $|V|\geq 3$ for some nonnegative integer $k$.
%Suppose there exists a proper rigid subgraph in $G$.
%Then, there is a (nonparallel, if $G$ is simple) panel-and-hinge realization $(G,\bp)$ in $\mathbb{R}^d$ satisfying
%$\rank R(G,\bp)=D(|V|-1)-k$.

Let us first describe a proof sketch.
Let $G'=(V',E')$ be a proper rigid subgraph in a minimal $k$-dof-graph $G=(V,E)$.
Note that $G'$ is a minimal $0$-dof-graph by Lemma~\ref{lemma:subgraph} with $1<|V'|<|V|$.
Let $G/E'=((V\setminus V')\cup\{v^*\}, E\setminus E')$ be the graph obtained from $G$ by contracting the edges of $E'$,
where $v^*$ is the new vertex obtained by the contraction.
Then, by Lemma~\ref{lemma:contraction}, $G/E'$ is a minimal $k$-dof-graph with $|(V\setminus V')\cup\{v^*\}|<|V|$.
Therefore, by the inductive hypothesis (i.e.(\ref{eq:hypothesis})), there exist panel-and-hinge realizations $(G',\bp_1)$ and $(G/E',\bp_2)$ 
satisfying $\rank R(G',\bp_1)=D(|V'|-1)$ and $\rank R(G/E',\bp_2)=D(|V\setminus V'\cup\{v^*\}|-1)-k$.
Based on these realizations, we shall construct a realization of $G$.
Intuitively, we shall replace the body associated with $v^*$ in $(G/E',\bp_2)$ by $(G',\bp_1)$, 
by regarding $(G',\bp_1)$ as a rigid body in $\mathbb{R}^d$ 
and show that the rank of the resulting framework becomes 
$\rank R(G',\bp_1)+\rank R(G/E',\bp_2)$, which is equal to $D(|V|-1)-k$.
For the sake of the lack of the genericity for multigraphs, 
we will need three subcases.
Lemma~\ref{lemma:case3-1_multi} deals with the case where $G$ is not a simple graph.
Lemma~\ref{lemma:case3-1_simple} deals with the case where $G$ is simple and 
there exists a proper rigid subgraph $G'=(V',E')$ such that $G/E'$ is also simple.
Lemma~\ref{lemma:case3-1_other} deals with the rest of the cases.

%In each case, we use the bijection $\phi:E\setminus E'\rightarrow E/E'$ defined as
%\begin{equation*}
%\phi(uv)=\begin{cases}
%uv & \text{if } uv\in (E\setminus E')\setminus \delta_G(V') \\
%uv^* & \text{if  $uv\in \delta_G(V')$ with $u\in V\setminus V', v\in V'$},
%\end{cases}
%\end{equation*}
%which represents the one-to-one correspondence between the edges of $E\setminus E'$ and those after the contraction of $E'$.

\begin{lemma}
\label{lemma:case3-1_multi}
Let $G=(V,E)$ be a minimal $k$-dof-graph with $|V|\geq 3$.
Suppose that $G$ is not simple and also that (\ref{eq:hypothesis}) holds.
Then, there is a panel-and-hinge realization $(G,\bp)$ in $\mathbb{R}^d$ satisfying
$\rank R(G,\bp)=D(|V|-1)-k$.
\end{lemma}
\begin{proof}
Let $e$ and $f$ be multiple edges connecting $a$ and $b$ for some $a,b\in V$.
Then, notice that the graph $G[\{e,f\}]$ edge-induced by $\{e,f\}$ is a proper rigid subgraph in $G$ 
since $\widetilde{e}\cup\widetilde{f}$ contains $D$ edge-disjoint spanning trees on $\{a,b\}$.
Hence we can assume $G'=G[\{e,f\}]=(V'=\{a,b\}, E'=\{e,f\})$.
By Lemma~\ref{lemma:base}, there exists a panel-and-hinge realization $(G',\bp_1)$ of $G'$ such that $\rank R(G',\bp_1)=D$  and 
$\Pi_{G',\bp_1}(a)=\Pi_{G',\bp_1}(b)$.
Also, by (\ref{eq:hypothesis}), 
there exists a panel-and-hinge realization $(G/E',\bp_2)$ satisfying $\rank R(G/E',\bp_2)=D(|V|-2)-k$.
Since the choices of $\bp_1$ and $\bp_2$ are independent of each other, 
we can take $\bp_1$ and $\bp_2$ in such a way that $\Pi_{G/E',\bp_2}(v^*)=\Pi_{G',\bp_1}(a)=\Pi_{G',\bp_1}(b)$.
Define a mapping on $E$ as
\begin{equation}
\label{eq:p3_multi}
\bp(e)=\begin{cases}
\bp_1(e) & \text{ if } e\in E'(=\{e,f\}) \\
\bp_2(e) & \text{ if } e\in E\setminus E'.
\end{cases}
\end{equation}
Intuitively, $(G,\bp)$ is a panel-and-hinge realization of $G$ obtained from $(G/E',\bp_2)$ 
by identifying the panels of $a$ and $b$ with that of $v^*$,
and in fact it can be easily verified from the definition that $(G,\bp)$ is a panel-and-hinge realization satisfying 
$\Pi_{G,\bp}(u)=\Pi_{G/E',\bp_2}(u)$ for each $u\in V\setminus \{a,b\}$ 
and $\Pi_{G,\bp}(u)=\Pi_{G/E',\bp_2}(v^*)$ for each $u\in \{a,b\}$.

%We should verify that $(G,\bp)$ is actually a panel-and-hinge realization.
%\begin{claim}
%\label{claim:multi}
%$(G,\bp)$ is a panel-and-hinge realization such that 
%$\Pi_{G,\bp}(u)=\Pi_{G/E',\bp_2}(u)$ for each $u\in V\setminus \{a,b\}$ 
%and $\Pi_{G,\bp}(u)=\Pi_{G/E',\bp_2}(v^*)$ for each $u\in \{a,b\}$.
%\end{claim}
%\begin{proof}
%Consider $u\in V\setminus \{a,b\}$.
%For an edge $e\in \delta_G(u)$, we have $\bp(e)=\bp_2(e)\subset \Pi_{G/E',\bp_2}(u)$ since $e$ is incident to $u$ in $G/E'$.
%This implies that all the hinges $\bp(e)$ for $e\in \delta_G(u)$ are contained in $\Pi_{G/E',\bp_2}(u)$ 
%and thus $\Pi_{G,\bp}(u)=\Pi_{G/E',\bp_2}(u)$ holds.
%
%Let us consider $a\in V'$.
%For $e\in \delta_G(a)\cap (E\setminus E')$,  we have $\bp(e)=\bp_2(e)\subset \Pi_{G/E',\bp_2}(v^*)$ since $e$ is incident to $v^*$ in $G/E'$.
%This implies that all the hinges $\bp(e)$ for $e\in \delta_G(a)\cap (E\setminus E')$ are contained in $\Pi_{G/E',\bp_2}(v^*)$.
%Also, for $e\in \delta_G(a)\cap E'$, we have $\bp(e)\subset \Pi_{G/E',\bp_2}(v^*)$ 
%from $\bp(e)=\bp_1(e)\subset \Pi_{G',\bp_1}(a)=\Pi_{G/E',\bp_2}(v^*)$. 
%Thus $\Pi_{G,\bp}(a)=\Pi_{G/E',\bp_2}(v^*)$ follows.
%The symmetric argument between $a$ and $b$ also implies $\Pi_{G,\bp}(b)=\Pi_{G/E',\bp_2}(v^*)$.
%\end{proof}

Let us take a look at the rigidity matrix $R(G,\bp)$:
\begin{equation}
\label{eq:case3-1_R}
R(G,\bp)=
\begin{array}{c@{\,}rc|c@{\,}l}

 &
 & \multicolumn{1}{c}{\scriptstyle V'}
 & \multicolumn{1}{c}{\scriptstyle V\setminus V'}
 &
 \\
 \scriptstyle E'            & \ldelim({2}{5pt}[] & R(G',\bp_1) &  {\bf 0}         & \ldelim){2}{5pt}[] \\ \cline{3-4} 
 \scriptstyle E\setminus E' &                    & {\bf *}     & R(G,\bp;E\setminus E',V\setminus V')     &                 \\
\end{array}
\end{equation}
Since $\bp(e)=\bp_2(e)$ for every $e\in E\setminus E'$ by (\ref{eq:p3_multi}), we 
can take\footnote{Recall the discussion in Section~2; the entries of the rigidity matrix are not uniquely defined and depend on the choice of
a basis of the orthogonal complement of the vector space spanned by each screw center although the null space of the rigidity matrix is determined uniquely.} 
the entries of $R(G,\bp;E\setminus E', V\setminus V')$ to be 
\begin{equation}
\label{eq:case_multi}
R(G,\bp;E\setminus E',V\setminus V')=R(G/E',\bp_2;E\setminus E',V\setminus V').
\end{equation}
We also remark that $R(G/E',\bp_2;E\setminus E',V\setminus V')$ is the matrix obtained from $R(G/E',\bp_2)$
by deleting the $D$ consecutive columns associated with $v^*$.
Hence, by (\ref{eq:case_multi}) and Lemma~\ref{lemma:vertex_deletion}, 
we obtain
\begin{equation}
\label{eq:case_multi2}
\rank R(G,\bp;E\setminus E',V\setminus V')=\rank R(G/E',\bp_2;E\setminus E',V\setminus V')=\rank R(G/E',\bp_2).
\end{equation}
As a result, by (\ref{eq:case3-1_R}) and (\ref{eq:case_multi2}), 
we obtain
$\rank R(G,\bp)\geq \rank R(G',\bp_1)+\rank R(G,\bp;E\setminus E', V\setminus V')
=\rank\ R(G',\bp_1)+ \rank R(G/E',\bp_2)
= D+D(|V|-2)-k=D(|V|-1)-k$.
\end{proof}

\begin{lemma}
\label{lemma:case3-1_simple}
Let $G=(V,E)$ be a minimal $k$-dof-graph with $|V|\geq 3$.
Suppose that $G$ is simple and contains a proper rigid subgraph $G'=(V',E')$ such that $G/E'$ is simple.
Also suppose that (\ref{eq:hypothesis}) holds.
Then, there exists a nonparallel panel-and-hinge realization $(G,\bp)$ satisfying $\rank R(G,\bp)=D(|V|-1)-k$.
\end{lemma}
\begin{proof}
The story of the proof is basically the same as that of Lemma~\ref{lemma:case3-1_multi} although we need a slightly different mapping $\bp$.
By Lemma~\ref{lemma:contraction} and (\ref{eq:hypothesis}), there exist {\em nonparallel} panel-and-hinge realizations $(G',\bp_1)$ and $(G/E',\bp_2)$ 
satisfying $\rank R(G',\bp_1)=D(|V'|-1)$ and $\rank R(G/E',\bp_2)=D(|V\setminus V'\cup\{v\}|-1)-k$.
From the definition of generic nonparallel panel-and-hinge realizations of simple graphs discussed in Section~\ref{subsec:genericity}, 
$\bp_1$ and $\bp_2$ can be taken in such a way that 
the set of all coefficients appearing in the equations expressing 
the hyperplanes $\Pi_{G',\bp_1}(v_1)$ for all $v_1\in V'$ and $\Pi_{G/E',\bp_2}(v_2)$ for all $v_2\in V\setminus V'$ 
is algebraically independent over $\mathbb{Q}$.
We define a mapping ${\bf p}$ on $E$ as follows:
\begin{equation}
\label{eq:p3}
\bp(e)=\begin{cases}
\bp_1(e) & \text{ if } e\in E' \\
\bp_2(e) & \text{ if } e\in E\setminus (E'\cup \delta_G(V')) \\
\Pi_{G/E',\bp_2}(u)\cap \Pi_{G',\bp_1}(v) & \text{ if } e=uv\in \delta_G(V') \text{ with } u\in V\setminus V', v\in V'. 
\end{cases}
\end{equation}
Then, $(G,\bp)$ is a nonparallel panel-and-hinge realization of $G$
since all $\bp(e)$ for  $e\in \delta_{G}(v_1)$ are contained in $\Pi_{G',\bp_1}(v_1)$ for each $v_1\in V'$ 
and all $\bp(e)$ for $e\in \delta_G(v_2)$ are contained in $\Pi_{G/E',\bp_2}(v_2)$ for each $v_2\in V\setminus V'$.
%Although there may be some vertices $u,v\in V$ such that $\Pi_{G,\bp}(u)=\Pi_{G,\bp}(v)$, 
%we can now convert $(G,\bp)$ to a nonparallel panel-and-hinge realization of $G$ 
%by perturbing the hyperplane appeared in $(G,\bp)$ without decreasing the rank of the rigidity matrix
%from the genericity of nonparallel realizations discussed in Section~\ref{subsec:genericity}.
Consider the rigidity matrix $R(G,\bp)$, which can be described in the same way as (\ref{eq:case3-1_R}).
We shall derive $\rank R(G,\bp;E\setminus E', V\setminus V')=\rank R(G/E',\bp_2)$ as was done in the proof of the previous lemma.
This will prove $\rank R(G,\bp)=D(|V|-1)-k$ as before.
To obtain $\rank R(G,\bp;E\setminus E', V\setminus V')=\rank R(G/E',\bp_2)$, 
we shall compare $\bp_2$ with the restriction $\bp_{E\setminus E'}$ of $\bp$ on $E\setminus E'$, that is,
%Let $\phi:\delta_{G/E'}(v^*)\rightarrow W$ that maps an edge $e\in \delta_{G/E'}$
\begin{equation}
\label{eq:p3'}
\bp_{E\setminus E'}(e)=\begin{cases}
\bp_2(e) & \text{if } e \in E\setminus (E'\cup \delta_G(V')) \\
\Pi_{G/E',\bp_2}(u)\cap \Pi_{G',\bp_1}(v) & \text{if } e=uv\in\delta_G(V') \text{ with } u\in V\setminus V', v\in V'. 
\end{cases}
\end{equation}
The body-and-hinge framework $(G/E',\bp_{E\setminus E'})$ is not a panel-and-hinge realization in general; 
all the hinges $\bp_{E\setminus E'}(e)$ of $e\in\delta_{G/E'}(u)$ are contained in the panel $\Pi_{G/E',\bp_2}(u)$ for each $u\in V\setminus V'$ 
while the hinges $\bp_{E\setminus E'}(e)$ of $e\in \delta_{G/E'}(v^*)$ may not be on a panel.
Intuitively, $(G/E',\bp_{E\setminus E'})$ is a body-and-hinge realization of $G/E'$ such that 
every $v\in V\setminus V'$ is realized as a panel $\Pi_{G/E',\bp_2}(v)$ while $v^*$ may be realized as a $d$-dimensional body.
\begin{claim}
\label{claim:case3-1}
$\rank R(G/E', \bp_{E\setminus E'})\geq \rank R(G/E', \bp_2)$ holds.
\end{claim}
\begin{proof}
%By Proposition~\ref{prop:preliminaries}, the rank of the rigidity matrix of a generic body-and-hinge realization of $G/E'$
%is equal to $D(|V\setminus V'\cup\{v^*\}|-1)-k$.
%Hence, the rank of the rigidity matrix of a generic realization of $G/E'$ 
%can achieve at most $D(|V\setminus V'\cup\{v^*\}|-1)-k$, which is equal to $\rank R(G/E', \bp_2)$.
%Thus, we have $\rank R(G/E',\bp_{E/E'})\leq \rank R(G/E',\bp_2)$.
%Let us consider the converse direction.
%We assume that $G$ is a simple graph.
%The case when $G$ has some parallel edges can be proved in the same way by using the genericity of nonparallel realizations discussed in Section~\ref{subsec:genericity}.
As mentioned in \cite{Jackson:07} and also in Section~\ref{subsec:genericity},
each entry of the rigidity matrix of a nonparallel realization 
can be written in terms of the coefficients appearing in the equations representing the panels associated with $v\in V$.
Although $(G/E',\bp_{E\setminus E'})$ is not a panel-and-hinge realization, 
due to the definition (\ref{eq:p3'}),
each entry of $R(G/E',\bp_{E\setminus E'})$ is also described in 
terms of the coefficients appearing in the equations expressing $\Pi_{G',\bp_1}(v_1)$ for $v_1\in V'$ 
and $\Pi_{G/E',\bp_2}(v_2)$ for $v_2\in V\setminus V'$.
Notice that the collection of all the nonparallel panel-and-hinge realizations of $G/E'$ 
is a subset of the collection of all possible realizations of $G/E'$ determined by  (\ref{eq:p3'}) for some choices of  $\bp_1$ and $\bp_2$,
because for any $(G/E',\bp_2)$ we can obtain $(G/E',\bp_{E\setminus E'})=(G/E',\bp_2)$ 
by setting $\Pi_{G',\bp_1}(u)$ to be $\Pi_{G/E',\bp_2}(v^*)$ for all $u\in V'$.
This implies $\rank R(G/E',\bp_{E\setminus E'})\geq \rank R(G/E',\bp_2)$ since
$\bp_1$ and $\bp_2$ have been taken in such a way that  
the set of all coefficients appearing in the equations representing the panels
is algebraically independent over $\mathbb{Q}$ 
(and hence $R(G/E',\bp_{E\setminus E'})$ takes the maximum rank over all realizations of $G/E'$ determined by (\ref{eq:p3'})).
\end{proof}

Let us take a look at the rigidity matrix $R(G,\bp)$ written in (\ref{eq:case3-1_R}).
%$$
%\begin{array}{c@{}cc}
%   & \begin{array}{cc} \   V' \ \ \ \ \ \ \ & \ \ \ \ \ V\setminus V' \ \ \ \ \ \  \end{array}\\
%\begin{array}{c}E' \\ E\setminus E' \end{array} &
%\left(\begin{array}{c|c}
%R(G',\bp_1) &  0 \\ \hline
%* &  R(G,\bp;E\setminus E', V\setminus V')  
%\end{array}\right).
%\end{array}
%$$
Since $\bp_{E\setminus E'}$ is the restriction of $\bp$ on $E\setminus E'$,
we have 
\begin{equation}
\label{eq:case3-1}
R(G,\bp;E\setminus E',V\setminus V')=R(G/E',\bp_{E/E'};E\setminus E',V\setminus V').
\end{equation}
We remark that $R(G/E', \bp_{E\setminus E'};E\setminus E', V\setminus V')$ is the matrix
obtained from $R(G/E',\bp_{E\setminus E'})$ by deleting the $D$ consecutive columns associated with $v^*$.
Therefore, by (\ref{eq:case3-1}), Lemma~\ref{lemma:vertex_deletion}, and Claim~\ref{claim:case3-1}, we obtain
\begin{align}
\nonumber
\rank R(G,\bp;E\setminus E',V\setminus V')&=
\rank R(G/E', \bp_{E\setminus E'};E\setminus E', V\setminus V') \\
&=\rank R(G/E',\bp_{E/E'}) \nonumber \\
&\geq \rank R(G/E',\bp_2). \label{eq:case3-1-2}
\end{align}
%Thus, combining it with Claim~\ref{claim:case3-1}, we obtain
%\begin{equation}
%\label{eq:case3-1-3}
%\rank R(G,\bp;E\setminus E', V\setminus V')\geq \rank R(G/E',\bp_2).
%\end{equation}
By (\ref{eq:case3-1_R}) and (\ref{eq:case3-1-2}),
we eventually obtain
$\rank R(G,\bp)\geq \rank R(G',\bp_1)+\rank R(G,\bp;E\setminus E', V\setminus V') 
\geq \rank R(G',\bp_1)+ \rank R(G/E',\bp_2) 
= D(|V'|-1)+D(|V\setminus V'\cup\{v^*\}|-1)-k=D(|V|-1)-k$.
\end{proof}

\begin{lemma}
\label{lemma:case3-1_other}
Let $G$ be a minimal $k$-dof-graph with $|V|\geq 3$.
Suppose that $G$ is simple but contains no proper rigid subgraph $G'=(V',E')$ such that $G/E'$ is simple.
Also suppose that (\ref{eq:hypothesis}) holds.
Then, there exists a nonparallel panel-and-hinge realization $(G,\bp)$ satisfying $\rank R(G,\bp)=D(|V|-1)-k$.
\end{lemma}
\begin{proof}
We shall consider the different approach from those of the previous lemmas, based on the following fact. 
\begin{claim}
\label{claim:no_simple}
Suppose that there exists no proper rigid subgraph $G'=(V',E')$ such that $G/E'$ is simple.
Then, $G$ is a $0$-dof-graph and there exists a vertex $v$ of degree two such that $G_v$ is a minimal $0$-dof-graph,
where $G_v$ is the graph obtained from $G$ by the removal of $v$.
\end{claim}
\begin{proof}
Let us take a vertex-inclusionwise maximal proper rigid subgraph $G'=(V',E')$ of $G$.
Note that $G'$ is a minimal $0$-dof-graph by Lemma~\ref{lemma:subgraph}. 
Since $G/E'$ is not simple, there exist a vertex $v\in V\setminus V'$ and 
two edges, denoted by $e$ and $f$, which are incident to $v$ and some other vertices of $V'$.
Let $G''=(V'\cup\{v\},E'\cup\{e,f\})$.
Then, $G'$ is the graph obtained from $G''$ by the removal of $v$.
By Lemma~\ref{lemma:removal}, we obtain $\de(\widetilde{G''})\leq \de(\widetilde{G'})=0$.
Hence $G''$ is also a $0$-dof-graph and equivalently a rigid subgraph of $G$.
Since $G'$ is taken as a vertex-inclusionwise maximal proper rigid subgraph of $G$,
$G''$ cannot be a proper rigid subgraph and consequently $V=V'\cup\{v\}$ holds.
This implies that $G''$ contains $D$ edge-disjoint spanning trees on $V$
and $G$ is a $0$-dof-graph.
Also,  the minimality of $G$ implies $G=G''$.
Since $v$ is a vertex of degree two in $G''=G$, the removal of $v$ from $G$ results in the minimal $0$-dof-graph $G'$.
\end{proof}

Let $v$ be a vertex of degree two whose removal results in a minimal $0$-dof-graph $G_v$ as shown in Claim~\ref{claim:no_simple}.
Let $N_G(v)=\{a,b\}$. Note that $G_v$ is simple since $G$ is simple.
Hence, by (\ref{eq:hypothesis}), there exists a nonparallel panel-and-hinge realization $(G_v,\bq)$ satisfying 
$\rank R(G_v,\bq)=D(|V\setminus\{v\}|-1)$.
We define a mapping $\bp$ on $E$ as follows: 
\begin{equation*}
\label{eq:case3-1-1.5}
\bp(e)=\begin{cases}
\bq(e) & \text{ if } e\in E\setminus\{va,vb\} \\
\Pi^{\circ}\cap \Pi_{G_v,\bq}(a) & \text{ if } e=va \\
\Pi^{\circ}\cap \Pi_{G_v,\bq}(b) & \text{ if } e=vb,
\end{cases}
\end{equation*}
where $\Pi^{\circ}$ is a $(d-1)$-affine subspace such that it is not parallel to $\Pi_{G_v,\bq}(u)$ for $u\in V\setminus\{v\}$ and 
does not contain $\Pi_{G_v,\bq}(a)\cap \Pi_{G_v,\bq}(b)$
(so that $\Pi^{\circ}\cap \Pi_{G_v,\bq}(a)\neq \Pi^{\circ}\cap \Pi_{G_v,\bq}(b)$).
Clearly, $(G,\bp)$ is a nonparallel panel-and-hinge realization.
Let us take a look at the rigidity matrix $R(G,\bp)$:
\begin{equation}
\label{eq:case3_1_R_1}
R(G,\bp)=
\begin{array}{c@{\,}rc|c@{\,}l}

 &
 & \multicolumn{1}{c}{\scriptstyle v}
 & \multicolumn{1}{c}{\scriptstyle V\setminus \{v\}}
 &
 \\
 \scriptstyle va            & \ldelim({3}{5pt}[] & R(G,\bp;va,v) &  {\bf *}         & \ldelim){3}{5pt}[] \\ \cline{3-4} 
 \scriptstyle vb            &                    & R(G,\bp;vb,v) &  {\bf *}         &                    \\ \cline{3-4}
 \scriptstyle E\setminus \{va,vb\} &             & {\bf 0}     & R(G,\bp;E\setminus \{va,vb\},V\setminus \{v\})     &                  \\
\end{array}
\end{equation}
Since $G_v=(V\setminus\{v\},E\setminus\{va,vb\})$ and also $\bp(e)=\bq(e)$ holds for any $e\in E\setminus\{va,vb\}$,
we can take the entries of the rigidity matrix in such a way that $R(G,\bp;E\setminus\{va,vb\},V\setminus\{v\})=R(G_v,\bq)$.
Also, notice that the top-left $2(D-1)\times D$-submatrix, i.e.~$R(G,\bp;\{va,vb\},v)$, 
has the rank equal to $D$ by Lemma~\ref{lemma:base}.
Thus, by (\ref{eq:case3_1_R_1}), we obtain
$\rank R(G,\bp)\geq \rank R(G,\bp;\{va,vb\},v)+\rank R(G,\bp;E\setminus\{va,vb\},V\setminus\{v\})
=D+\rank R(G_v,\bq) 
=D+D(|V\setminus\{v\}|-1)
=D(|V|-1)$.
\end{proof}

This completes the case where $G$ contains a proper rigid subgraph.

\subsection{The case where $G$ is 2-edge-connected and contains no proper rigid subgraph}
\label{subsec:case3-2}
The remaining case for proving Theorem~\ref{theorem:main} is the one in which $G$ is 2-edge-connected and has no proper rigid subgraph.
We shall further split this case into two subcases depending on whether $k>0$ or $k=0$.
The following Lemma~\ref{lemma:case3-2-0} deals with the case of $k>0$,
which can be handled  as an easier case.
%The case of $k=0$ is the most difficult one.
%Since the proof is rather complicated, we shall first present the proof for $3$-dimension in Lemma~\ref{lemma:case3-2}
%and then extend the idea to a higher dimensional case in Lemma~\ref{lemma:case3-2_d}.

Let us first show the following two easy observations, 
which contribute to the claim that the graph $G_v^{ab}$ obtained by the splitting off is always  simple.
\begin{lemma}
\label{claim:simple}
Let $G=(V,E)$ be a 2-edge-connected minimal $k$-dof-graph with $|V|\geq 3$ for some nonnegative integer $k$.
Suppose also that $G$ contains no proper rigid subgraph.
Then, the followings hold:
\begin{description}
\item[(i)] If $|V|=3$, then $k=0$ and there is a nonparallel panel-and-hinge realization $(G,\bp)$ satisfying $\rank R(G,\bp)=D(|V|-1)$.
\item[(ii)] If $|V|\geq 4$, then $G_v^{ab}$ is a simple graph for any vertex $v$ of degree two, where $N_G(v)=\{a,b\}$.
\end{description}
\end{lemma}
\begin{proof}
We remark that $G$ is simple since multiple edges induce a proper rigid subgraph.
When $|V|=3$, $G$ is a triangle because $G$ is simple and 2-edge-connected.
Hence, by Lemma~\ref{lemma:base2}, there is a nonparallel panel-and-hinge framework $(G,\bp)$ satisfying $\rank R(G,\bp)=D(|V|-1)$.
%
%Also it is not difficult to see that $G$ is a $0$-dof-graph 
%since $\widetilde{G}$ contains $D$ edge-disjoint spanning trees; $T_i=\{(va)_i,(vb)_i\}$ for $1\leq i\leq D-2$, 
%$T_{D-1}=\{(va)_{D-1},(ab)_{1}\}$ and $T_{D}=\{(vb)_{D-1},(ab)_2\}$.

Let us consider (ii).
If $G_v^{ab}$ is not simple, then $G_v^{ab}$ contains two parallel edges between $a$ and $b$ because the original graph $G$ is simple.
This implies $ab\in E$.
Therefore $G$ contains a triangle $G[\{va,vb,ab\}]$ as its subgraph.
Since a triangle is a $0$-dof-graph,
$G$ contains a proper rigid subgraph, contradicting the lemma assumption.
\end{proof}

\begin{lemma}
\label{lemma:case3-2-0}
Let $G=(V,E)$ be a 2-edge-connected minimal $k$-dof-graph with $|V|\geq 3$ for some integer $k$ with $k>0$.
Suppose that there exists no proper rigid subgraph in $G$ and that (\ref{eq:hypothesis}) holds.
Then, there is a nonparallel panel-and-hinge realization $(G,\bp)$ in $\mathbb{R}^d$ satisfying $\rank R(G,\bp)=D(|V|-1)-k$.
\end{lemma}
\begin{proof}
By Lemma~\ref{lemma:degree2}, $G$ has a vertex $v$ of degree two with $N_G(v)=\{a,b\}$.
Since $k>0$ and there is no proper rigid subgraph in $G$, Lemma~\ref{lemma:operation} implies that 
$G_v^{ab}$ is a minimal $(k-1)$-dof-graph.
Also, by Lemma~\ref{claim:simple}, we can assume that $G_v^{ab}$ is simple.
Hence,  by (\ref{eq:hypothesis}), 
there exists a generic nonparallel panel-and-hinge realization $(G_v^{ab},\bq)$ in $\mathbb{R}^d$,
which satisfies 
\begin{equation}
\label{eq:0_rank_G_v^{ab}}
\rank R(G_v^{ab},\bq)=D(|V\setminus\{v\}|-1)-(k-1).
\end{equation}

%Since the rank of the rigidity matrix is invariant under any isometric transformation,
%we may assume that $\bq(ab)$ is a line passing through $(0,0,0)$ and $(0,1,0)$,
%and the hyperplanes $\Pi_{G_v^{ab},\bq}(a)$ and $\Pi_{G_v^{ab},\bq}(b)$ associated with $a$ and $b$ in $(G_v^{ab},\bq)$ 
%are expressed by 
%$\Pi_{G_v^{ab},\bq}(a)=\{(x,y,z)\in \mathbb{R}^3\mid z=0\}$ 
%and $\Pi_{G_v^{ab},\bq}(b)=\{(x,y,z)\in \mathbb{R}^3\mid z=c_bx \}$, respectively, for some constants
%$c_b$ with $c_b\neq 0,1$.
%Note that two hyperplanes $H_{(G_v^{ab,q)}(a)$ and $H_{(G_v^{ab,q)}(b)$ associated to $a$ and $b$ in $(G_v^{ab,q)$ 
%are not parallel since $(G_v^{ab,q)$ is not degenerate.
Let $E_v=E\setminus\{va,vb\}$. 
Define a mapping ${\bf p}_1$ on $E$  as follows:
\begin{equation}
\label{eq:0_6_7_p1}
\bp_1(e)=\begin{cases}
\bq(e) & \text{ if } e\in E_v(=E\setminus\{va,vb\}) \\
L & \text{ if } e=va \\
\bq(ab) & \text{ if } e=vb,
\end{cases}
\end{equation}
where $L$ is a $(d-2)$-affine subspace contained in $\Pi_{G_{v}^{ab},\bq}(a)$.
In Figure~\ref{fig:case3-1}, we illustrate this realization.
\begin{figure}[t]
\centering
\includegraphics[width=0.6\textwidth]{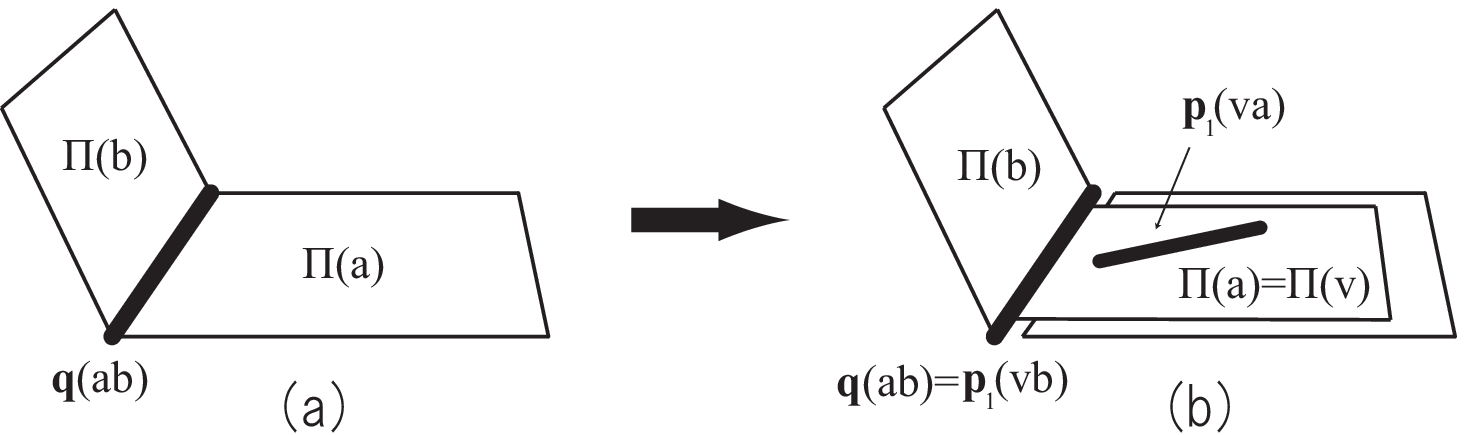}
\caption{The realizations given in the proof of Lemma~\ref{lemma:case3-2-0} around $v$, where
the panels associated with the vertices other than $v,a,b$ are omitted. (a)$(G_v^{ab},\bq)$. (b)$(G,\bp_1)$.}
\label{fig:case3-1}
\end{figure}
We need to prove the following fact.
%Note that we still have two degree of freedoms for the choice of $p_{va}$, that is, the choices of $p_x^1$ and $p_y^1$.
\begin{claim}
\label{claim:coplanarity}
$(G,\bp_1)$ is a panel-and-hinge realization of $G$ for any choice of $(d-2)$-affine subspace $L$ contained in $\Pi_{G_v^{ab},\bq}(a)$.
\end{claim}
\begin{proof}
From the definition of $\bp_1$ and from the fact that $(G_v^{ab},\bq)$ is a panel-and-hinge realization,
every hinge $\bp_1(e)$ for $e=uw\in E\setminus \{va,vb\}$ is appropriately contained in the panels 
$\Pi_{G_v^{ab},\bq}(u)$ and $\Pi_{G_v^{ab},\bq}(w)$. 
This implies that,  for every $u\in V\setminus\{v,a,b\}$,  
all the hinges $\bp_1(e)$ for $e\in\delta_G(u)$ are contained in $\Pi_{G_v^{ab},\bq}(u)$.

Notice that $\bp_1(vb)=\bq(ab)\subset \Pi_{G_v^{ab},\bq}(b)$ 
and hence all the hinges $\bp_1(e)$ for $e\in\delta_G(b)$ are contained in $\Pi_{G_v^{ab},\bq}(b)$.
Similarly, $\bp_1(va)=L\subset \Pi_{G_v^{ab},\bq}(a)$ implies that
all the hinges $\bp_1(e)$ for $e\in\delta_G(a)$  are contained in $\Pi_{G_v^{ab},\bq}(a)$.

Finally, as for the two hinges $\bp_1(va)$ and $\bp_1(vb)$ for $\delta_G(v)=\{va,vb\}$,
$\bp_1(vb)=\bq(ab)\subset \Pi_{G_v^{ab},\bq}(a)$ and $\bp_1(va)=L\subset\Pi_{G_v^{ab},\bq}(a)$ hold.
Hence, all the hinges $\bp_1(e)$ for $e\in \delta_G(v)$ are contained in $\Pi_{G_v^{ab},\bq}(a)$.
Thus $(G,\bp_1)$ is a panel-and-hinge realization.
\end{proof}
Although $\Pi_{G,\bp_1}(v)$ and $\Pi_{G,\bp_1}(a)$ are parallel in $(G,\bp_1)$, 
at the end of the proof we will convert $(G,\bp_1)$ to a nonparallel panel-and-hinge realization 
by slightly rotating $\Pi_{G,\bp_1}(v)$ 
without decreasing the rank of the rigidity matrix.
An important observation provided by the configuration $\bp_1$ of (\ref{eq:0_6_7_p1}) is as follows.
From $\bp_1(e)=\bq(e)$  for every $e\in E_v$, 
\begin{equation}
\label{eq:0_other}
R(G,\bp_1;E_v,V\setminus\{v\})=R(G_v^{ab},\bq;E_v,V\setminus\{v\}),
\end{equation}
(i.e., the part of the framework $(G,\bp_1)$ which is not related to $\{va,vb,ab\}$ is the exactly same as that of $(G_v^{ab},\bq)$).
Let us take a look at $R(G,\bp_1)$:
\begin{equation}
\label{eq:0_R_1}
R(G,\bp_1)=
\begin{array}{c@{\,}rc|c|c|c@{\,}l}

 &
 & \multicolumn{1}{c}{\scriptstyle v}
 & \multicolumn{1}{c}{\scriptstyle a}
 & \multicolumn{1}{c}{\scriptstyle b}
 & \multicolumn{1}{c}{\scriptstyle V\setminus \{v,a,b\}}
 &
 \\
 \scriptstyle va    & \ldelim({3}{5pt}[] & r(\bp_1(va)) & -r(\bp_1(va))  & {\bf 0} & {\bf 0}  & \ldelim){3}{5pt}[] \\ \cline{3-6} 
 \scriptstyle vb    &                    & r(\bp_1(vb)) &  {\bf 0}       & -r(\bp_1(vb)) & {\bf 0} &                 \\ \cline{3-6}
 \scriptstyle E_v   &                    & {\bf 0}      & \multicolumn{3}{c}{R(G,\bp_1;E_v,V\setminus \{v\})}  &                  \\
\end{array}.
\end{equation}
%We remark that the subsequent observations for $R(G,\bp_1)$ can also apply to
%$R_2$ by a symmetric argument, that is, by changing the role of $a$ and $b$.
We shall perform the fundamental column operations on $R(G,\bp_1)$
which add the $j$-th column of $R(G,\bp_1;v)$ to that of $R(G,\bp_1;a)$  for each $1\leq j\leq D$ one by one, i.e.,
$R(G,\bp_1)$ is changed to 
\begin{equation}
\label{eq:0_R_1I_1}
R(G,\bp_1)=
\begin{array}{c@{\,}rc|c|c|c@{\,}l}

 &
 & \multicolumn{1}{c}{\scriptstyle v}
 & \multicolumn{1}{c}{\scriptstyle a}
 & \multicolumn{1}{c}{\scriptstyle b}
 & \multicolumn{1}{c}{\scriptstyle V\setminus \{v,a,b\}}
 &
 \\
 \scriptstyle va    & \ldelim({3}{5pt}[] & r(\bp_1(va)) & {\bf 0}  & {\bf 0} & {\bf 0}  & \ldelim){3}{5pt}[] \\ \cline{3-6} 
 \scriptstyle vb    &                    & r(\bp_1(vb)) &  r(\bp_1(vb))       & -r(\bp_1(vb)) & {\bf 0} &                 \\ \cline{3-6}
 \scriptstyle E_v   &                    & {\bf 0}      & \multicolumn{3}{c}{R(G,\bp_1;E_v,V\setminus \{v\})}  &                  \\
\end{array}.
\end{equation}
Substituting $\bp_1(vb)=\bq(ab)$ and (\ref{eq:0_other}) into (\ref{eq:0_R_1I_1}), $R(G,\bp_1)$ can be expressed by 
\begin{equation}
\label{eq:0_R_1I_1_3}
\begin{array}{r@{\,}l}
R(G,\bp_1) &= 
\begin{array}{c@{\,}rc|c|c|c@{\,}l}

 &
 & \multicolumn{1}{c}{\scriptstyle v}
 & \multicolumn{1}{c}{\scriptstyle a}
 & \multicolumn{1}{c}{\scriptstyle b}
 & \multicolumn{1}{c}{\scriptstyle V\setminus \{v,a,b\}}
 &
 \\
 \scriptstyle va    & \ldelim({3}{5pt}[] & r(\bp_1(va)) & {\bf 0}  & {\bf 0} & {\bf 0}  & \ldelim){3}{5pt}[] \\ \cline{3-6} 
 \scriptstyle vb    &                    & r(\bp_1(vb)) &  r(\bq(ab))       & -r(\bq(ab)) & {\bf 0} &                 \\ \cline{3-6}
 \scriptstyle E_v   &                    & {\bf 0}      & \multicolumn{3}{c}{R(G_v^{ab},\bq;E_v,V\setminus \{v\})}  &                  \\
\end{array} \\
&=
\begin{array}{c@{\,}rc|c@{\,}l}

 &
 & \multicolumn{1}{c}{\scriptstyle v}
 & \multicolumn{1}{c}{\scriptstyle V\setminus \{v\}}
 &
 \\
 \scriptstyle va    & \ldelim({3}{5pt}[] & r(\bp_1(va)) & {\bf 0}  & \ldelim){3}{5pt}[] \\ \cline{3-4} 
 \begin{matrix} \scriptstyle vb  \\ \scriptstyle E_v \end{matrix}  &                    & 
 \begin{matrix} r(\bp_1(vb))     \\\hline {\bf 0}         \end{matrix}   & R(G_v^{ab},\bq)    & \\
\end{array}
\end{array}
\end{equation}
where we used the fact that $R(G_v^{ab},\bq)$ is expressed by
\begin{equation}
\label{eq:0_G_v^{ab}}
R(G_v^{ab},\bq)=
\begin{array}{c@{\,}rc|c|c@{\,}l}

 &
 & \multicolumn{1}{c}{\scriptstyle a}
 & \multicolumn{1}{c}{\scriptstyle b}
 & \multicolumn{1}{c}{\scriptstyle V\setminus \{v,a,b\}}
 &
 \\
 \scriptstyle ab    & \ldelim({2}{5pt}[] & r(\bq(ab)) & -r(\bq(ab))  & {\bf 0}  & \ldelim){2}{5pt}[] \\ \cline{3-5} 
 \scriptstyle E_v   &                    & \multicolumn{3}{c}{R(G_v^{ab},\bq;E_v,V\setminus \{v\})}  &                  \\
\end{array}.
\end{equation}

Therefore, from (\ref{eq:0_rank_G_v^{ab}}) and (\ref{eq:0_R_1I_1_3}), we can show that $R(G,\bp_1)$ has the desired rank as follows:
$\rank R(G,\bp_1)\geq \rank r(\bp_1(va))+\rank R(G_v^{ab},\bq) =D-1+D(|V\setminus \{v\}|-1)-(k-1) =D(|V|-1)-k$.
As we have remarked, $(G,\bp_1)$ can be converted to a nonparallel realization without decreasing the rank of the rigidity matrix
by Lemma~\ref{lemma:perturbation}. 
Thus, the statement follows.
\end{proof}

The remaining case is the most difficult one, in which $G$ is a 2-edge-connected minimal $0$-dof-graph and 
contains no proper rigid subgraph.
Since the proof is quite complicated, we shall first show the 3-dimensional case, which might be most important.
%Also, 2-dimensional case can be handled as an easiest case
%and so the  complete proof for 2-dimensional case is given in Appendix~\ref{app:d_2}, 
%which will give a concrete image of how the proof strategy goes and will help to understand the idea of higher dimensional case.
We then deal with the general dimensional case in Subsection~\ref{subsec:case3-2_d}. 
%Although the proof therein can be applied to the lower dimensional case ($d=2$ and $3$), 
%we shall separately give the proofs for $d=3$ and $d\ge 4$  for the clear understanding.

\subsubsection{$3$-dimensional case ($D=6$)}
\begin{lemma}
\label{lemma:case3-2}
Let $G=(V,E)$ be a 2-edge-connected minimal $0$-dof-graph with $|V|\geq 3$.
Suppose that there exists no proper rigid subgraph in $G$ and that (\ref{eq:hypothesis}) holds.
Then, there is a nonparallel panel-and-hinge realization $(G,\bp)$ in 3-dimensional space satisfying $\rank R(G,\bp)=D(|V|-1)$.
\end{lemma}
\begin{proof}
Let us first describe the proof outline.
Since $d=3$ and no proper rigid subgraph exists in $G$, 
Lemma~\ref{lemma:degree2} implies that there exist two vertices of degree two which are
adjacent with each other. Let $v$ and $a$ be such two vertices and let 
$N_G(v)=\{a,b\}$ for some $b\in V$ and $N_G(a)=\{v,c\}$ for some $c\in V$ (see Figure~\ref{fig:case3-2}(a)). 
Lemma~\ref{lemma:operation} implies that both $G_v^{ab}$ and $G_a^{vc}$ are minimal
$0$-dof-graphs. Here $G_a^{vc}$ is the graph obtained by performing the
 splitting off operation at $a$ along $vc$. By (\ref{eq:hypothesis}), there exist generic
 nonparallel 
panel-and-hinge realizations $(G_v^{ab},\bq)$ and
 $(G_a^{vc},\bq_{\rho})$ where $\bq_{\rho}$ will be defined later in (\ref{eq:q_rho}). 

We shall first consider the realization 
%By induction hypothesis, 
%there exists a generic nonparallel coplanar realization $(G_v^{ab},\bq)$ in $\mathbb{R}^d$.
 $(G,\bp_1)$ of $G$ based on $(G_v^{ab},\bq)$ which is exactly the same
 as the one we considered in Lemma~\ref{lemma:case3-2-0} (see
 Figure~\ref{fig:case3-2}(d)). Contrary to the case of
 Lemma~\ref{lemma:case3-2-0}, 
we need to construct two more panel-and-hinge realizations 
$(G, \bp_2)$  and $(G, \bp_3)$.  These two realizations are constructed
 based on the graphs $G_v^{ab}$ and $G_a^{vc}$, respectively. 
The realization $(G,\bp_2)$ can be constructed in the manner symmetric to $(G,\bp_1)$ by changing the role 
of $a$ and $b$ as shown in Figure~\ref{fig:case3-2}(e). Although 
$(G, \bp_1)$  and $(G, \bp_2)$ are not nonparallel, we will convert them
 to nonparallel panel-and-hinge realizations by slightly
 rotating 
the panel of $v$ without decreasing the rank of the rigidity matrix by
using Lemma~\ref{lemma:perturbation}.

It can be shown that, in the $2$-dimensional case, at least one of $R(G,\bp_1)$ and $R(G,\bp_2)$ attains the desired rank $D(|V|-1)$
as we will see in the general dimensional case. 
However, for $d=3$, there is a case where the ranks of $R(G,\bp_1)$ and $R(G,\bp_2)$ are both less than the desired value.
%This is a main difficulty of 3-dimensional case.
%To overcome the difficulty, 
To complete the proof, we will introduce one more realization $(G,\bp_3)$.
% based on the graph $G_a^{vc}$.
%Lemma 5.6 says that 
%there exists two vertices of degree two which are adjacent with each other if $d\geq 3$.
%Thus, without loss of generality, we may assume that $a$ is a vertex of degree two with $N_G(a)=\{v,c\}$ for some $c\in V$.
%Consider the graph $G_a^{vc}$ obtained by the splitting off at $a$ along $vc$. 
It is not difficult to see that $G_a^{vc}$ is isomorphic to $G_v^{ab}$
 (see Figure~\ref{fig:case3-2}(b)(f)) and hence there is the realization $(G_a^{vc},\bq_{\rho})$ 
representing  the same panel-and-hinge framework as $(G_v^{ab},\bq)$
by regarding the panel associated with $a$ in $(G_v^{ab},\bq)$ as that with $v$ in $(G_a^{vc},\bq_{\rho})$
as shown in Figure~\ref{fig:case3-2}(g).
We shall then construct the realization $(G,\bp_3)$ based on $(G_a^{vc},\bq_{\rho})$ as shown in Figure~\ref{fig:case3-2}(h) 
(see (\ref{eq:q_rho}) and (\ref{eq:6_7_p'}) for the definitions of $\bq_{\rho}$ and $\bp_3$).
Again we can convert $(G,\bp_3)$ to a nonparallel realization without changing the rank of the rigidity matrix.
Since $(G_a^{vc},\bq_{\rho})$ and $(G_v^{ab},\bq)$ are isomorphic,
they have the same rigidity matrix,
and hence we may think that the new hinge $\bp_3(ac)$ is inserted into $R(G_v^{ab},\bq)$. 
Since the position of $\bp_3(ac)$ is different from that of $\bp_1(va)$ or $\bp_2(vb)$, 
we expect that $\bp_3(ac)$ would eliminate a nontrivial infinitesimal motion appearing in $(G,\bp_1)$ and $(G,\bp_2)$.
We actually show that at least one of $(G,\bp_1)$, $(G,\bp_2)$ and $(G,\bp_3)$
 attains the desired rank by  computing the ranks of the
 rigidity matrices of these realizations with certain choices of $\bp_1(va), \bp_2(vb), \bp_3(ac)$. 

\begin{figure}[t]
\centering
\includegraphics[width=0.9\textwidth]{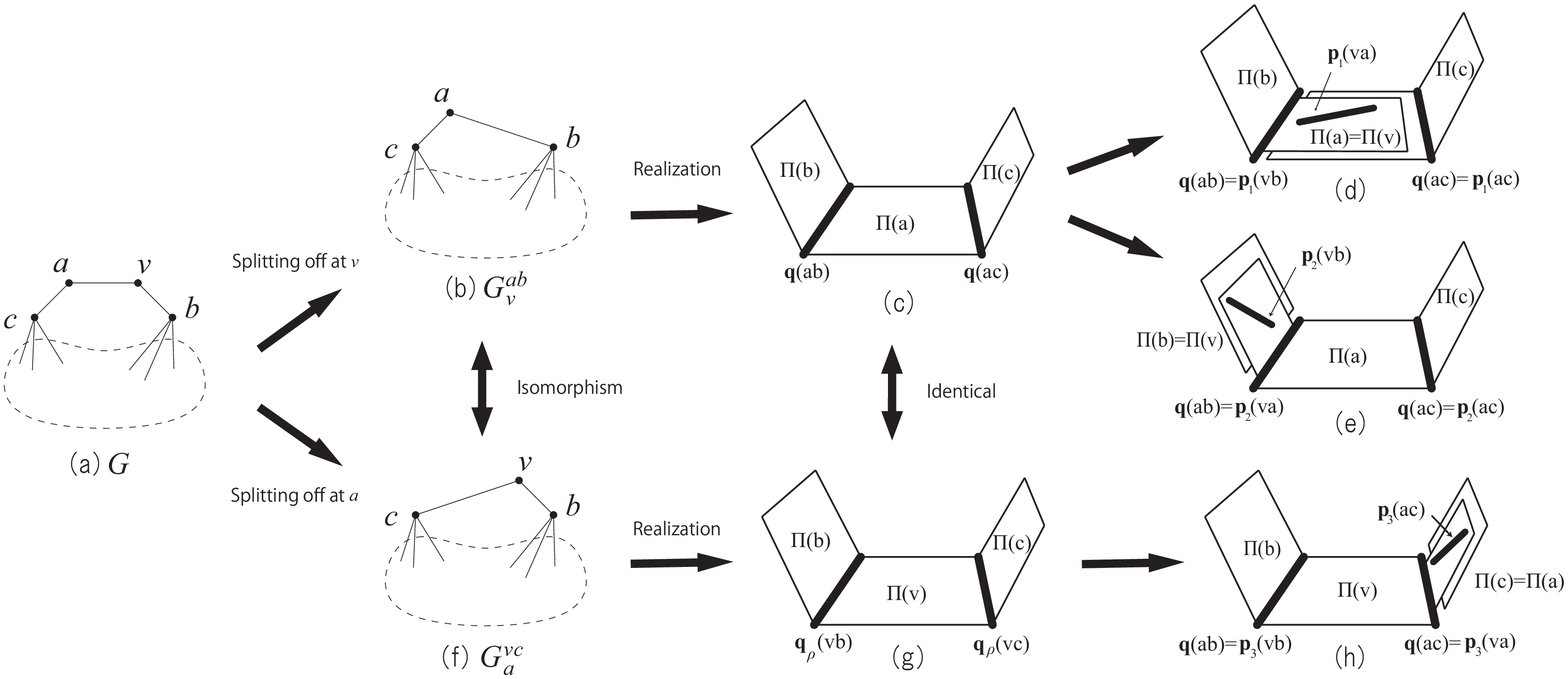}
\caption{The graphs and the realizations given in the proof of Lemma~\ref{lemma:case3-2} around $v$, 
where the panels associated with the vertices other than $v,a,b,c$ are omitted.
(a)$G$, (b)$G_v^{ab}$, (c)$(G_v^{ab},\bq)$, (d)$(G,\bp_1)$, (e)$(G,\bp_2)$, (f)$G_a^{vc}$, (g)$(G_a^{vc},\bq_{\rho})$ and (h)$(G,\bp_3)$.}
\label{fig:case3-2}
\end{figure}

Let us start the rigorous proof.
%We assume $d\geq 3$ (see Appendix~\ref{app:d_2} for $d=2$).
As remarked in the proof of the previous lemma,  $G$ is a simple
 graph (since otherwise multiple edges induce a proper rigid subgraph).
%Recall also that $v$ and $a$ are the two vertices of degree two
%with $N_G(v)=\{a,b\}$ and $N_G(a)=\{v,c\}$ for some $b,c\in V$.
Also, by Lemma~\ref{claim:simple}, 
%Claims~\ref{claim:V3} and~\ref{claim:simple},
we can assume that $|V|\geq 4$ and $G_v^{ab}$ is simple in the subsequent proof.
Hence, by (\ref{eq:hypothesis}), there exists a generic nonparallel panel-and-hinge realization $(G_v^{ab},\bq)$ in $\mathbb{R}^d$,
which satisfies 
\begin{equation}
\label{eq:rank_G_v^{ab}}
\rank R(G_v^{ab},\bq)=
%\begin{cases}
D(|V\setminus\{v\}|-1). %& \text{in Case~1} \\
%D(|V\setminus\{v\}|-1)-(k-1) & \text{in Case~2}.
%\end{cases}
\end{equation}

%Since the rank of the rigidity matrix is invariant under any isometric transformation,
%we may assume that $\bq(ab)$ is a line passing through $(0,0,0)$ and $(0,1,0)$,
%and the hyperplanes $\Pi_{G_v^{ab},\bq}(a)$ and $\Pi_{G_v^{ab},\bq}(b)$ associated with $a$ and $b$ in $(G_v^{ab},\bq)$ 
%are expressed by 
%$\Pi_{G_v^{ab},\bq}(a)=\{(x,y,z)\in \mathbb{R}^3\mid z=0\}$ 
%and $\Pi_{G_v^{ab},\bq}(b)=\{(x,y,z)\in \mathbb{R}^3\mid z=c_bx \}$, respectively, for some constants
%$c_b$ with $c_b\neq 0,1$.
%Note that two hyperplanes $H_{(G_v^{ab,q)}(a)$ and $H_{(G_v^{ab,q)}(b)$ associated to $a$ and $b$ in $(G_v^{ab,q)$ 
%are not parallel since $(G_v^{ab,q)$ is not degenerate.

Let $E_v=E \setminus \{va,vb\}$. Let us recall that the mapping $\bp_1$ on
 $E$ was defined by (\ref{eq:0_6_7_p1}). 
%We should note the following fact.
%Note that we still have two degree of freedoms for the choice of $p_{va}$, that is, the choices of $p_x^1$ and $p_y^1$.
We symmetrically define a mapping $\bp_2$ on $E$ as follows:
\begin{equation}
\label{eq:6_7_p2}
\bp_2(e)=\begin{cases}
\bq(e) & \text{ if } e\in E_v \\
\bq(ab) & \text{ if } e=va \\
L' & \text{ if } e=vb,
\end{cases}
\end{equation}
where $L'$ is a $(d-2)$-affine subspace contained in $\Pi_{G_v^{ab},\bq}(b)$.
Then, the argument symmetric to Claim~\ref{claim:coplanarity} implies that $(G,\bp_2)$ is a panel-and-hinge realization of $G$.
The frameworks $(G,\bp_1)$ and $(G,\bp_2)$ are depicted in Figure~\ref{fig:case3-2}(d) and (e), respectively.

Putting aside for a while the discussion concerning how $R(G, \bp_3)$ is represented, we shall first investigate when $R(G,\bp_1)$ or $R(G,\bp_2)$
takes the desired rank for some choice of
 $L\subset\Pi_{G_v^{ab},\bq}(a)$ or $L'\subset \Pi_{G_v^{ab},\bq}(b)$. 
%Important observations provided by the mappings $\bp_1$ and $\bp_2$ are as 
%follows. From $\bp_1(vb)=\bp_2(va)=\bq(ab)$, we have 
%\begin{equation}\label{eq:6.26}
%r(\bp_1(vb))=r(\bp_2(va))=r(\bq(ab)).
%\end{equation}
%From $\bp_1(e)=\bq(e)$ and $\bp_2(e)=\bq(e)$ for every $e\in E_v$, we
% also have 
%\begin{equation}\label{eq:6.27}
%R(G,\bp_1;E_v,V\setminus \{v\})=R(G,\bp_2;E_v,V\setminus \{v\})=
%R(G_v^{ab},\bq;E_v,V\setminus \{v\}).
%\end{equation}
%Let us first show an outline.
$R(G,\bp_1)$ was described by (\ref{eq:0_R_1I_1_3}), and $R(G,\bp_2)$
is given as follows.
\begin{equation}
\label{eq:R_2}
R(G,\bp_2)=
\begin{array}{c@{\,}rc|c|c|c@{\,}l}

 &
 & \multicolumn{1}{c}{\scriptstyle v}
 & \multicolumn{1}{c}{\scriptstyle a}
 & \multicolumn{1}{c}{\scriptstyle b}
 & \multicolumn{1}{c}{\scriptstyle V\setminus \{v,a,b\}}
 &
 \\
 \scriptstyle vb    & \ldelim({3}{5pt}[] & r(\bp_2(vb)) & -r(\bp_2(vb))  & {\bf 0} & {\bf 0}  & \ldelim){3}{5pt}[] \\ \cline{3-6} 
 \scriptstyle va    &                    & r(\bp_2(va)) & {\bf 0}       & -r(\bp_2(va)) & {\bf 0} &                 \\ \cline{3-6}
 \scriptstyle E_v   &                    & {\bf 0}      & \multicolumn{3}{c}{R(G,\bp_2;E_v,V\setminus \{v\})}  &                  \\
\end{array}
\end{equation}

In a manner similar to the case of $R(G,\bp_1)$, we perform the
 fundamental column operations on $R(G,\bp_2)$ which add the $j$-th column of 
$R(G,\bp_2;v)$ to that of $R(G,\bp_2;b)$ for each $1\le j \le D$.
Substituting $\bp_2(va)=\bq(ab)$
and $R(G,\bp_2;E_v,V\setminus \{v\})=R(G_v^{ab},\bq;E_v,V\setminus \{v\})$ into the resulting matrix, we obtain
\begin{equation}
\label{eq:R_2I_2}
R(G,\bp_2)=
\begin{array}{c@{\,}rc|c@{\,}l}

 &
 & \multicolumn{1}{c}{\scriptstyle v}
 & \multicolumn{1}{c}{\scriptstyle V\setminus \{v\}}
 &
 \\
 \scriptstyle vb    & \ldelim({3}{5pt}[] & r(\bp_2(vb)) & {\bf 0}  & \ldelim){3}{5pt}[] \\ \cline{3-4} 
 \begin{matrix} \scriptstyle va  \\ \scriptstyle E_v \end{matrix}  &                    & 
 \begin{matrix} r(\bp_2(va))     \\\hline {\bf 0}         \end{matrix}   & R(G_v^{ab},\bq)    & \\
\end{array}
\end{equation}
(see the proof of Lemma~\ref{lemma:case3-2-0} for more details).
Let us remark the difference between $k=0$ (the current situation) and 
$k>0$ (Lemma~\ref{lemma:case3-2-0}). In the proof of Lemma~\ref{lemma:case3-2-0}, we have proved that 
$(G, \bp_1)$ attains the desired rank as  
$\rank R(G,\bp_1)\ge D-1 + \rank R(G_v^{ab}, \bq)= D-1 +D(|V\setminus
 \{v\}|-1)-(k-1) =D(|V|-1)-k$, where $\rank R(G_v^{ab}, \bq)$ was equal
 to $D(|V\setminus
 \{v\}|-1)-(k-1)$  since $G_v^{ab}$ was a minimal
 $(k-1)$-dof-graph.
In contrast, because of (\ref{eq:rank_G_v^{ab}}), 
we only have $\rank R(G,\bp_1)\ge D-1 + \rank R(G_v^{ab}, \bq)= D-1 +D(|V\setminus\{v\}|-1)=D(|V|-1) -1$ in the current situation,
which does not complete the proof.
(Similarly, $\rank R(G,\bp_2)\ge D(|V|-1) -1$.) 

%In order to show that this inequality is not tight for at least one of $R(G,\bp_1)$ and $R(G,\bp_2)$,
We shall further convert $R(G,\bp_1)$ of (\ref{eq:0_R_1I_1_3}) to the matrix
 given in (\ref{eq:JR_1I_1}) by applying appropriate 
fundamental row operations based on the claim below (Claim~\ref{claim:3-2-1}).
The matrix given by (\ref{eq:JR_1I_1}) has the following two properties; (i) the
 entries of the submatrix $R(G,\bp_1;(vb)_{i^*}, V \setminus \{v\})$
(induced by the row of $(vb)_{i^*}$ and the columns of $V \setminus
 \{v\})$ are all zero, and (ii) the bottom-right 
submatrix (denoted by $R(G_v^{ab}\setminus (ab)_{i^*}, \bq)$ in (\ref{eq:JR_1I_1})) 
 has the rank equal to $D(|V \setminus \{v\}| -1)$. This implies that,
 if the top-left  $D \times D$-submatrix of (\ref{eq:JR_1I_1}) has a full rank, then
 $\rank R(G,\bp_1) = D + D(|V \setminus \{v\}|- 1) = D(|V |- 1)$
holds. Symmetrically, we shall convert $R(G,\bp_2)$ to that given in (\ref{eq:JR_2I_2})
 which has the  properties similar to (i) and (ii).
If one of the top-left $D \times D$ submatrices
of (\ref{eq:JR_1I_1}) and (\ref{eq:JR_2I_2}) has a full rank, we are done. 
However, this is not always the case. 
%the top-left $D \times D$ submatrices
%of both (6.39) and (6.40) may not have a full rank. 
Hence, we introduce another
realization $(G, \bp_3)$, whose rigidity matrix is  given in (\ref{eq:R'}). In a
manner similar to $R(G,\bp_1)$ and $R(G,\bp_2)$, we convert $R(G,\bp_3)$ to that given in (\ref{eq:J'R'I'}). We will show that
the top-left $D \times D$  submatrix has a full rank in at least one of (\ref{eq:JR_1I_1}), (\ref{eq:JR_2I_2}) and (\ref{eq:J'R'I'}), which
completes the proof. 

%We shall now show  how $R_1$ and $R_2$ will be converted to 
%those given in (\ref{eq:JR_1I_1}) and (\ref{eq:JR_2I_2}), respectively. We shall then show how
 %we will convert $R(G,\bp_3)$ to that given in (\ref{eq:J'R'I'}). Finally, we will prove
 %that  at least one of $R_1$, $R_2$ and $R(G,\bp_3)$
 %attains the desired rank. 

Let us first show how $R(G,\bp_1)$ given in (\ref{eq:0_R_1I_1_3}) is converted to (\ref{eq:JR_1I_1}). 
For this, let us focus on $R(G_v^{ab}, \bq)$ given in (\ref{eq:0_G_v^{ab}}) for a while. 
We say that a row vector of $R(G_v^{ab},\bq)$ is {\it redundant}
 if removing it from $R(G_v^{ab}, \bq)$ does not decrease the rank of
 the matrix. The following claim is a key observation, which states a
 relation between the combinatorial matroid
 $\mathcal{M}(\widetilde{G_v^{ab}})$ (i.e.~the union of the $D$ graphic
 matroids) 
and the linear matroid derived from the row dependence of $R(G_v^{ab}, \bq)$.

\begin{claim}
\label{claim:3-2-1}
In $R(G_v^{ab},\bq)$, at least one of row vectors associated with $ab$ is redundant.
\end{claim}
\begin{proof}
Recall a combinatorial property of $G_v^{ab}$ given in Lemma~\ref{lemma:splitting_off}: 
Since $G_v^{ab}$ is a minimal $k$-dof-graph ($0$-dof-graph), 
Lemma~\ref{lemma:splitting_off}(ii) says that
there exists a base $B'$ of the matroid $\mathcal{M}(\widetilde{G_v^{ab}})$ satisfying $|B'\cap \widetilde{ab}|<D-1$.
This implies that $\widetilde{G_v^{ab}}$ has some redundant edge $(ab)_i$ among $\widetilde{ab}$ with respect to ${\cal M}(\widetilde{G_v^{ab}})$
(i.e., removing $(ab)_i$ from $\widetilde{G_v^{ab}}$ does not decrease the rank of ${\cal M}(\widetilde{G_v^{ab}})$).
We show that this redundancy also appears in the linear matroid derived from $R(G_v^{ab},\bq)$.

Now let us start the proof.
Let $B'$ be a base of the matroid ${\cal M}(\widetilde{G_v^{ab}})$ satisfying $|B'\cap \widetilde{ab}|<D-1$ mentioned above.
Let $h=|B'\cap \widetilde{ab}|$.
Then, we have $h\leq D-2$ and $|B'|=D(|V\setminus\{v\}|-1)$ (since $G_v^{ab}$ is a $0$-dof-graph).
We shall consider the graph $G_v$ that is obtained from $G_v^{ab}$ by removing $ab$.
Clearly, $B'\setminus\widetilde{ab}$ is an independent set of ${\cal M}(\widetilde{G_v})$
with the cardinality $D(|V\setminus\{v\}|-1)-h$
and hence we have ${\rm def}(\widetilde{G_v})\leq h$ by (\ref{eq:pre}).
Namely, $G_v$ is a $k'$-dof-graph for some nonnegative integer $k'$ with $k'\leq h\leq D-2$.

Let $\bq_{E_v}$ denote the restriction of $\bq$ to the edge set $E_v$ of $G_v$.
We claim the following.
\begin{equation}
\label{eq:redundant}
\rank R(G_v,\bq_{E_v})=D(|V\setminus\{v\}|-1)-k'\geq D(|V\setminus\{v\}|-1)-(D-2).
\end{equation}
To see this, recall that,
by the induction hypothesis (\ref{eq:hypothesis}), 
the rigidity matrix of any 
generic nonparallel panel-and-hinge realization of $G_v$ takes 
the rank equal to $D(|V\setminus\{v\}|-1)-k'$.\footnote{The genericity of a nonparallel panel-and-hinge realization for a simple graph says 
that, if one particular nonparallel realization achieves the rank equal to $D(|V\setminus\{v\}|-1)-k'$, 
then all generic nonparallel realizations attain the same rank (see Section~6.1).}
Recall also that $(G_v^{ab}, \bq)$ was defined as a generic nonparallel realization of $G_v^{ab}$ in the inductive step.
Hence, $\bq$ was taken in such a way that 
the set of all the coefficients appearing in the equations expressing the panels  
is algebraically independent over $\mathbb{Q}$.
This property clearly remains in the realization restricted to $E_v$ 
and hence $(G_v,\bq_{E_v})$ is also a generic nonparallel panel-and-hinge realization of $G_v$.
Thus, (\ref{eq:redundant}) holds.

Note that $R(G_v,\bq_{E_v})$ is the matrix obtained from $R(G_v^{ab},\bq)$ by removing the 
$D-1$ rows associated with $ab$. 
Note also that $R(G_v^{ab},\bq)$ has the rank equal to $D(|V\setminus\{v\}|-1)$ by (\ref{eq:rank_G_v^{ab}}).
Hence, from (\ref{eq:redundant}),
adding at most $D-2$ row vectors associated with $ab$ to the rows of $R(G_v,\bq_{E_v})$,
we obtain a set of row vectors which spans the row space of $R(G_v^{ab},\bq)$.
This implies that at least one row vector associated with $ab$ is redundant.
\end{proof}

Let $i^*$ be the index of a redundant row associated with $ab$ shown
 in Claim~\ref{claim:3-2-1}.
Namely, denoting by $R(G_v^{ab}\setminus (ab)_{i^*},\bq)$ 
the matrix obtained from $R(G_v^{ab},\bq)$ by removing the row associated with $(ab)_{i^*}$,
we have
\begin{equation}
\label{eq:i^*}
\rank R(G_v^{ab}\setminus (ab)_{i^*},\bq)=\rank R(G_v^{ab},\bq)=D(|V\setminus\{v\}|-1),
\end{equation}
where the last equation follows from (\ref{eq:rank_G_v^{ab}}). Also, 
since $R(G_v^{ab},\bq;(ab)_{i^*})$ is redundant, it can be expressed
 by a linear combination of the other row vectors of $R(G_v^{ab},\bq)$. 
%We can hence introduce
Thus, there exists a scalar $\lambda_{e_j}$ for each $e_j\in\widetilde{E_{v}}\cup\widetilde{ab}$ 
such that 
\begin{equation}
\label{eq:lambda}
\lambda_{(ab)_{i^*}}R(G_v^{ab},\bq;(ab)_{i^*}) + \sum_{1\leq j\leq D-1, j\ne i^*}\lambda_{(ab)_j}R(G_v^{ab},\bq;(ab)_j)
+\sum_{ e\in E_{v}, 1\leq j\leq D-1} \lambda_{e_j}R(G_v^{ab},\bq;e_j)={\bf 0},
\end{equation}
where 
\begin{equation}
\label{eq:i^*=1}
\lambda_{(ab)_{i^*}}=1.
\end{equation}
%and otherwise $\lambda_{e_j}\in \mathbb{R}$.
%Also, 
In other words, the fundamental row operations (i.e.~the addition of the last two terms of
 the left-hand side of 
(\ref{eq:lambda}) to the row  associated with $(ab)_{i^*}$) change the matrix $R(G_v^{ab},\bq)$ 
so that the row associated with $(ab)_{i^*}$ becomes a zero vector. Such
 row operations can be extended to $R(G,\bp_1)$ of (\ref{eq:0_R_1I_1_3}), by focusing on the fact that 
$R(G,\bp_1;E\setminus\{vb\},V\setminus\{v\})$ is equal to $R(G_v^{ab},\bq)$, 
so that the part of a row vector of $R(G,\bp_1)$ corresponding to $V\setminus\{v\}$ becomes zero as Figure~\ref{fig:row_operations1}.
\begin{figure}[h]
\centering
\includegraphics[width=0.9\textwidth]{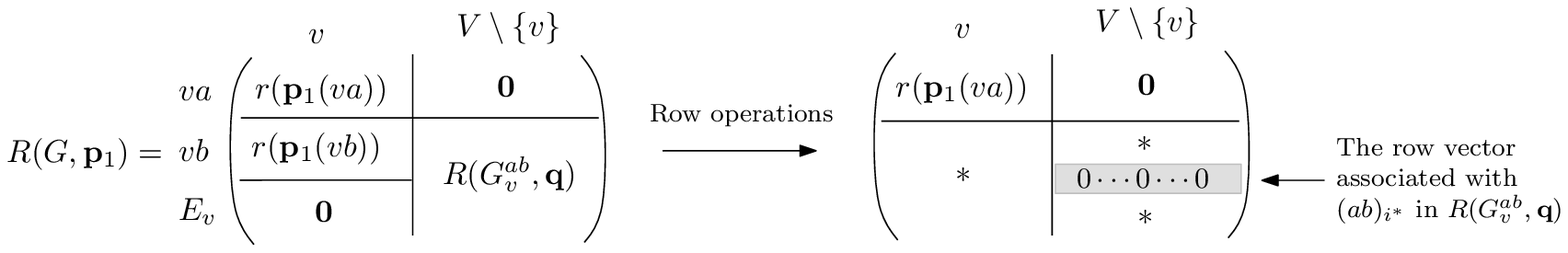}
\caption{Row operations on $R(G,\bp_1)$.}
\label{fig:row_operations1}
\end{figure}

Let us show that the resulting matrix can be described as (\ref{eq:JR_1I_1}) (up to some row exchange).
To show this, let us take a look at these row operations more precisely.
Since the row associated with 
$(ab)_{i^*}$ in $R(G_v^{ab},\bq)$ corresponds to that with $(vb)_{i^*}$
in $R(G,\bp_1)$ of (\ref{eq:0_R_1I_1_3}), performing such row operations is equivalent to 
the addition of 
\begin{equation}\label{eq:row_operation}
\sum_{1\leq j\leq D-1, j\ne i^*} \lambda_{(ab)_j} R(G,\bp_1;(vb)_j)
+\sum_{e\in E_{v}, 1\leq j\leq D-1} \lambda_{e_j}R(G,\bp_1;e_j)
\end{equation}
to the row associated with $(vb)_{i^*}$ of (\ref{eq:0_R_1I_1_3}).
Namely, the row operations convert $R(G,\bp_1;(vb)_{i^*})$ of (\ref{eq:0_R_1I_1_3}) to 
\begin{equation}
\label{eq:new_row}
\lambda_{(ab)_{i^*}}R(G,\bp_1;(vb)_{i^*}) + \sum_{1\leq j\leq D-1, j\ne i^*}\lambda_{(ab)_j}R(G,\bp_1,(vb)_j)
+\sum_{ e\in E_{v}, 1\leq j\leq D-1} \lambda_{e_j}R(G,\bp_1;e_j),
\end{equation}
where $\lambda_{(ab)_{i^*}}=1$.
(Compare it with the left hand side of (\ref{eq:lambda}).)
From $R(G,\bp_1;E\setminus\{va\}, V\setminus \{v\})=R(G_v^{ab},\bq)$ and (\ref{eq:lambda}), 
the components of (\ref{eq:new_row}) associated with $V\setminus\{v\}$ become all zero. 
As for the remaining part of the new row vector (\ref{eq:new_row}) (i.e., the components associated with $v$),
since $R(G,\bp_1;E_v,v)={\bf 0}$ holds as shown in (\ref{eq:0_R_1I_1_3}), we have
\begin{align*}
&\lambda_{(ab)_{i^*}}R(G,\bp_1;(vb)_{i^*}, v) + \sum_{1\leq j\leq D-1, j\ne i^*} \lambda_{(ab)_j} R(G,\bp_1;(vb)_j,v) 
+\sum_{e\in E_{v}, 1\leq j\leq D-1} \lambda_{e_j}R(G,\bp_1;e_j,v) \nonumber \\
&=\sum_{1\leq j\leq D-1} \lambda_{(ab)_j} R(G,\bp_1;(vb)_j,v) \nonumber
\end{align*}
Note that $R(G,\bp_1;(vb)_j,v)=r_j(\bp_1(vb))=r_j(\bq(ab))$ from the definition of the rigidity matrix and from $\bp_1(vb)=\bq(ab)$. 
Thus, we have seen that the new row vector (\ref{eq:new_row}) can be written as 
\begin{equation} 
\label{eq:r}
\begin{array}{cc} (\Bvector{v}{\sum_{j} \lambda_{(ab)_j}r_j(\bq(ab))}, & \Bvector{V\setminus\{v\}}{\bf 0}\ \  )\end{array}.
\end{equation}
As a result, the fundamental row operations change $R(G,\bp_1)$ of (\ref{eq:0_R_1I_1_3}) to the matrix described as (up to some row exchange of $(vb)_{i^*}$)
\begin{equation}
\label{eq:JR_1I_1}
R(G,\bp_1)= 
\begin{array}{c@{\,}rc|c@{\,}l}

 &
 & \multicolumn{1}{c}{\scriptstyle v}
 & \multicolumn{1}{c}{\scriptstyle V\setminus \{v\}}
 &
 \\
 \scriptstyle va    & \ldelim({3}{5pt}[] & r(\bp_1(va)) & {\bf 0}  & \ldelim){3}{5pt}[] \\ \cline{3-4} 
 \scriptstyle (vb)_{i^*} &               & \sum_{j}\lambda_{(ab)_j}r_j(\bq(ab)) & {\bf 0} & \\ \cline{3-4}
                         &               & {\bf *}                     & R(G_v^{ab}\setminus (ab)_{i^*},\bq) & \\
\end{array}
\end{equation}
Applying the symmetric argument to $R(G,\bp_2)$ shown in (\ref{eq:R_2I_2}), we also have
\begin{equation}
\label{eq:JR_2I_2}
R(G,\bp_2)= 
\begin{array}{c@{\,}rc|c@{\,}l}

 &
 & \multicolumn{1}{c}{\scriptstyle v}
 & \multicolumn{1}{c}{\scriptstyle V\setminus \{v\}}
 &
 \\
 \scriptstyle vb    & \ldelim({3}{5pt}[] & r(\bp_2(vb)) & {\bf 0}  & \ldelim){3}{5pt}[] \\ \cline{3-4} 
 \scriptstyle (va)_{i^*} &               & \sum_{j}\lambda_{(ab)_j}r_j(\bq(ab)) & {\bf 0} & \\ \cline{3-4}
                         &               & {\bf *}                     & R(G_v^{ab}\setminus (ab)_{i^*},\bq) & \\
\end{array}
\end{equation}
Note that the same $\lambda_{(ab)_j}, 1\leq j\leq D-1$ and the index $i^*$ are used in (\ref{eq:JR_1I_1}) and (\ref{eq:JR_2I_2})
since they are determined by $(G_v^{ab},\bq)$, and are independent of $\bp_1$ and $\bp_2$.
It is not difficult to see that, 
if the top-left $D\times D$-submatrix of (\ref{eq:JR_1I_1}) (that is, $R(G,\bp_1;\widetilde{va}+(vb)_{i^*},v)$) 
has a full rank,
then we have $\rank R_1\geq D+\rank R(G_v^{ab}\setminus(ab)_{i^*},\bq)=D(|V|-1)$ 
by (\ref{eq:JR_1I_1}) and (\ref{eq:i^*}), and we are done.
Symmetrically, if the top-left $D\times D$-submatrix of (\ref{eq:JR_2I_2}) has a full rank,
then $R(G,\bp_2)$ attains the desired rank.
In $d=2$ we can show that this $D\times D$-submatrix has a full rank in at least one of (\ref{eq:JR_1I_1}) and (\ref{eq:JR_2I_2})
but this is not always true for $d\geq 3$.
Hence, we shall introduce another framework $(G,\bp_3)$ in the following discussion.

Let us consider the splitting off at $a$ along $vc$. (Recall that $a$ is a vertex of degree two.) 
Then, since $v$ and $a$ are adjacent vertices of degree two, 
it is not difficult to see that the resulting graph $G_a^{vc}$ is isomorphic to $G_v^{ab}$  (see Figure~\ref{fig:case3-2}(b) and (f))   
by the mapping $\rho:V\setminus\{a\}\rightarrow V\setminus\{v\}$, i.e.,~$\rho(v)=a$ and $\rho(u)=u$ for $u\in V\setminus \{v,a\}$.
The isomorphism $\rho$ induces the mapping $\bq_{\rho}$ on $E\setminus\{va,ac\}\cup\{vc\}$ in a natural way defined by, for $e=uw\in E\setminus\{va,ac\}\cup\{vc\}$,
\begin{equation}
\label{eq:q_rho}
\bq_{\rho}(e)=\bq(\rho(u)\rho(w))=\begin{cases}
\bq(ab) & \text{if } e=vb \\
\bq(ac) & \text{if } e=vc \\
\bq(e) & \text{otherwise}.
\end{cases}
\end{equation}
The isomorphism $\rho$ between $G_v^{ab}$ and $G_a^{vc}$ implies that 
$(G_a^{vc},\bq_{\rho})$ and $(G_v^{ab},\bq)$ represent the same panel-and-hinge framework in $\mathbb{R}^d$.
In particular, we have $\Pi_{G_a^{vc},\bq_{\rho}}(u)=\Pi_{G_v^{ab},\bq}(u)$ for each $u\in V\setminus \{v,a\}$ 
and $\Pi_{G_a^{vc},\bq_{\rho}}(v)=\Pi_{G_v^{ab},\bq}(a)$.
%We can hence take the entries of $R(G_a^{vc},\bq_{\rho})$ to be the exactly same as those of $R(G_v^{ab},\bq)$,
%that is, 
%\begin{equation}
%\label{eq:G_a}
%R_{G_a^{vc},\bq_{\rho}}[uw,u]=R(G_v^{ab},\bq;\rho(u)\rho(w),\rho(u)] \quad \text{ for each } uw\in E\setminus\{va,ac\}\cup\{vc\}.
%\end{equation}
%\begin{figure}[t]
%\centering
%\includegraphics[width=0.5\textwidth]{isomorphism.eps}
%\caption{$G_v^{ab}$ and $G_a^{vc}$.}
%\label{fig:isomorhism}
%\end{figure}

We shall construct a similar extension of the mapping $\bq_{\rho}$ as was done in $\bp_1$ or $\bp_2$.
Define a mapping $\bp_3$ on $E$ as
\begin{equation}
\label{eq:6_7_p'}
\bp_3(e)=\begin{cases}
\bq_{\rho}(e) & \text{if } e\in E\setminus\{va,ac\} \\
L'' & \text{if } e=ac \\
\bq_{\rho}(vc) & \text{if } e=va
\end{cases}
\end{equation}
where $L''$ is a $(d-2)$-affine subspace contained in $\Pi_{G_a^{vc},\bq_{\rho}}(c)(=\Pi_{G_v^{ab},\bq}(c))$
and it can be shown as in Claim~\ref{claim:coplanarity} that $(G,\bp_3)$ is a panel-and-hinge realization of $G$ in $\mathbb{R}^d$ 
for any choice of $L''\subset \Pi_{G_v^{ab},\bq}(c)$ (see Figure~\ref{fig:case3-2}(h)).
Combining (\ref{eq:q_rho}) and (\ref{eq:6_7_p'}), we actually have
\begin{equation}
\bp_3(e)=\begin{cases}
\bq(ac) & \text{if } e=va\\
\bq(ab) & \text{if } e=vb \\
L'' & \text{if } e=ac \\
\bq(e) &\text{otherwise}.
\end{cases}
\end{equation}
Therefore, we have
\begin{align}
\nonumber
r(\bp_3(va))&=r(\bq(ac)) \\ \label{eq:p'q}
r(\bp_3(vb))&=r(\bq(ab)) \\ \nonumber
R(G,\bp_3;E\setminus\{va,vb,ac\},V\setminus\{v,a\})&=R(G_v^{ab},\bq;E\setminus\{va,vb,ac\},V\setminus\{v,a\}). 
\end{align}

We are now going to convert the rigidity matrix $R(G,\bp_3)$ to that given in (\ref{eq:J'R'I'}),
which is an analogous form to (\ref{eq:JR_1I_1}) or (\ref{eq:JR_2I_2}).
Let us take a look at $R(G,\bp_3)$:
\begin{equation}
\label{eq:R'}
R(G,\bp_3)=
\begin{array}{c@{\,}rc|c|c|c|c@{\,}l}

 &
 & \multicolumn{1}{c}{\scriptstyle a}
 & \multicolumn{1}{c}{\scriptstyle v}
 & \multicolumn{1}{c}{\scriptstyle c}
 & \multicolumn{1}{c}{\scriptstyle b}
 & \multicolumn{1}{c}{\scriptstyle V\setminus \{v,a,b,c\}}
 &
 \\
 \scriptstyle ac    & \ldelim({4}{5pt}[] & -r(\bp_3(ac)) & {\bf 0}  & r(\bp_3(ac)) & {\bf 0}  & {\bf 0} & \ldelim){4}{5pt}[] \\ \cline{3-7} 
 \scriptstyle vb    &                    & {\bf 0}       & r(\bp_3(vb))  & {\bf 0} & -r(\bp_3(vb))  & {\bf 0} &              \\ \cline{3-7} 
 \scriptstyle va    &                    & -r(\bp_3(va)) & r(\bp_3(va))  & {\bf 0} & {\bf 0}  & {\bf 0} &                    \\ \cline{3-7} 
                    &                    & {\bf 0}      & {\bf 0}   & \multicolumn{3}{c}{R(G,\bp_3;E\setminus\{va,vb,ac\},V\setminus \{v,a\})}  & \\
\end{array}.
\end{equation}
Consider the fundamental column operations 
which add the $j$-th column of $R(G,\bp_3;a)$ to that of $R(G,\bp_3;c)$ for each $1\leq j\leq D$.
Then $R(G,\bp_3)$ results in the matrix as follows:
\begin{equation}
\label{eq:R'I'}
R(G,\bp_3)=
\begin{array}{c@{\,}rc|c|c|c|c@{\,}l}

 &
 & \multicolumn{1}{c}{\scriptstyle a}
 & \multicolumn{1}{c}{\scriptstyle v}
 & \multicolumn{1}{c}{\scriptstyle c}
 & \multicolumn{1}{c}{\scriptstyle b}
 & \multicolumn{1}{c}{\scriptstyle V\setminus \{v,a,b,c\}}
 &
 \\
 \scriptstyle ac    & \ldelim({4}{5pt}[] & -r(\bp_3(ac)) & {\bf 0}  & {\bf 0} & {\bf 0}  & {\bf 0} & \ldelim){4}{5pt}[] \\ \cline{3-7} 
 \scriptstyle vb    &                    & {\bf 0}       & r(\bp_3(vb))  & {\bf 0} & -r(\bp_3(vb))  & {\bf 0} &              \\ \cline{3-7} 
 \scriptstyle va    &                    & -r(\bp_3(va)) & r(\bp_3(va))  & -r(\bp_3(va)) & {\bf 0}  & {\bf 0} &                    \\ \cline{3-7} 
                    &                    & {\bf 0}      & {\bf 0}   & \multicolumn{3}{c}{R(G,\bp_3;E\setminus\{va,vb,ac\},V\setminus \{v,a\})}  & \\
\end{array}.
\end{equation}
Substituting (\ref{eq:p'q}) into (\ref{eq:R'I'}), $R(G,\bp_3)$ becomes
\begin{equation}
\label{eq:R'I'_2}
R(G,\bp_3)=
\begin{array}{c@{\,}rc|c|c|c|c@{\,}l}

 &
 & \multicolumn{1}{c}{\scriptstyle a}
 & \multicolumn{1}{c}{\scriptstyle v}
 & \multicolumn{1}{c}{\scriptstyle c}
 & \multicolumn{1}{c}{\scriptstyle b}
 & \multicolumn{1}{c}{\scriptstyle V\setminus \{v,a,b,c\}}
 &
 \\
 \scriptstyle ac    & \ldelim({4}{5pt}[] & -r(\bp_3(ac)) & {\bf 0}  & {\bf 0} & {\bf 0}  & {\bf 0} & \ldelim){4}{5pt}[] \\ \cline{3-7} 
 \scriptstyle vb    &                    & {\bf 0}       & r(\bq(ab))  & {\bf 0} & -r(\bq(ab))  & {\bf 0} &              \\ \cline{3-7} 
 \scriptstyle va    &                    & -r(\bp_3(va)) & r(\bq(ac))  & -r(\bq(ac)) & {\bf 0}  & {\bf 0} &                    \\ \cline{3-7} 
                    &                    & {\bf 0}      & {\bf 0}   & \multicolumn{3}{c}{R(G_v^{ab},\bq;E\setminus\{va,vb,ac\},V\setminus \{v,a\})}  & \\
\end{array}
\end{equation}
Note that, from the definition of a rigidity matrix given in Section~2, $R(G_v^{ab},\bq)$ can be described by (up to some row exchanges)
\begin{equation}
\label{eq:G_v^{ab}_2}
R(G_v^{ab},\bq)=
\begin{array}{c@{\,}rc|c|c|c@{\,}l}

 &
 & \multicolumn{1}{c}{\scriptstyle a}
 & \multicolumn{1}{c}{\scriptstyle c}
 & \multicolumn{1}{c}{\scriptstyle b}
 & \multicolumn{1}{c}{\scriptstyle V\setminus \{v,a,b,c\}}
 &
 \\
 \scriptstyle ab    & \ldelim({3}{5pt}[] & r(\bq(ab))  & {\bf 0} & -r(\bq(ab))  & {\bf 0} & \ldelim){3}{5pt}[]             \\ \cline{3-6} 
 \scriptstyle ac    &                    & r(\bq(ac))  & -r(\bq(ac)) & {\bf 0}  & {\bf 0} &                    \\ \cline{3-6} 
                    &                    & {\bf 0}   & \multicolumn{3}{c}{R(G_v^{ab},\bq;E\setminus\{va,vb,ac\},V\setminus \{v,a\})}  & \\
\end{array}
\end{equation}
which is equal to $R(G,\bp_3;E\setminus\{ac\},V\setminus\{a\})$ given in (\ref{eq:R'I'_2}).
Therefore, (\ref{eq:R'I'_2}) becomes 
\begin{equation}
\label{eq:R'I'_3}
R(G,\bp_3)=
\begin{array}{c@{\,}rc|c@{\,}l}

 &
 & \multicolumn{1}{c}{\scriptstyle a}
 & \multicolumn{1}{c}{\scriptstyle V\setminus \{a\}}
 &
 \\
 \scriptstyle ac    & \ldelim({4}{5pt}[] & -r(\bp_3(ac))  & {\bf 0} & \ldelim){4}{5pt}[]  \\ \cline{3-4} 
 \begin{matrix} \scriptstyle vb \\ \scriptstyle va \\ \ \end{matrix} & & 
 \begin{matrix} {\bf 0} \\ \hline -r(\bp_3(va)) \\ \hline {\bf 0} \end{matrix} & R(G_v^{ab},\bq) & \\
\end{array}.
\end{equation}
Recall that the row of $R(G_v^{ab},\bq)$ associated with $(ab)_{i^*}$ is redundant, i.e.,
removing the row associated with $(ab)_{i^*}$ from $R(G_v^{ab},\bq)$ preserves the rank of $R(G_v^{ab},\bq)$  as shown in (\ref{eq:i^*}).
In order to indicate the dependence of the row vectors within $R(G_v^{ab},\bq)$,
we have introduced $\lambda_{e_j}\in \mathbb{R}$ for each $e_j\in\widetilde{E_v}\cup\widetilde{ab}$ (where $E_v=E\setminus\{va,vb\}$)
such that they satisfy (\ref{eq:lambda}).

We shall again consider the row operations which convert $R(G_v^{ab},\bq;(ab)_{i^*})$ to a zero vector within $R(G_v^{ab},\bq)$ 
and apply these row operations to the matrix $R(G,\bp_3)$ of (\ref{eq:R'I'_3}) as Figure~\ref{fig:row_operation2}.
\begin{figure}[h]
\centering
\includegraphics[width=0.9\textwidth]{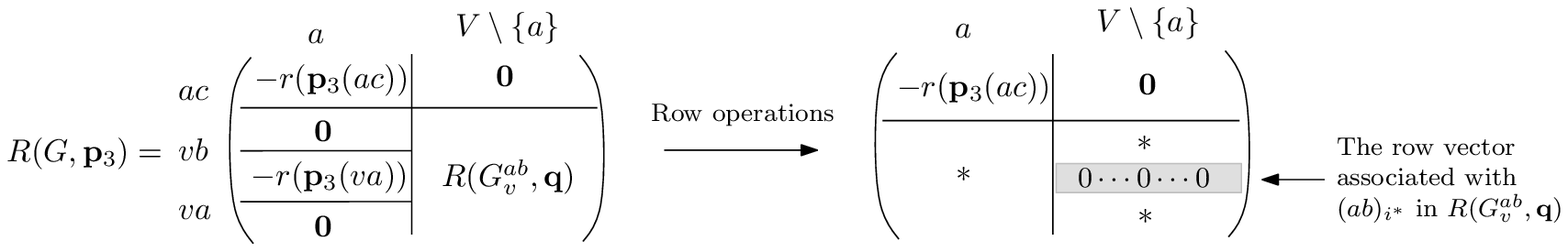}
\caption{Row operations on $R(G,\bp_3)$.}
\label{fig:row_operation2}
\end{figure}

Note that, within the equation $R(G,\bp_3;E\setminus\{ac\},V\setminus\{a\})=R(G_v^{ab},\bq)$ of (\ref{eq:R'I'_3}), 
the rows associated with $(ab)_j$ and $(ac)_j$ in $R(G_v^{ab},\bq)$ 
correspond to those associated with $(vb)_j$ and $(va)_j$ in $R(G,\bp_3)$, respectively, 
(compare (\ref{eq:R'I'_2}) and (\ref{eq:G_v^{ab}_2})).
In particular, the row of $(ab)_{i^*}$ in $R(G_v^{ab},\bq)$ corresponds to
that of $(vb)_{i^*}$ in $R(G,\bp_3)$.
Therefore, the resulting matrix can be described as (\ref{eq:J'R'I'}).

Let us check it more formally. We perform the fundamental row operations which add
\begin{equation*}
\sum_{1\leq j\leq D-1,j\neq i^*}\lambda_{(ab)_j}R(G,\bp_3;(vb)_j)+\sum_{1\leq j\leq D-1}\lambda_{(ac)_j}R(G,\bp_3;(va)_j)
+\sum_{\begin{subarray}{c} e\in E\setminus\{va,vb,ac\} \\ 1\leq j\leq D-1 \end{subarray}} \lambda_{e_j}R(G,\bp_3;e_j)
\end{equation*}
to $R(G,\bp_3;(vb)_{i^*})$ of (\ref{eq:R'I'_3}).
Namely, by $\lambda_{(ab)_{i^*}}=1$, the resulting row becomes
\begin{equation}
\label{eq:new_3}
\sum_{1\leq j\leq D-1}\lambda_{(ab)_j}R(G,\bp_3;(vb)_j)+\sum_{1\leq j\leq D-1}\lambda_{(ac)_j}R(G,\bp_3;(va)_j)
+\sum_{\begin{subarray}{c} e\in E\setminus\{va,vb,ac\} \\ 1\leq j\leq D-1 \end{subarray}} \lambda_{e_j}R(G,\bp_3;e_j).
\end{equation}
Since there is the row correspondence between (\ref{eq:R'I'_2}) and (\ref{eq:G_v^{ab}_2}) described as, 
\begin{align*}
R(G,\bp_3;vb,V\setminus\{a\})&=R(G_v^{ab},\bq;ab), \\ 
R(G,\bp_3;va,V\setminus\{a\})&=R(G_v^{ab},\bq;ac), \\
R(G,\bp_3;e,V\setminus\{a\})&=R(G_v^{ab},\bq;e) && \text{for each } e\in E\setminus\{va,vb,ac\},
\end{align*}
(\ref{eq:lambda}) implies
\begin{equation*}
\begin{split} 
\sum_{1\leq j\leq D-1} \lambda_{(ab)_j} R(G,\bp_3;(vb)_j,V\setminus\{a\}) 
&+\sum_{1\leq j\leq D-1} \lambda_{(ac)_j} R(G,\bp_3;(va)_j,V\setminus\{a\}) \\
&+\sum_{\begin{subarray}{c} e \in  E\setminus\{va,vb,ac\} \\ 1\leq j\leq D-1 \end{subarray}} \lambda_{e_j}R(G,\bp_3;e_j,V\setminus\{a\})={\bf 0},
\end{split}
\end{equation*}
and hence the part of the new row vector (\ref{eq:new_3}) associated with $V\setminus\{a\}$ actually becomes zero.
The rest part of (\ref{eq:new_3}) is the $D$ consecutive components associated with $a$.
Since the entries of $R(G,\bp_3;E\setminus\{ac,va\},a)$ of (\ref{eq:R'I'_3}) are all zero,
we have 
\begin{align*}
&\sum_{1\leq j\leq D-1} \lambda_{(ab)_j} R(G,\bp_3;(vb)_j,a)
+\sum_{1\leq j\leq D-1} \lambda_{(ac)_j} R(G,\bp_3;(va)_j,a)
+\sum_{\begin{subarray}{c} e \in  E\setminus\{va,vb,ac\} \\ 1\leq j\leq D-1 \end{subarray}} \lambda_{e_j}R(G,\bp_3;e_j,a)  \\
&=\sum_{1\leq j\leq D-1} \lambda_{(ac)_j} R(G,\bp_3;(va)_j,a) 
\qquad \qquad (\text{by }R(G,\bp_3;E\setminus \{ac,va\},a)={\bf 0} \text{ from (\ref{eq:R'I'_3})})  \\
%&=\sum_{1\leq j\leq D-1} \lambda_{(ac)_j} R_{G,\bp_3}[(va)_j,a]  \qquad \qquad (\text{see }(\ref{eq:R'I'_2})) \\
&=-\sum_{1\leq j\leq D-1} \lambda_{(ac)_j}r_j(\bp_3(va)) \qquad \qquad (\text{by } R(G,\bp_3;va,a)=-r(\bp_3(va)) \text{ from the definition})\\
&=-\sum_{1\leq j\leq D-1} \lambda_{(ac)_j}r_j(\bq(ac)). \qquad \qquad (\text{by } \bp_3(va)=\bq(ac) \text{ from } (\ref{eq:p'q}))  
\end{align*}
Therefore, the new row vector (\ref{eq:new_3}) is actually described as 
\begin{equation*} 
\begin{array}{cc} (\Bvector{a}{-\sum_{j} \lambda_{(ac)_j}r_j(\bq(ac))}, & \Bvector{V\setminus\{a\}}{\bf 0}\ \  ),\end{array}
\end{equation*}
and as a result $R(G,\bp_3)$ of (\ref{eq:R'I'_3}) is changed to the following matrix by the row operations:
\begin{equation}
\label{eq:J'R'I'}
R(G,\bp_3)=
\begin{array}{c@{\,}rc|c@{\,}l}

 &
 & \multicolumn{1}{c}{\scriptstyle a}
 & \multicolumn{1}{c}{\scriptstyle V\setminus \{a\}}
 &
 \\
\scriptstyle ac & \ldelim({3}{5pt}[] & -r(\bp_3(ac)) & {\bf 0} & \ldelim){3}{5pt} \\\cline{3-4}
\scriptstyle (vb)_{i^*} &            & -\sum_j\lambda_{(ac)_j}r_j(\bq(ac)) & {\bf 0} &  \\\cline{3-4}
                        &            & {\bf *}         & R(G_v^{ab}\setminus (ab)_{i^*},\bq)
\end{array},
\end{equation}
where $R(G_v^{ab}\setminus (ab)_{i^*},\bq)$ was defined as 
the matrix obtained from $R(G_v^{ab},\bq)$ by removing 
the row associated with $(ab)_{i^*}$ (see (\ref{eq:i^*})).

We finally show in the following Claim~\ref{claim:entries} 
that at least one of the top-left $D\times D$-submatrices of (\ref{eq:JR_1I_1}), (\ref{eq:JR_2I_2}) and (\ref{eq:J'R'I'}) 
has a full rank.
Let us simply denote these three submatrices by $M_1,M_2$ and $M_3$, respectively, i.e., 
\begin{equation*}
M_1=\begin{pmatrix} r(\bp_1(va)) \\ \sum_j\lambda_{(ab)_j}r_j(\bq(ab)) \end{pmatrix}, \quad
M_2=\begin{pmatrix} r(\bp_2(vb)) \\ \sum_j\lambda_{(ab)_j}r_j(\bq(ab)) \end{pmatrix}, \quad
M_3=\begin{pmatrix} r(\bp_3(ac)) \\ \sum_j\lambda_{(ac)_j}r_j(\bq(ac)) \end{pmatrix}.
\end{equation*}
From the definitions (\ref{eq:0_6_7_p1}), (\ref{eq:6_7_p2}) and (\ref{eq:6_7_p'}), 
$\bp_1(va)=L$, $\bp_2(vb)=L'$, and $\bp_3(ac)=L''$ hold,
where we may choose any $(d-2)$-affine subspaces $L\subset \Pi_{G_v^{ab},\bq}(a)$, $L'\subset \Pi_{G_v^{ab},\bq}(b)$, and 
$L''\subset \Pi_{G_v^{ab},\bq}(c)$, respectively.
Hence, $M_1, M_2$ and $M_3$ are described as
\begin{equation}
\label{eq:top_left}
M_1=\begin{pmatrix} r(L) \\ \sum_j\lambda_{(ab)_j}r_j(\bq(ab)) \end{pmatrix}, \quad
M_2=\begin{pmatrix} r(L') \\ \sum_j\lambda_{(ab)_j}r_j(\bq(ab)) \end{pmatrix}, \quad
M_3=\begin{pmatrix} r(L'') \\ \sum_j\lambda_{(ac)_j}r_j(\bq(ac)) \end{pmatrix}.
\end{equation}
The following Claim~\ref{claim:entries} implies that at least one of $R(G,\bp_1)$, $R(G,\bp_2)$ and $R(G,\bp_3)$ attains the desired rank $D(|V|-1)$.
(For example, if $M_3$ has a full rank, 
then we have $\rank R(G,\bp_3)\geq \rank M_3+\rank R(G_v^{ab}\setminus (ab)_{i^*},\bq)=D+D(|V|-2)=D(|V|-1)$ 
from (\ref{eq:J'R'I'}) and (\ref{eq:i^*}).), completing the proof of Lemma~\ref{lemma:case3-2}.

%Hence, what remains to be shown is Claim~\ref{claim:entries}.
\begin{claim}
\label{claim:entries}
At least one of $M_1, M_2$ and $M_3$ has a full rank for some choices of 
$L\subset \Pi_{G_v^{ab},\bq}(a)$, $L'\subset \Pi_{G_v^{ab},\bq}(b)$ and 
$L''\subset \Pi_{G_v^{ab},\bq}(c)$.
\end{claim}
\begin{proof}
Let $r\in\mathbb{R}^D$ be $\sum_{j}\lambda_{(ab)_j}r_j(\bq(ab))$.
Note that $r$ is a nonzero vector since $\{r_1(\bq(ab)),\dots, r_{D-1}(\bq(ab))\}$ is linearly independent and also $\lambda_{(ab)_{i^*}}=1$.
Suppose that $M_1$ given in (\ref{eq:top_left}) does not have a full rank.
Then, $r$ is contained in the row space of $r(L)$.
This is equivalent to that $r$ is contained in the orthogonal complement of the vector space spanned by a $(d-1)$-extensor 
($2$-extensor) $C(L)$ associated with $L$.
Similarly, if $M_2$ given in (\ref{eq:top_left}) does not have a full rank,
then $r$ is contained in the orthogonal complement of the vector space spanned by $C(L')$.

We claim the following:
If $M_3$ does not have a full rank, then $r$ is also contained in the orthogonal complement of the space spanned by $C(L'')$.
To see this, we need to remind you of the equation (\ref{eq:lambda}) 
of $D|V|$-dimensional vectors (indicating the row dependence within $R(G_v^{ab},\bq)$).
Focusing on the $D$ consecutive components associated with $a$ of (\ref{eq:lambda}),
we have
\begin{equation}
\label{eq:lambda'}
{\bf 0}=\sum_{\begin{subarray}{c}e\in E_v\cup\{ab\} \\ 1\leq j\leq D-1\end{subarray}}\lambda_{e_j}R(G_v^{ab},\bq;e_j,a).
\end{equation}
Recall that $R(G_v^{ab},\bq;e,a)={\bf 0}$ holds if $e$ is not incident to $a$ in $G_v^{ab}$ according to the definition of 
a rigidity matrix.
%all entries of $R(G_v^{ab},\bq;E_v\setminus\{ac\},a]$ are zero.
Since only $ab$ and $ac$ are incident to $a$ in $G_v^{ab}$, we actually have 
\begin{align*}
\sum_{\begin{subarray}{c}e\in E_v\cup\{ab\} \\ 1\leq j\leq D-1\end{subarray}}\lambda_{e_j}R(G_v^{ab},\bq;e_j,a)
&=\sum_{1\leq j\leq D-1}\lambda_{(ab)_j}R(G_v^{ab},\bq;(ab)_j,a)+\sum_{1\leq j\leq D-1}\lambda_{(ac)_j}R(G_v^{ab},\bq;(ac)_j,a) \\
&=\sum_{1\leq j\leq D-1}\lambda_{(ab)_j}r_j(\bq(ab))+\sum_{1\leq j\leq D-1}\lambda_{(ac)_j}r_j(\bq(ac)). \label{eq:ab_ac}
\end{align*}
Combining this and (\ref{eq:lambda'}), we consequently obtain 
\begin{equation}
\label{eq:r_ac}
r=-\sum_{1\leq j\leq D-1} \lambda_{(ac)_j}r_j(\bq(ac)).
\end{equation}
If $M_3$ does not have a full rank, then $\sum_j \lambda_{(ac)_j}r_j(\bq(ac))$ 
is contained in the orthogonal complement of the space spanned by $C(L'')$,
which means that $r$ is contained in the orthogonal complement of the space spanned by $C(L'')$ as well, by (\ref{eq:r_ac}).

As a result, we found that, if none of $M_1$, $M_2$ and $M_3$ has a full rank,
then there exists a nonzero vector $r\in\mathbb{R}^D$ contained in the orthogonal complement of 
the vector space spanned by the set of vectors ($2$-extensors)
\begin{equation}
\label{eq:space}
\left(\bigcup_{L\subset \Pi(a)} C(L) \right) \bigcup 
\left(\bigcup_{L'\subset \Pi(b)} C(L') \right) \bigcup
\left(\bigcup_{L''\subset \Pi(c)} C(L'') \right),
\end{equation}
where $\Pi(v)$ denotes $\Pi_{G_v^{ab},\bq}(v)$ for each $v\in \{a,b,c\}$.
Therefore, in order to show that at least one of  $M_1$, $M_2$ and $M_3$ has a full rank, 
it is enough to show that the dimension of the vector space spanned by (\ref{eq:space}) is equal to $D$.

To show this, let us take four points in $\mathbb{R}^3$ as follows; 
$p_1=\Pi(a)\cap \Pi(b)\cap \Pi(c), p_2\in \Pi(a)\cap \Pi(b)\setminus \Pi(c), p_3\in \Pi(b)\cap \Pi(c)\setminus \Pi(a)$,
and $p_4\in \Pi(c)\cap \Pi(a)\setminus \Pi(b)$.
Since $(G_v^{ab},\bq)$ is a generic nonparallel framework, 
the set of the coefficients appearing in the equations expressing $\Pi(a), \Pi(b)$, and $\Pi(c)$ is algebraically independent over the rational field,
and thus we can always take such four points in such a way that they are affinely independent.
Also, notice that any line connecting two points among $\{p_1,\dots,p_4\}$ is contained in $\Pi(a)\cup \Pi(b)\cup \Pi(c)$,
and hence any $2$-extensor obtained from two points among $\{p_1,\dots,p_4\}$ belongs to (\ref{eq:space}).
Recall that, by Lemma~\ref{lemma:extensor}, the set of $2$-extensors $\{{\bm p}_i\vee {\bm p}_j\mid 1\leq i<j\leq 4\}$ is linearly independent,
which means that the dimension of the vector space spanned by (\ref{eq:space}) is equal to ${4\choose 2}=6(=D)$.
\end{proof}

As a result, at least one of $(G,\bp_1)$, $(G,\bp_2)$ and $(G,\bp_3)$
attains the rank equal to $D(|V|-1)$. As we already remarked, 
although $(G,\bp_1)$, $(G,\bp_2)$ and $(G,\bp_3)$ are not nonparallel, 
we can convert them to nonparallel panel-and-hinge realizations 
by slightly rotating the panel associated with $v$ (or that associated with $a$) 
without decreasing the rank of the rigidity matrix by Lemma~\ref{lemma:perturbation}.
This completes the proof of Lemma~\ref{lemma:case3-2}.
\end{proof}

\subsubsection{General dimension}
\label{subsec:case3-2_d}
Finally let us describe the general dimensional version of Lemma~\ref{lemma:case3-2}.
\begin{lemma}
\label{lemma:case3-2_d}
Let $G=(V,E)$ be a 2-edge-connected minimal $0$-dof-graph with $|V|\geq 3$.
Suppose that there exists no proper rigid subgraph in $G$ and that (\ref{eq:hypothesis}) holds.
Then, there is a nonparallel panel-and-hinge realization $(G,\bp)$ in $\mathbb{R}^d$ satisfying $\rank R(G,\bp)=D(|V|-1)$.
\end{lemma}
\begin{proof}
By Lemma~\ref{lemma:degree2}, either $G$ is a cycle of length at most $d$ or $G$ has a chain of length $d$. 
If $G$ is a cycle of length at most $d$, then we are done by Lemma~\ref{lemma:base2}.
Hence, let us consider the case where $G$ has a chain $v_0v_1v_2\dots v_{d}$ of length $d$ (where $d_G(v_i)=2$ for $1\leq i\leq d-1$).

The proof strategy is exactly the same as $d=3$.
In general case, we shall consider $d$ distinct frameworks
based on the chain $v_0v_1v_2\dots v_{d}$.
For each $1\leq i\leq d-1$,  the graph $G_{v_i}^{v_{i-1}v_{i+1}}$ obtained by the splitting off at $v_i$
is a minimal $0$-dof-graph by Lemma~\ref{lemma:operation}.
Let us simply denote $G_{v_i}^{v_{i-1}v_{i+1}}$ by $G_i$ for each $1\leq i\leq d-1$.
By (\ref{eq:hypothesis}), 
there exist nonparallel panel-and-hinge realizations $(G_i,\bq_i)$ for $1\leq i\leq d-1$ such that they 
represent the same framework in $\mathbb{R}^d$.
Based on $(G_i,\bq_i)$, we shall construct $d$ distinct frameworks for $G$ in the same way as 3-dimensional case.
We shall show that at least one of them attains the desired rank.

Let us start the proof.
Consider $G_1=(V\setminus \{v_1\}, E\setminus\{v_0v_1,v_1v_2\}\cup\{v_0v_2\})$.
By (\ref{eq:hypothesis}), there exists a generic nonparallel panel-and-hinge framework $(G_1,\bq_1)$
which satisfies
\begin{equation}
\label{eq:d_rank_G_1'}
R(G_1,\bq_1)=D(|V|-2).
\end{equation}
We first define two frameworks $(G,\bp_0)$ and $(G,\bp_1)$ based on $(G_1,\bq_1)$;
for $e\in E$,
\begin{equation}
\label{eq:d_p_0}
\bp_0(e)=\begin{cases}
\bq_1(e) & \text{if } e\in E\setminus\{v_0v_1,v_1v_2\} \\
L_0 & \text{if } e=v_0v_1 \\
\bq_1(v_0v_2) & \text{if } e=v_1v_2
\end{cases}
\end{equation}
\begin{equation}
\label{eq:d_p_1}
\bp_1(e)=\begin{cases}
\bq_1(e) & \text{if } e\in E\setminus\{v_0v_1,v_1v_2\} \\
\bq_1(v_0v_2) & \text{if } e=v_0v_1 \\
L_1 & \text{if } e=v_1v_2
\end{cases}
\end{equation}
where $L_0$ is a $(d-2)$-affine subspace contained in $\Pi_{G_1,\bq_1}(v_0)$
and $L_1$ is the one contained in $\Pi_{G_1,\bq_1}(v_2)$.
Then, as was shown in Claim~\ref{claim:coplanarity},
it is not difficult to see that $(G,\bp_0)$ is a panel-and-hinge framework such that
$\Pi_{G,\bp_0}(u)=\Pi_{G_1,\bq_1}(u)$ for $u\in V\setminus\{v_1\}$ and 
$\Pi_{G,\bp_0}(v_1)=\Pi_{G_1,\bq_1}(v_0)$.
Similarly, $(G,\bp_1)$ is a panel-and-hinge framework such that
$\Pi_{G,\bp_1}(u)=\Pi_{G_1,\bq_1}(u)$ for $u\in V\setminus\{v_1\}$ and 
$\Pi_{G,\bp_1}(v_1)=\Pi_{G_1,\bq_1}(v_2)$.

Let us take a look at $R(G,\bp_0)$ given by 
\begin{equation}
\label{eq:d_R_0}
R(G,\bp_0)=
\begin{array}{c@{\,}rc|c|c|c@{\,}l}

 &
 & \multicolumn{1}{c}{\scriptstyle v_1}
 & \multicolumn{1}{c}{\scriptstyle v_0}
 & \multicolumn{1}{c}{\scriptstyle v_2}
 & \multicolumn{1}{c}{\scriptstyle V\setminus \{v_0,v_1,v_2\}}
 &
 \\
 \scriptstyle v_0v_1    & \ldelim({3}{5pt}[] & r(\bp_0(v_0v_1)) & -r(\bp_0(v_0v_1))  & {\bf 0} & {\bf 0}  & \ldelim){3}{5pt}[] \\ \cline{3-6} 
 \scriptstyle v_1v_2    &                    & r(\bp_0(v_1v_2)) & {\bf 0}       & -r(\bp_0(v_1v_2)) & {\bf 0} &                 \\ \cline{3-6}
                    &                    & {\bf 0}      & \multicolumn{3}{c}{R(G,\bp_0;E\setminus\{v_0v_1,v_1v_2\},V\setminus \{v_1\})}  &          \\
\end{array}.
\end{equation}

We shall apply the matrix manipulation as was given in the proof of Lemma~\ref{lemma:case3-2},
which converts the matrix of (\ref{eq:0_R_1}) to that of (\ref{eq:JR_1I_1}).
Here, the vertices $v_1, v_0$ and $v_2$ play the roles of $v,a$ and $b$ (of $R(G,\bp_1)$ in the previous lemma), respectively.
The matrix $R(G,\bp_0)$ of (\ref{eq:d_R_0}) is changed to the following form by appropriate column and row operations 
by using the fact that the rows associated with $v_0v_2$ in $R(G_1,\bq_1)$ correspond to those associated with $v_1v_2$ in $R(G,\bp_0)$:
\begin{equation}
\label{eq:d_JR0I}
R(G,\bp_0)=
\begin{array}{c@{\,}rc|c@{\,}l}

 &
 & \multicolumn{1}{c}{\scriptstyle v_1}
 & \multicolumn{1}{c}{\scriptstyle V\setminus \{v_1\}}
 &
 \\
 \scriptstyle v_0v_1    & \ldelim({3}{5pt}[] & r(L_0) & {\bf 0}  & \ldelim){3}{5pt}[] \\ \cline{3-4} 
 \scriptstyle (v_1v_2)_{i^*}    &            & \sum_{j}\lambda_{(v_0v_2)_j}r_j(\bq_1(v_0v_2)) & {\bf 0}  &      \\ \cline{3-4}
                                &            & {\bf *}                    & R(G_1\setminus(v_0v_2)_{i^*},\bq_1) & \\ 
\end{array},
\end{equation}
where several new notations appearing in (\ref{eq:d_JR0I}) are defined as follows:
the integer $i^*$ is the index of a redundant row vector among those associated with $v_0v_2$ in $R(G_1,\bq_1)$
(such a redundant edge always exists by Claim~\ref{claim:3-2-1}),
$R(G_1\setminus\{(v_0v_2)_{i^*}\},\bq_1)$ denotes the matrix obtained by removing the row of $(v_0v_2)_{i^*}$ from $R(G_1,\bq_1)$ and satisfies
\begin{equation}
\label{eq:d_i^*}
\rank R(G_1\setminus\{(v_0v_2)_{i^*}\},\bq_1)=\rank R(G_1,\bq_1)=D(|V|-2).
\end{equation}
Also, the scalar $\lambda_{(v_0v_2)_j}$ comes from the redundancy of $R(G_1,\bq_1;(v_0v_2)_{i^*})$ within $R(G_1,\bq_1)$, i.e., 
since $R(G_1,\bq_1;(v_0v_2)_{i^*})$ is redundant in $R(G_1,\bq_1)$, it can be expressed
by a linear combination of the other row vectors of $R(G_1,\bq_1)$
and hence we have introduced $\lambda_{e_j}$ for each $e\in E\setminus\{v_0v_1,v_1v_2\}\cup \{v_0v_2\}$ and $1\leq j\leq D-1$ 
such that $\lambda_{(v_0v_2)_{i^*}}=1$ and
\begin{equation}
\label{eq:d_lambda}
\sum_{\begin{subarray}{c} e\in E\setminus\{v_0v_1,v_1v_2\}\cup\{v_0v_2\} \\ 1\leq j\leq D-1 \end{subarray}}\lambda_{e_j}R(G_1,\bq_1;e_j)={\bf 0}.
\end{equation}
This dependency will play a key role in the proof.

Symmetrically, we can convert $R(G,\bp_1)$ to the following matrix by appropriate row and column fundamental operations:
\begin{equation}
\label{eq:d_JR1I}
R(G,\bp_1)=
\begin{array}{c@{\,}rc|c@{\,}l}

 &
 & \multicolumn{1}{c}{\scriptstyle v_1}
 & \multicolumn{1}{c}{\scriptstyle V\setminus \{v_1\}}
 &
 \\
 \scriptstyle v_1v_2    & \ldelim({3}{5pt}[] & r(L_1) & {\bf 0}  & \ldelim){3}{5pt}[] \\ \cline{3-4} 
 \scriptstyle (v_0v_1)_{i^*}    &            & \sum_{j}\lambda_{(v_0v_2)_j}r_j(\bq_1(v_0v_2)) & {\bf 0}  &      \\ \cline{3-4}
                                &            & {\bf *}                    & R(G_1\setminus(v_0v_2)_{i^*},\bq_1) & \\ 
\end{array},
\end{equation}
Notice that the row vectors associated with $v_0v_2$ in $R(G_1,\bq_1)$ 
 correspond to those with $v_0v_1$ in $R(G,\bp_1)$.

We are now going to construct the other $d-2$ frameworks.
Consider $G_i=G_{v_i}^{v_{i-1}v_{i+1}}=(V\setminus\{v_i\}, E\setminus\{v_{i-1}v_i,v_iv_{i+1}\}\cup\{v_{i-1}v_{i+1}\})$ for $2\leq i\leq d-1$.
We shall focus on the following isomorphism $\rho_i:V\setminus\{v_i\}\rightarrow V\setminus\{v_1\}$ 
between $G_1$ and $G_i$ for each $2\leq i\leq d-1$:
\begin{equation}
\label{eq:d_isomorphism}
\rho_i(u)=\begin{cases}
u & \text{if } u\in V\setminus\{v_1,v_2,\dots,v_{i-1},v_i\} \\
v_{j+1} & \text{if } u=v_j\in\{v_1,v_2,\dots,v_{i-1}\}.
\end{cases}
\end{equation}

Then, based on the isomorphism $\rho_i$, we consider the nonparallel panel-and-hinge framework $(G_i,\bq_i)$ for $2\leq i\leq d-1$,
which is exactly the same framework as $(G_1,\bq_1)$
such that 
\begin{equation}
\label{eq:d_q_panel}
\Pi_{G_i,\bq_i}(u)=\Pi_{G_1,\bq_1}(\rho_i(u)) \qquad \text{for each } u\in V\setminus\{v_i\}.
\end{equation} 
More formally, $(G_i,\bq_i)$ is defined by the mapping $\bq_i$ on the edge set of $G_i$, which is $E\setminus \{v_{i-1}v_i,v_iv_{i+1}\}\cup\{v_{i-1}v_{i+1}\}$,
defined as follows:
\begin{equation}
\label{eq:d_q_i}
\bq_i(uw)=\bq_1(\rho_i(u)\rho_i(w))=
\begin{cases}
\bq_1(uw) & \text{if } uw\in E\setminus\{v_0v_1,v_1v_2, \dots, v_{d-1}v_d\} \\
\bq_1(v_0v_2) & \text{if } uw=v_0v_1 \\
\bq_1(v_jv_{j+1}) & \text{if } uw=v_{j-1}v_j \text{ for  } 2\leq j\leq i-1 \\
\bq_1(v_iv_{i+1}) & \text{if } uw=v_{i-1}v_{i+1} \\
\bq_1(v_{j}v_{j+1}) & \text{if } uw=v_{j}v_{j+1} \text{ for } i+1\leq j\leq d-1.
\end{cases}
\end{equation}

Based on $(G_i,\bq_i)$, we shall construct the framework $(G,\bp_i)$ for each $2\leq i\leq d-1$ as follows (see Figure~\ref{fig:case3-2_d}):
\begin{equation}
\label{eq:d_p_i}
\bp_i(e)=\begin{cases}
\bq_i(e) & \text{if } e\in E\setminus\{v_{i-1}v_i,v_iv_{i+1}\} \\
\bq_i(v_{i-1}v_{i+1}) & \text{if } e=v_{i-1}v_i \\
L_i & \text{if } e=v_iv_{i+1} 
\end{cases}
\end{equation}
where $L_i$ is a $(d-2)$-affine subspace contained in $\Pi_{G_i,\bq_i}(v_{i+1})$.
Note  that $\Pi_{G_i,\bq_i}(v_{i+1})=\Pi_{G_1,\bq_1}(\rho_i(v_{i+1}))=\Pi_{G_1,\bq_1}(v_{i+1})$ 
by  (\ref{eq:d_isomorphism}) and (\ref{eq:d_q_panel}).
Hence, $L_i$ is a $(d-2)$-affine subspace satisfying
\begin{equation}
\label{eq:d_Li}
L_i\subset \Pi_{G_1,\bq_1}(v_{i+1}).
\end{equation}
As was shown in Claim~\ref{claim:coplanarity}, it is not difficult to see that $(G,\bp_i)$ is a panel-and-hinge framework
satisfying
$\Pi_{G,\bp_i}(u)=\Pi_{G_i,\bq_i}(u)=\Pi_{G_1,\bq_1}(\rho_i(u))$ for each $u\in V\setminus\{v_i\}$
and $\Pi_{G,\bp_i}(v_i)=\Pi_{G_i,\bq_i}(v_{i+1})=\Pi_{G_1,\bq_1}(v_{i+1})$.
Hence, $(G,\bp_i)$ is a panel-and-hinge framework such that
only the panels of $v_i$ and $v_{i+1}$ coincide and all the other pairs of panels are nonparallel.
We remark that $(G,\bp_i)$ can be converted to a nonparallel panel-and-hinge framework without decreasing the rank of the rigidity matrix 
by Lemma~\ref{lemma:perturbation}.

\begin{figure}[t]
\centering
\includegraphics[width=0.99\textwidth]{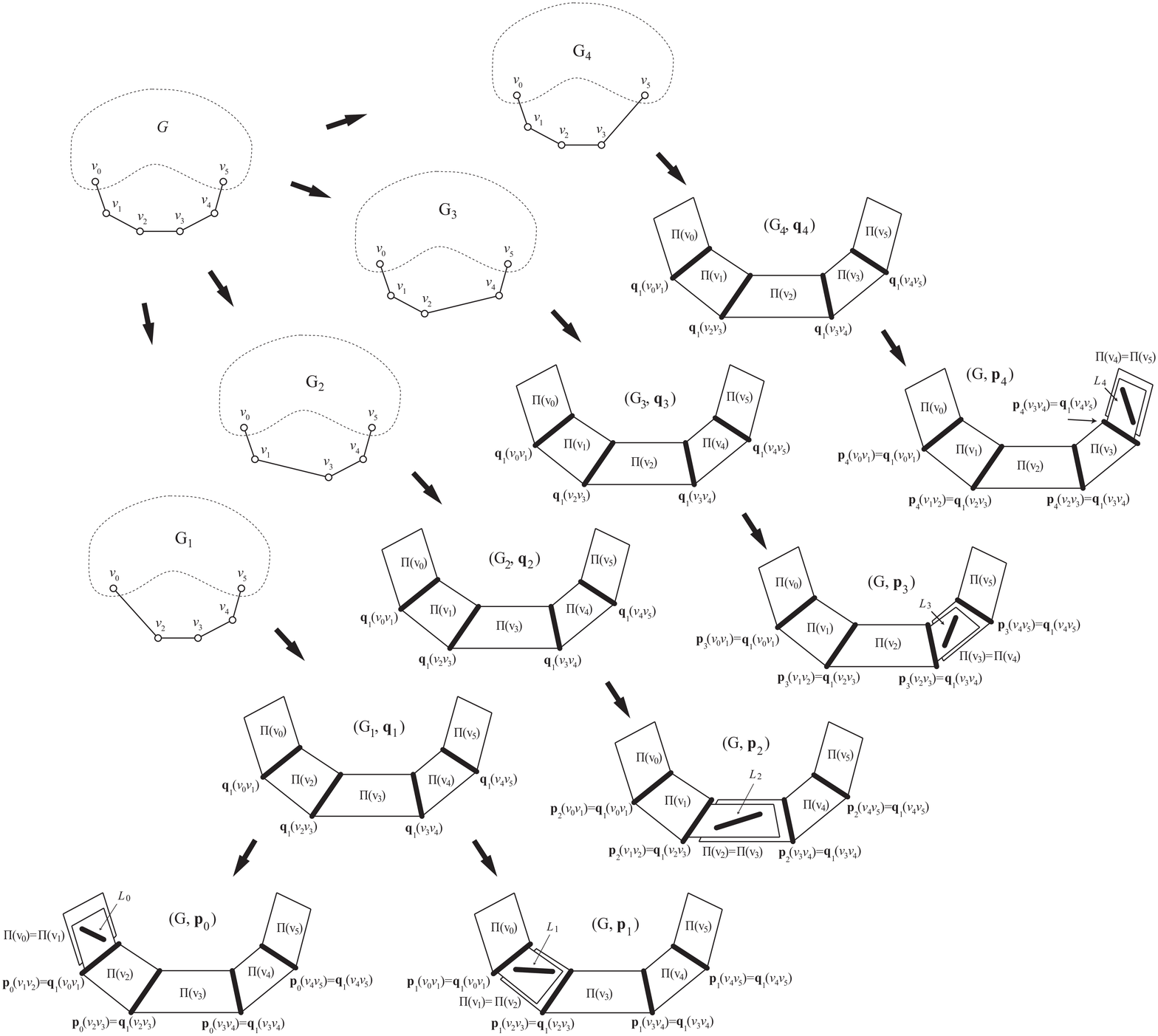}
\caption{The frameworks considered in the proof of Lemma~\ref{lemma:case3-2_d} for $d=5$,
where planes and bold segments represent $(d-1)$- and $(d-2)$-affine spaces, respectively}
\label{fig:case3-2_d}
\end{figure}

Combining (\ref{eq:d_q_i}) and (\ref{eq:d_p_i}), we have, for $2\leq i\leq d-1$,
\begin{align}
\bp_i(e)&=\bq_1(e) && \text{for } e\in E\setminus\{v_0v_1,v_1v_2,\dots, v_{d-1}v_d\} \nonumber \\
\bp_i(v_0v_1)&=\bq_1(v_0v_2) \nonumber \\
\bp_i(v_{j-1}v_j)&=\bq_1(v_jv_{j+1}) && \text{for } 2\leq j\leq i \label{eq:d_p_q}  \\
\bp_i(v_iv_{i+1})&=L_i \nonumber \\
\bp_i(v_{j'}v_{j'+1})&=\bq_1(v_{j'}v_{j'+1}) && \text{for } i+1\leq j'\leq d-1 \text{ (if } i\neq d-1). \nonumber
\end{align}

Let us consider $R(G,\bp_i)$:
\begin{equation}
\label{eq:d_Ri}
R(G,\bp_i)=
\begin{array}{c@{\,}rc|c|c|c@{\,}l}

 &
 & \multicolumn{1}{c}{\scriptstyle v_i}
 & \multicolumn{1}{c}{\scriptstyle v_{i+1}}
 & \multicolumn{1}{c}{\scriptstyle v_{i-1}}
 & \multicolumn{1}{c}{\scriptstyle V\setminus \{v_{i-1},v_i,v_{i+1}\}}
 &
 \\
 \scriptstyle v_iv_{i+1}    & \ldelim({3}{5pt}[] & r(\bp_i(v_iv_{i+1})) & -r(\bp_i(v_iv_{i+1}))  & {\bf 0} & {\bf 0}  & \ldelim){3}{5pt}[] \\ \cline{3-6} 
 \scriptstyle v_{i-1}v_i    &                    & r(\bp_i(v_{i-1}v_i)) & {\bf 0}       & -r(\bp_i(v_{i-1}v_i)) & {\bf 0} &                 \\ \cline{3-6}
                    &                    & {\bf 0}      & \multicolumn{3}{c}{R(G,\bp_i;E\setminus\{v_{i-1}v_i,v_iv_{i+1}\},V\setminus \{v_i\})}  &          \\
\end{array}.
\end{equation}
We now convert $R(G,\bp_i)$ to the matrix which contains $R(G_1,\bq_1)$ as its submatrix;
perform the column operations which add the $j$-th column of $R(G,\bp_i;v_i)$ to that of $R(G,\bp_i;v_{i+1})$  for each $1\leq j\leq D$
and then substitute all of (\ref{eq:d_p_q}) into the resulting matrix.
Then, it is not difficult to see that 
$R(G,\bp_i)$ of (\ref{eq:d_Ri}) is changed to
\begin{equation}
\label{eq:d_RiI}
R(G,\bp_i)=
\begin{array}{c@{\,}rc|c@{\,}l}

&
& \multicolumn{1}{c}{\scriptstyle v_i}
& \multicolumn{1}{c}{\scriptstyle V\setminus\{v_i\}}
&
\\
\scriptstyle v_iv_{i+1} & \ldelim({3}{5pt}[] & r(L_i) & {\bf 0} & \ldelim){3}{5pt}[] \\ \cline{3-4}
\begin{matrix} \scriptstyle v_{i-1}v_i \\  \ \end{matrix} &  &
\begin{matrix} r(\bq_1(v_iv_{i+1})) \\ \hline {\bf 0} \end{matrix} & R(G_1,\bq_1) & 
\end{array}
\end{equation}
where we used the following row correspondence 
between $R(G,\bp_i;E\setminus \{v_iv_{i+1}\}, V\setminus\{v_i\})$ and $R(G_1,\bq_1)$ derived from (\ref{eq:d_p_q}):
\begin{equation}
\label{eq:d_correspondence}
\begin{array}{cccc}
R(G,\bp_i) &  & R(G_1,\bq_1) & \\
%\left.
%\begin{aligned}
%&u \\
%&v_j
%\end{aligned}
%\right\} 
%&
%\Longleftrightarrow
%&
%\left\{
%\begin{aligned}
%&u \\
%&v_{j+1}
%\end{aligned}
%\right. 
%&
%\begin{aligned}
%&\text{for } u\in V\setminus\{v_1,v_2,\dots,v_i\} \\
%&\text{for } 1\leq j\leq i-1
%\end{aligned}
%\\
\left.
\begin{aligned}
&e \\
&v_0v_1 \\
&v_{j-1}v_j \\
&v_{j'}v_{j'+1}
\end{aligned}
\right\} & 
\Longleftrightarrow \qquad & 
\left\{
\begin{aligned}
&e   \\
&v_0v_2 \\
&v_jv_{j+1} \\
&v_{j'}v_{j'+1}
\end{aligned}
\right. & 
\qquad
\begin{aligned}
&\text{for } e\in E\setminus\{v_0v_1,\dots,v_{d-1}v_d\} \\
& \\
&\text{for } 2\leq j\leq i \\
&\text{for } i+1\leq j'\leq d-1
\end{aligned}
\end{array}
\end{equation}
(and the column correspondence 
follows from the isomorphism $\rho_i$ defined in (\ref{eq:d_isomorphism})).
In particular, the row associated with $(v_0v_2)_{i^*}$ in $R(G_1,\bq_1)$ corresponds to 
the row associated with $(v_0v_1)_{i^*}$ in $R(G,\bp_i)$.

Recall that the row of $R(G_1,\bq_1)$ associated with $(v_0v_2)_{i^*}$ is redundant, i.e.,
removing the row associated with $(v_0v_2)_{i^*}$ from $R(G_1,\bq_1)$ preserves the rank   as shown in (\ref{eq:d_i^*}).
As before, we consider the row operations which convert $R(G_1,\bq_1;(v_0v_2)_{i^*})$ to a zero vector 
and apply these row operations to the matrix $R(G,\bp_i)$ of (\ref{eq:d_RiI}).
More precisely, following the row correspondence described in (\ref{eq:d_correspondence}),
we shall perform the fundamental row operations which add the row vectors of $R(G,\bp_i;E\setminus \{v_iv_{i+1}\})$ to
the row $R(G,\bp_i;(v_0v_1)_{i^*})$ with the weight $\lambda_{e_j}$.
Namely, the row operations change the row associated with $(v_0v_1)_{i^*}$ in (\ref{eq:d_RiI}) to the following row vector: 
\begin{equation}
\label{eq:d_dependence3}
\small \begin{split}
\sum_{1\leq j\leq D-1}&\lambda_{(v_0v_2)_j}R(G,\bp_i;(v_0v_1)_j)
+\sum_{2\leq j'\leq i, 1\leq j\leq D-1}\lambda_{(v_{j'}v_{j'+1})_j}R(G,\bp_i;(v_{j'-1}v_{j'})_j) \\
&+\sum_{i+1\leq j'\leq d-1, 1\leq j\leq D-1} \lambda_{(v_{j'}v_{j'+1})_j}R(G,\bp_i;(v_{j'}v_{j'+1})_j) 
+\sum_{\begin{subarray}{c} e\in E\setminus\{v_0v_1,\dots,v_{d-1}v_d\} \\ 1\leq j\leq D-1 \end{subarray}} \lambda_{e_j}R(G,\bp_i;e_j).
\end{split}
\end{equation}
where $\lambda_{(v_0v_2)_{i^*}}=1$.

By (\ref{eq:d_lambda}),
all the entries of the part of the new row vector (\ref{eq:d_dependence3}) associated with $V\setminus\{v_i\}$ become zero.
%In fact, when considering the components associated with $V\setminus\{v_i\}$ in (\ref{eq:d_dependence3}),
%we have
%\begin{align*}
%R(G)_i[(v_0v_1)_j,V\setminus\{v_i\}]&=R(G_1,\bq_1;(v_{0}v_{2})_j) \\
%R_i[(v_{j'-1}v_{j'})_j,V\setminus\{v_i\}]&=R_{G_1,\bq_1}[(v_{j'}v_{j'+1})_j] && \text{for } 2\leq j'\leq i\\
%R_i[(v_{j'}v_{j'+1})_j,V\setminus\{v_i\}]&=R_{G_1,\bq_1}[(v_{j'}v_{j'+1})_j] && \text{for } i+1\leq j'\leq d-1 \\
%R_i[e_j,V\setminus\{v_i\}]&=R_{G_1,\bq_1}[e_j] && \text{for } e\in E\setminus\{v_0v_1,\dots,v_{d-1}v_d\}
%\end{align*}
%from the row correspondence given in (\ref{eq:d_correspondence}).
%Thus, the components associated with $V\setminus\{v_i\}$ in the new row vector (\ref{eq:d_dependence3}) 
%can be  described by 
%\begin{equation*}
%\sum_{\begin{subarray}{c} e\in E\setminus\{v_0v_1,v_1v_2\}\cup\{v_0v_2\} \\ 1\leq j\leq D-1\end{subarray}} \lambda_{e_j}R_{G_1,\bq_1}[e_j]
%\end{equation*}
%which is equal to zero by (\ref{eq:d_lambda}).
Also, as for the rest part of the new row (\ref{eq:d_dependence3}) associated with $v_i$,
since the entries of $R(G,\bp_i;E\setminus\{v_{i-1}v_i,v_iv_{i+1}\},v_i)$ of (\ref{eq:d_RiI}) are all zero,
the row operations will change $R(G,\bp_i;(v_0v_1)_{i^*},v_i)$ to
\begin{equation*}
\sum_{1\leq j\leq D-1} \lambda_{(v_iv_{i+1})_j}r_j(\bp_i(v_{i-1}v_i))
\end{equation*}
which is equal to
\begin{equation*}
\sum_{1\leq j\leq D-1} \lambda_{(v_iv_{i+1})_j}r_j(\bq_1(v_iv_{i+1}))
\end{equation*}
since $\bp_i(v_{i-1}v_i)=\bq_1(v_iv_{i+1})$ by (\ref{eq:d_p_q}).
Therefore, the fundamental row operations actually change $R(G,\bp_i)$ of (\ref{eq:d_RiI}) to the following matrix 
(up to some row exchange of $(v_0v_1)_{i^*}$):
\begin{equation}
\label{eq:d_JRiI}
R(G,\bp_i)= 
\begin{array}{c@{\,}rc|c@{\,}l}

 &
 & \multicolumn{1}{c}{\scriptstyle v_i}
 & \multicolumn{1}{c}{\scriptstyle V\setminus \{v_i\}}
 &
 \\
 \scriptstyle v_iv_{i+1}    & \ldelim({3}{5pt}[] & r(L_i) & {\bf 0}  & \ldelim){3}{5pt}[] \\ \cline{3-4} 
 \scriptstyle (v_0v_1)_{i^*}    &            & \sum_{j}\lambda_{(v_iv_{i+1})_j}r_j(\bq_1(v_iv_{i+1})) & {\bf 0}  &      \\ \cline{3-4}
                                &            & {\bf *}                    & R(G_1\setminus(v_0v_2)_{i^*},\bq_1) & \\ 
\end{array}.
\end{equation}
%where $R(G_1\setminus (v_0v_1)_{i^*},\bq_1)$ was defined as 
%the matrix obtained from $R(G_1,\bq_1)$ by removing 
%the redundant row vector associated with $(v_0v_1)_{i^*}$.

Let us denote the top-left $D\times D$-submatrices of  (\ref{eq:d_JR0I}) and  (\ref{eq:d_JR1I}) by $M_0$ and $M_1$, 
and also that of (\ref{eq:d_JRiI}) by $M_i$ for $2\leq i\leq d-1$, i.e., 
\begin{align}
\nonumber %\label{eq:d_top_left}
M_0&=\begin{pmatrix} r(L_0) \\ \sum_j\lambda_{(v_0v_{2})_j}r_j(\bq(v_0v_{2})) \end{pmatrix}, \qquad
M_1=\begin{pmatrix} r(L_1) \\ \sum_j\lambda_{(v_0v_{2})_j}r_j(\bq(v_0v_{2})) \end{pmatrix} \\ 
M_i&=\begin{pmatrix} r(L_i) \\ \sum_j\lambda_{(v_iv_{i+1})_j}r_j(\bq(v_iv_{i+1})) \end{pmatrix} \quad \text{ for } 2\leq i\leq d-1. \nonumber
\end{align}
Recall that $L_0$ can be taken as any $(d-2)$-affine subspace satisfying $L_0\subset \Pi_{G_1,\bq_1}(v_0)$
while for $1\leq i\leq d-1$ $L_i$ can be any $(d-2)$-affine subspace satisfying $L_i\subset \Pi_{G_1,\bq_1}(v_{i+1})$.
Then, as in the proof of 3-dimensional case, the remaining task is
to show the following fact:
\begin{equation}
\label{eq:claim_Mi}
\text{At least one of $M_0, M_1, \dots, M_{d-1}$
has a full rank for an appropriate choice of $L_i, 0\leq i\leq d-1$.}
\end{equation}
If this is true, then we obtain $\rank R(G,\bp_i)\geq \rank M_i+\rank R(G_1\setminus\{(v_0v_2)_{i^*}\},\bq_1)=D+D(|V|-2)=D(|V|-1)$ by (\ref{eq:d_i^*}).

Let us show (\ref{eq:claim_Mi}).
Let $r=\sum_{j}\lambda_{(v_0v_2)_j}r_j(\bq(v_0v_2))$.
Suppose $M_0$ (and $M_1$, resp.) does not have a full rank.
Then, $r$ is contained in the row space of $r(L_0)$ (and $r(L_1)$, resp.), that is, $r$ is contained in the orthogonal complement of 
the vector space spanned by a $(d-1)$-extensor $C(L_0)$ (and that of $C(L_1)$, respectively).
Also, due to the fact that $v_i$ is a vertex of degree two in $G_1$ for all $2\leq i\leq d-1$,
we can easily show the following fact in a manner similar to the previous lemma (c.f. (\ref{eq:r_ac})):
\begin{equation}
\label{eq:d_linearity}
\sum_{1\leq j\leq D-1}\lambda_{(v_iv_{i+1})_j}r_j(\bq(v_iv_{i+1}))=\pm r \quad \text{ for } 2\leq i\leq d-1
\end{equation}
Notice that (\ref{eq:d_linearity}) implies that 
$M_i$ does not have a full rank if and only if $r$ is contained in the orthogonal complement of the vector space spanned by $C(L_i)$.
Therefore, none of $M_0,M_1,\dots,M_{d-1}$ has a full rank for any choice of $L_i, 0\leq i\leq d-1$ 
if and only if $r$ is contained in the orthogonal complement of the vector space spanned by
\begin{equation}
\label{eq:d_space}
\bigcup_{0\leq i\leq d-1}\left(\bigcup_{L_i\subset \Pi_{i}} C(L_i)\right),
\end{equation}
where $\Pi_0=\Pi_{G_1,\bq_1}(v_0)$ and $\Pi_i=\Pi_{G_1,\bq_1}(v_{i+1})$ for each $i$ with $1\leq i\leq d-1$.
Therefore, in order to show that at least one of  $M_0$, $M_1, \dots, M_{d-1}$ has a full rank, 
it is enough to show that the dimension of the vector space spanned by (\ref{eq:d_space}) is equal to $D$.

Since the framework $(G_1,\bq_1)$ is a generic nonparallel framework, 
the set of the coefficients appearing in the equations expressing $\Pi_i, 0\leq i\leq d-1$ is algebraically independent over the rational field.
Therefore, for any $j$ hyperplanes among them, their intersection forms a $(d-j)$-dimensional affine space.
Hence we can take $d+1$ distinct points in $\mathbb{R}^d$ as follows;
for each $0\leq i\leq d-1$, $p_{i}\in \bigcap_{0\leq j\leq d-1,j\neq i}\Pi_j\setminus \Pi_i$  and $p_{d}=\bigcap_{0\leq j\leq d-1}\Pi_j$.
Clearly, $\{p_0,p_1,\dots,p_{d}\}$ is affinely independent,
and also notice that any $(d-2)$-affine subspace spanned by $(d-1)$ points among them is contained in $\bigcup_{0\leq j\leq d-1} \Pi_j$.
This implies that any $(d-1)$-extensor obtained from $(d-1)$ points belongs to the set of (\ref{eq:d_space}).
Thus the dimension of the vector space spanned by (\ref{eq:d_space}) is equal 
to ${d+1 \choose d-1}=D$ by Lemma~\ref{lemma:extensor}.

This completes the proof of Lemma~\ref{lemma:case3-2_d} as well as Theorem~\ref{theorem:main} in general dimension.
\end{proof}

\section*{Acknowledgments}
We would like to thank Audrey Lee-St.~John,  Meera Sitharam, Ileana Streinu and Louis Theran
for giving several useful comments on our draft paper 
at Barbados Workshop 2009: Geometric constraints with applications in CAD and biology. 

The first author is supported by Grant-in-Aid for Scientific Research (B) and Grant-in-Aid for Scientific Research (C), JSPS.
The second author is supported by Grant-in-Aid for JSPS Research Fellowships for Young Scientists.

\bibliographystyle{abbrv}

\end{document}